\documentclass[11pt, notitlepage]{article}
\usepackage{amssymb,amsmath,comment}
\usepackage{amsthm}
\allowdisplaybreaks
\catcode`\@=11 \@addtoreset{equation}{section}
\def\thesection{\arabic{section}}

\def\theequation{\thesection.\arabic{equation}}
\catcode`\@=12
\usepackage{colortbl}
\usepackage{hyperref}
\usepackage[mathscr]{eucal}
\usepackage{epsf}
\usepackage{esint}
\usepackage{a4wide}
\def\R{\mathbb{R}}

\DeclareMathOperator*{\esssup}{ess\,sup}
\DeclareMathOperator*{\essinf}{ess\,inf}

\newcommand{\la} {\lambda}

\newcommand{\noi} {\noindent}
\newcommand{\na} {\nabla}

\newcommand{\mb} {\mathbb}

\newcommand{\Tail} {\mathrm{Tail}}

\newcommand{\supp}{\mathrm{supp}}
\newcommand{\osc}{\mathrm{osc}}

\usepackage[all]{xy}
\catcode`\@=11
\def\theequation{\@arabic{\c@section}.\@arabic{\c@equation}}
\catcode`\@=12

\newtheorem{Theorem}{Theorem}[section]
\newtheorem{Lemma}[Theorem]{Lemma}
\newtheorem{prop}[Theorem]{Proposition}
\newtheorem{Corollary}[Theorem]{Corollary}
\newtheorem{Remark}[Theorem]{Remark}
\newtheorem{Definition}[Theorem]{Definition}

\begin{document}

{\vspace{0.01in}}

\title{On the regularity theory for mixed local and nonlocal weighted quasilinear elliptic equations}

\author{Sanjit Biswas\footnote{Department of Mathematical Sciences, Indian Institute of Science Education and Research, Berhampur, 760010, Odisha, India, Email: sanjitbiswas410@gmail.com } \,\,and Prashanta Garain\footnote{ Department of Mathematical Sciences, Indian Institute of Science Education and Research, Berhampur, 760010, Odisha, India, Email: pgarain92@gmail.com}}

\date{}
\maketitle

\begin{abstract}\noindent
We investigate a broad class of mixed local and nonlocal degenerate $p$-Laplace equations with general right-hand sides. The degeneracy is governed by Muckenhoupt $A_p$-weights, yielding a highly nonuniform elliptic framework in which both the local and nonlocal operators may degenerate simultaneously. We establish a comprehensive local regularity theory, including local boundedness and lower semicontinuity of weak subsolutions, weak Harnack inequalities for weak supersolutions, Harnack inequalities, and local H\"older continuity of weak solutions. Our approach combines weighted analytic techniques with the De Giorgi--Nash--Moser iteration method, adapted to the mixed local--nonlocal setting. To the best of our knowledge, this is the first systematic regularity theory for mixed local and nonlocal equations with Muckenhoupt weights. In particular, our results are new even for homogeneous linear equations ($p=2$) under the natural assumption $w\in A_2$, and therefore substantially extend the existing regularity theory for mixed local--nonlocal equations to a degenerate weighted framework with general right-hand sides.
\end{abstract}

\maketitle

\noi {Keywords: Mixed local and nonlocal weighted $p$-Laplace equations, Muckenhoupt weights, De Giorgi-Nash-Moser theory, local boundedness, Harnack inequality, H\"older continuity, semicontinuity.}

\noi{\textit{2020 Mathematics Subject Classification: 35B65, 35J70, 35J92, 35R11.}

\bigskip

\tableofcontents

\section{Introduction and main results}
\subsection{Introduction}
The interaction between local and nonlocal diffusion has attracted considerable attention in recent years owing to its numerous applications in nonlinear analysis, probability theory, continuum mechanics, phase transitions, anomalous diffusion, image processing, and material science, see \cite{Hitchhiker'sguide, Valdinoci} and the references therein. Mixed local--nonlocal equations combine the classical diffusion generated by local differential operators with long-range interactions described by nonlocal operators, thereby providing a more realistic mathematical framework for modeling several physical, biological, and engineering phenomena. Although the regularity theory for purely local and purely nonlocal equations has witnessed remarkable developments over the past decades, the corresponding theory for mixed equations remains comparatively less understood. The analysis becomes significantly more challenging in the presence of Muckenhoupt weights, where the interaction between the local and nonlocal operators occurs in a highly nonhomogeneous geometric setting.

The primary objective of the present paper is to develop a comprehensive local regularity theory for a broad class of mixed local and nonlocal weighted quasilinear equations. More precisely, we establish local boundedness of weak subsolutions, lower semicontinuity, weak Harnack inequalities for weak supersolutions, Harnack inequalities, and local H\"older continuity of weak solutions to the mixed local and nonlocal weighted equation
\begin{align}\label{ME}
    \mathcal{M}_\alpha u=\mathcal{F}(x,u)\quad \text{in }\Omega,
\end{align}
where $\Omega\subset\mathbb{R}^n$ is a bounded domain with $n\ge2$. The operator $\mathcal{M}_\alpha$ and the nonlinearity $\mathcal{F}$ are explicitly defined below in \eqref{mo} and \eqref{F}, respectively.

Our motivation stems from the fact that, unlike the purely local and purely nonlocal settings, a regularity theory for mixed local and nonlocal weighted equations is essentially absent from the existing literature. To the best of our knowledge, even in the linear case $p=2$, under the natural assumption $w\in A_2$ (the Muckenhoupt weights), and in the homogeneous case $\mathcal{F}\equiv0$, the fundamental qualitative properties of weak solutions, such as local boundedness, weak Harnack inequalities, Harnack inequalities, and local H\"older continuity, have not yet been established. Consequently, the results obtained in the present paper are new even for mixed weighted linear equations and provide, to the best of our knowledge, the first systematic regularity theory for this class of problems.

Unless otherwise mentioned, throughout the rest of the paper, the nonlinearity $\mathcal{F}:\Omega\times\mathbb{R}\to\mathbb{R}$ is a measurable function assumed to satisfy the natural growth condition
\begin{equation}\label{F}
|\mathcal{F}(x,\xi)|
\leq |f(x)|+|g(x)||\xi|^{p-1},
\end{equation}
for almost every $(x,\xi)\in\Omega\times\mathbb{R}$, where $f$ and $g$ are real-valued measurable functions satisfying
\[
\frac{f}{w},\,\frac{g}{w}\in L_{\mathrm{loc}}^{q}(\Omega,w),
\]
with $1<p<\infty$, $w\in A_p$ (the Muckenhoupt weights defined below in section 2), and $q>\kappa'$ (given below in \eqref{exponent}) together with the condition
\begin{align}\label{conditiong}
\sup_{B_r(x_0)\Subset\Omega}
r^p
\left(
\fint_{B_r(x_0)}
\left(\frac{|g|}{w}\right)^q
w\,dx
\right)^{\frac{1}{q}}
\leq M,
\end{align}
for some constant $M>0$.

For $\alpha\geq0$, $0<s<1<p<\infty$, and $w\in A_p$, the mixed local and nonlocal weighted $p$-Laplace operator $\mathcal{M}_\alpha$ is defined by
\begin{equation}\label{mo}
\mathcal{M}_\alpha u
:=
\operatorname{div}\big(\mathcal{A}(x,\nabla u)\big)
+\alpha(-\Delta_{p,w})^{s}u,
\end{equation}
where the local part is governed by a weighted quasilinear operator satisfying the standard homogeneity, $p$-growth, and ellipticity assumptions, while the nonlocal part is represented by the weighted fractional $p$-Laplacian. To be more precise, $\mathcal{A}:\Omega\times\mathbb{R}^n\to \mathbb{R}^n$ is a Carath\'eodory function, which satisfies the following hypothesis for almost every $x\in\Omega$ and for every $\xi\in\mathbb{R}^n$:
\begin{enumerate}
\item[(A1)] $\mathcal{A}(x,c\,\xi)=|c|^{p-2}c\,\mathcal{A}(x,\xi),\quad\forall c\in\mathbb{R}$,

\item[(A2)] $\mathcal{A}(x,\xi)\cdot \xi\geq C_1|\xi|^p w$,

\item[(A3)] $|\mathcal{A}(x,\xi)|\leq C_2|\xi|^{p-1}w$,
\end{enumerate}
for some positive constants $C_1$ and $C_2$.

For example, the function $\mathcal{A}(x,\xi)=w(x)H(\nabla u)^{p-1}\nabla_{\xi}H(\nabla u)$, with $1<p<\infty$ satisfies the conditions (A1)--(A3) above, where $H$ is the Finsler-Minkowski norm, that is $H:\mb{R}^n\to[0,\infty)$ is $C^1(\mathbb{R}^n\setminus\{0\})$, which is a strictly convex function, such that $H(\xi)=0$ iff $\xi=0$, $H(t\xi)=|t|H(\xi)$ for every $\xi\in\mathbb{R}^n$, $t\in\mathbb{R}$ and there exist constants $c_1,~c_2>0$ such that $c_1|\xi|\leq H(\xi)\leq c_2|\xi|$ for every $\xi\in\mathbb{R}^n$. Here the symbol $\nabla_\xi H(\na u)$ denotes $\left.\nabla_{\xi} H(\xi)\right|_{\xi=\nabla u(x)}$, where $\nabla_{\xi} H(\xi)$ represents the gradient of $H$ with respect to its own variable $\xi$ and $\nabla u(x)$ denotes the gradient of $u$ with respect to the spatial variable $x$. We present some examples now to provide a little more insight into the weighted anisotropic $p$-Laplace operator $H_{p,w}$ defined by
$$
H_{p,w}\,u:=\text{div}(w(x)H(\nabla u)^{p-1}\nabla_{\xi}H(\nabla u)).
$$ 
For further background on the anisotropic $p$-Laplace operator, we refer the reader to \cite{BFKzamp, Xiathesis, MV} and the references therein.

\medskip
\noindent\textbf{Examples.} Let $x=(x_1,x_2,\ldots,x_n)\in\mathbb{R}^n$.

\begin{enumerate}
\item[(i)] For $q>1$, define
\begin{equation}\label{ex11}
F_q(x):=\left(\sum_{i=1}^{n}|x_i|^q\right)^{\frac1q}.
\end{equation}

\item[(ii)] For $\lambda,\mu>0$, define
\begin{equation}\label{ex2}
F_{\lambda,\mu}(x):=\sqrt{\lambda\sqrt{\sum_{i=1}^{n}x_i^{4}}+\mu\sum_{i=1}^{n}x_i^{2}}.
\end{equation}
\end{enumerate}

It follows from \cite{MV} that the functions $F_q,F_{\lambda,\mu}:\mathbb{R}^n\to[0,\infty)$, defined by \eqref{ex11} and \eqref{ex2}, respectively, are Finsler--Minkowski norms.

\begin{Remark}\label{exrmk1}
Let $i=1,2$, and let $\lambda_i,\mu_i>0$ satisfy
\[
\frac{\lambda_1}{\mu_1}\neq\frac{\lambda_2}{\mu_2}.
\]
Then the norms $F_{\lambda_1,\mu_1}$ and $F_{\lambda_2,\mu_2}$, defined by \eqref{ex2}, are nonisometric on $\mathbb{R}^n$; see \cite{MV}.
\end{Remark}

\begin{Remark}\label{exrmk2}
For the choice $H=F_q$ in \eqref{ex11}, the corresponding anisotropic $p$-Laplace operator takes the form
\begin{equation*}
H_{p,w}u=
\sum_{i=1}^{n}\frac{\partial}{\partial x_i}
\left(w(x)
\left(\sum_{k=1}^{n}\left|\frac{\partial u}{\partial x_k}\right|^{q}\right)^{\frac{p-q}{q}}
\left|\frac{\partial u}{\partial x_i}\right|^{q-2}
\frac{\partial u}{\partial x_i}
\right).
\end{equation*}
In particular, the operator $H_{p,w}$ defined above reduces to the weighted $p$-Laplace operator when $q=2$, and to the weighted pseudo $p$-Laplace operator when $q=p$, namely,
\begin{equation*}
H_{p,w}u=
\begin{cases}
\operatorname{div}\bigl(w(x)|\nabla u|^{p-2}\nabla u\bigr),
& \text{if } q=2,\;1<p<\infty,\\[1ex]
\displaystyle\sum_{i=1}^{n}
\frac{\partial}{\partial x_i}
\left(w(x)|u_i|^{p-2}u_i\right),
& \text{if } q=p\in(1,\infty),
\end{cases}
\end{equation*}
where $u_i=\frac{\partial u}{\partial x_i}$ for $i=1,\ldots,n$.
\end{Remark}
The weighted fractional $p$-Laplace operator (see \cite{Ok24}) is denoted by $(-\Delta)^s_{p,w}u$ and defined as
$$(-\Delta)^s_{p,w}u(x):=\text{P.V.}\int_{\R^n}\frac{|u(x)-u(y)|^{p-2}(u(x)-u(y))}{|x-y|^{sp}}K(x,y)\;dy,$$
where P.V. denotes the principal value and the kernel $K:\R^n\times\R^n\to\R$ is a symmetric Lebesgue measurable function such that for almost every $x,y\in\mathbb{R}^n$, we have
\begin{align}\label{kernal}
   \la \frac{w(x)w(y)}{w(B_{x,y})}\leq K(x,y)\leq \Lambda \frac{w(x)w(y)}{w(B_{x,y})},
\end{align}
with $B_{x,y}:=B_{\frac{|x-y|}{2}}(\frac{x+y}{2})$, $w(B_{x,y}):=\int_{B_{x,y}}w(z)\;dz$ and $\la\geq 0,\Lambda>0$ are constants.

Consequently, equation \eqref{ME} generalizes the mixed anisotropic and fractional weighted $p$-Laplace equation
\begin{equation*}
-\operatorname{div}\!\left(w(x)H(\nabla u)^{p-1}\nabla_{\xi}H(\nabla u)\right)
+\alpha(-\Delta_{p,w})^s u
=\mathcal{F}(x,u)
\quad\text{in }\Omega,
\end{equation*}
which, in turn, generalizes the mixed local and nonlocal weighted $p$-Laplace equation
\begin{equation*}
-\operatorname{div}\!\left(w(x)|\nabla u|^{p-2}\nabla u\right)
+\alpha(-\Delta_{p,w})^s u
=\mathcal{F}(x,u)
\quad\text{in }\Omega.
\end{equation*}
Therefore, equation \eqref{ME} encompasses a broad class of local and nonlocal equations arising in nonlinear analysis.

The regularity theory for nonlinear elliptic equations has a long and distinguished history. In the purely local setting, qualitative properties of weak solutions to the $p$-Laplace equation, including local boundedness, Harnack inequalities, H\"older continuity, and higher regularity, are by now classical. We refer to \cite{ Dibe1, PLin1,  Moser, PTolks, Uhl} and the references therein for comprehensive accounts. The weighted theory has also received considerable attention. A pioneering contribution is due to \cite{Fabes82}, where the authors established the Harnack inequality and H\"older continuity for weak solutions of the linear weighted equation
\[
\operatorname{div}(w(x)\nabla u)=0,
\]
under the assumption $w\in A_2$. Their work laid the foundation for the study of degenerate elliptic equations with Muckenhoupt weights. Subsequently, nonlinear weighted equations of the form
\begin{equation*}
\Delta_{p,w}u=0
\quad\text{in }\Omega,
\end{equation*}
and its associated nonhomogeneous version were investigated in \cite{Heinonen, E.W., Wu} and the references therein. Very recently, regularity theory for double phase problems with Muckenhoupt  condition have been studied in \cite{Deining26, Tran1, Tran2}.

The regularity theory for nonlocal equations has undergone remarkable developments over the last two decades. Interior regularity, Harnack inequalities, and H\"older continuity for fractional and integro-differential equations have been established in numerous works; see, for example, \cite{BiswasLip, Lind2, Lind1, DKPhar, DKPahp, Silvestre, Kas07, Kas09, Kas12}. In contrast, the weighted nonlocal theory is still in its infancy. To the best of our knowledge, the only work establishing regularity results for weighted fractional $p$-Laplace equations is \cite{Ok24}. More precisely, the authors proved local boundedness of weak subsolutions, weak Harnack inequalities for weak supersolutions, Harnack inequalities, and local H\"older continuity of weak solutions to
\begin{equation*}
(-\Delta_{p,w})^su=0
\quad\text{in }\Omega,
\end{equation*}
under the assumptions $1<p<\infty$, $0<s<1$, and $w\in A_p$.

Motivated by these developments, the regularity theory for mixed local and nonlocal equations has attracted considerable attention in recent years. Early contributions in the linear setting relied primarily on probabilistic techniques; see, for instance, \cite{Chen, Chen12, Fo}. Subsequently, purely analytical methods have been developed to study both linear mixed equations, see \cite{Cozzisiam, Valdinocicpde, VecchiBO, BVV, Valdinoci} and the references therein. For nonlinear mixed equations, leading to local boundedness, H\"older continuity, weak Harnack inequalities, and Harnack inequalities among other regularity estimates in the unweighted setting can be found in \cite{Cozzijde, Biswas, BTopp, DasBiswas, Iwonajlms, Min, Garainna, Garainjde, GKK, GK, GLcvpde} and the references therein. These works have significantly advanced the understanding of mixed diffusion phenomena.

The preceding discussion naturally raises the question of whether the classical regularity theory can be extended to mixed local and nonlocal equations in the weighted setting. The present paper provides an affirmative answer. The interaction between the local and nonlocal weighted operators gives rise to substantial analytical difficulties that cannot be addressed by a direct combination of the existing theories for the purely local and purely nonlocal settings.

The principal objective of the present paper is to bridge this gap by developing a comprehensive local regularity theory for the mixed weighted equation \eqref{ME}. More precisely, under the general growth assumption \eqref{F} on the right-hand side, we establish local boundedness of weak subsolutions, lower semicontinuity, weak Harnack inequalities for weak supersolutions, Harnack inequalities, and local H\"older continuity of weak solutions. To the best of our knowledge, these results are the first of their kind for mixed local and nonlocal weighted equations. In particular, they are new even in the linear case $p=2$, under the assumption $w\in A_2$, and for the homogeneous equation $\mathcal{F}\equiv0$. Consequently, the present work provides the first systematic regularity theory for mixed local and nonlocal weighted equations and extends, within a unified framework, the existing regularity theories for both purely local and purely nonlocal weighted problems.

\subsection{Main results}
The main results of this article are stated below. The first one is the following local boundedness result for weak subsolutions. The main ingredients are Lemma \ref{WIT}, Lemma \ref{Energy} and the iteration Lemma \ref{seqconv}.
\begin{Theorem}[Local boundedness]\label{Bdd}
Let $\alpha>0$ and $u$ be a weak subsolution of equation \eqref{ME}. Then, for any $\delta\in(0,1]$, there exists a constant $C=C(n,p,s,q,\kappa,\Lambda,\alpha,C_1,C_2,[w]_p,M)>0$ such that
\begin{align*}
    \sup_{B_\frac{r}{2}(x_0)}u&\leq C\delta^{-\frac{p-1}{p}\frac{\kappa q'}{\kappa-q'}}\Big(\fint_{B_r(x_0)}u_+^p\,w\;dx\Big)^\frac{1}{p}+\delta \mathrm{Tail}\Big(u_+;x_0,\frac{r}{2}\Big)\nonumber+r^{p'}\Big(\fint_{B_r(x_0)}\Big(\frac{|f|}{w}\Big)^qw\;dx\Big)^\frac{1}{q(p-1)}
\end{align*}
whenever $B_r(x_0)\Subset\Omega$ with $0<r\leq 1$. Here, $M,\;\kappa$ and $\mathrm{Tail}$ are given by \eqref{conditiong}, \eqref{exponent} and \eqref{tail}, respectively.  
\end{Theorem}

Our second main result is the following local H\"older continuity for weak solutions.
\begin{Theorem}[Local H\"older continuity]\label{cty}
Let $\alpha>0$ and $u$ be a weak solution of \eqref{ME}, with $\mathcal{F}$ satisfying \eqref{F} such that $f/w\in L^q_{\mathrm{loc}}(\Omega,w)$ with $q>n$ and $g\equiv 0$. Then $u$ is locally H\"{o}lder continuous in $\Omega$. Furthermore, there exist constants $C>0$ and ${\sigma}\in (0,1)$, depending on $n,\;p,\;q,\;s,\;\Lambda,$ and $[w]_p$, such that  
    \begin{align}\label{Holder}
    \mathrm{osc}_{B_\rho(x_0)}\;u&:=\sup_{B_\rho(x_0)} u-\inf_{B_\rho(x_0)}\;u\nonumber\\
    &\leq C\Big(\frac{\rho}{r}\Big)^{{\sigma}}\Big\{\mathrm{Tail}\Big(u;x_0,\frac{r}{2}\Big)+\Big(\fint_{B_{r}(x_0)}|u|^pw\;dx\Big)^\frac{1}{p}+
    r^{p'}\Big(\fint_{B_{r}(x_0)}\Big(\frac{|f|}{w}\Big)^qw\;dx\Big)^\frac{1}{q(p-1)}\Big\},
    \end{align}
    for every $B_{r}(x_0)\Subset\Omega$ with $0<r\leq 1$ and $0<\rho\leq\frac{r}{2}$.
\end{Theorem}

Our third main results is devoted to the Harnack  inequality for weak solutions.
\begin{Theorem}[Harnack inequality]\label{Harthm}
Let $\alpha>0$ and $u$ be a weak solution of \eqref{ME} such that $u\geq 0$ in $B_R(x_0)\Subset\Omega$, with $\mathcal{F}$ satisfying \eqref{F}, where $f/w,\; g/w \in L^q_{\mathrm{loc}}(\Omega,w)$ with $q>n$ and $g$ satisfying \eqref{conditiong}. Assume, in addition to the standing hypotheses, that the constant $\lambda>0$ in \eqref{kernal}. Then there exists {a} constant $C=C(n,p,s,q,\kappa,{\lambda},\Lambda,\alpha,C_1,C_2,[w]_p,M)>0$ such that 
    \begin{align}
        \sup_{B_\frac{r}{2}(x_0)} u\leq C\Big\{\inf_{B_r(x_0)} u+\Big(\frac{r}{R}\Big)^{p'}\Tail(u_-;x_0,R)+r^{p'}\Big(\fint_{B_{8r}(x_0)}\Big(\frac{|f|}{w}\Big)^q~w\;dx\Big)^\frac{1}{q(p-1)}\Big\},
    \end{align}
    for every $r\in(0,1]$ with $r<\frac{R}{16}$.
\end{Theorem}

Our fourth main result is concerning the weak Harnack inequality for weak solutions.
\begin{Theorem}[Weak Harnack inequality]\label{wkHarnack}
Let $\alpha>0$ and $u$ be a weak supersolution of equation (\ref{ME}) such that $u\geq 0$ in $B_R(x_0)\Subset\Omega$, with $\mathcal{F}$ satisfying \eqref{F} such that $f/w,\; g/w \in L^q_{\mathrm{loc}}(\Omega,w)$ with $q>n$ and $g$ satisfying \eqref{conditiong}.  Then for every $0<l<\kappa(p-1)$, there exists a constant $C=C(n,p,q,s,\Lambda,\alpha,C_1,C_2,l,[w]_p,M)>0$, such that
\begin{align}\label{FWH}
    \Big(\fint_{B_\frac{r}{2}(x_0)}u^{l}w\;dx\Big)^\frac{1}{l}\leq C\Big\{\inf_{B_\frac{r}{2}(x_0)}u+\Big(\frac{r}{R}\Big)^{p'}\Tail(u_-;x_0,R)+r^{p'}\Big(\fint_{B_{8r}(x_0)}\Big(\frac{|f|}{w}\Big)^qw\;dx\Big)^\frac{1}{q(p-1)}\Big\},
\end{align}
 for every $r\in(0,1]$ with $r<\frac{R}{16}$.
\end{Theorem}

Finally, we have the following semicontinuity result.
\begin{Theorem}[Lower semicontinuity]\label{lscthm}
    {Suppose $\mathcal{F}$ satisfies (\ref{F}) where ${f}/{w},{g}/{w}\in L^q_{\mathrm{loc}}(\Omega,w)$ with $q>\kappa'$, and $f,g$ satisfy condition (\ref{conditiong}).} Let $\alpha>0$ and $u$ be a weak supersolution of \eqref{ME}, which is bounded below in $\R^n$. Then $u$ is lower semicontinuous in $\Omega$. {Here, $\kappa$ is given in (\ref{exponent}).}
\end{Theorem}

\begin{Corollary}[Upper semicontinuity]\label{upper}
    {Suppose $\mathcal{F}$ satisfies (\ref{F}) where ${f}/{w},{g}/{w}\in L^q_{\mathrm{loc}}(\Omega,w)$ with $q>\kappa'$, and $f,g$ satisfy condition (\ref{conditiong}).} Let $\alpha>0$ and $u$ be a weak subsolution of \eqref{ME}, which is bounded above in $\R^n$. Then $u$ is upper semicontinuous in $\Omega$.
\end{Corollary}

\begin{Remark}\label{optmq}
We remark that the integrability condition $q>n$ in Theorems~{\ref{cty}}--\ref{wkHarnack} is not expected to be optimal. Indeed, in the unweighted setting $w\equiv 1$, by Remark \ref{exprmk}, we can choose $\kappa=n/(n-p)$ when $1<p<n$ and therefore, proceeding along the lines of the proof, all our main results are valid under the assumption $q>n/p$. We also refer to \cite{DasBiswas, Garainjde} in this concern.
\end{Remark}

\begin{Remark}\label{rmk-lambda0}
It is worth emphasizing that all the main results above remain valid under the standing assumption $\lambda\geq 0$ and $\Lambda>0$, with the sole exception of the Harnack inequality (Theorem~\ref{Harthm}), where the stronger assumption $\lambda>0$ is required, mainly due to the tail estimate in Lemma \ref{Tailestimate}. In particular, the local boundedness, local H\"older continuity, weak Harnack inequality, and semicontinuity results are established without any positive lower bound on the kernel. To the best of our knowledge, these results are new in the literature even in the case $w\equiv 1$ for $\la=0$.
\end{Remark}

\begin{Remark}\label{mrrmk}
Specifically, if $\alpha=0$, the exact same proof gives that the corresponding estimates in Theorems \ref{Bdd}--\ref{wkHarnack} hold without the tail terms appearing on the right-hand sides, and the constants $C$ are independent of $\alpha$, $\lambda$, $\Lambda$, and $s$. Furthermore, Theorem \ref{lscthm} remains valid if the assumption that $u$ is bounded below in $\mathbb{R}^n$ is replaced by the weaker assumption that $u$ is bounded below only in $\Omega$.
\end{Remark}

The proofs of our main results rely on a delicate adaptation of the De~Giorgi--Nash--Moser theory to the weighted  framework. The main difficulty stems from the simultaneous presence of three distinct features: the Muckenhoupt weight, the interaction between the local and nonlocal operators, and the nonhomogeneous right-hand side satisfying the general growth condition \eqref{F}. None of the existing regularity theories can be applied directly, since the local weighted theory, the weighted nonlocal theory, and the mixed unweighted theory each treat only one aspect of the present problem.

The Muckenhoupt weight $w$ introduces a highly nonhomogeneous geometry, under which many of the standard tools in the classical De~Giorgi theory require substantial modifications. In particular, the weighted Sobolev and Poincar\'e inequalities, Caccioppoli estimates, covering arguments, and iteration procedure must all be carefully adapted to the weighted setting. At the same time, the mixed nature of the operator necessitates a delicate balance between the local energy estimates and the nonlocal tail contributions throughout the iteration. An additional difficulty arises from the lower-order term $\mathcal{F}(x,u)$, whose dependence on the solution through the growth condition \eqref{F} prevents it from being treated as a simple forcing term. The assumption \eqref{conditiong} plays a crucial role in controlling this contribution during the iteration process and in obtaining estimates that remain stable across different scales.

Our proofs build upon and unify several different regularity theories. More precisely, they combine the nonlocal De Giorgi-Nash-Moser theory developed in \cite{DKPhar, DKPahp}, the weighted fractional regularity theory in \cite{Ok24}, the weighted local theory in \cite{Fabes82}, and the analytical framework for mixed local--nonlocal equations introduced in \cite{GK}. A careful synthesis of these ideas enables us to overcome the difficulties created by the interaction of the weight, the mixed diffusion, and the nonlinear lower-order term, leading to a unified regularity theory for equation \eqref{ME}.

\subsection*{Notation and standing assumptions} Throughout the rest of the paper, we make use of the following notation and assumptions, unless otherwise mentioned.
\begin{itemize}

\item $\Omega$ will denote an open and bounded subset of $\mathbb{R}^n$ with $n\geq 2$.

\item For a given subset $\omega$ of an open set $\mathcal{O}$ in $\mathbb{R}^n$, we write $\omega\Subset\mathcal{O}$, if $\omega$ is compactly contained in $\mathcal{O}$.

\item Recall the constants $\la\geq 0$ and $\Lambda>0$ in \eqref{kernal}.

\item Recall $\kappa=\frac{n}{n-1}+\eta$ as given in \eqref{exponent}.

\item $C_1,C_2$ will denote the constants in the {hypotheses} (A2)-(A3), whereas {$\alpha>0$} will denote the constant involved in the operator $\mathcal{M}_\alpha$ and $M>0$ will denote the constant in the condition \eqref{conditiong}, respectively.

\item For a given measurable subset $S$ of $\mathbb{R}^n$ and given functions $f,g$ and $h$ on $S$, we write $f\leq g\leq h$ on $S$ to mean $f\leq g\leq h$ almost everywhere in $S$.

    \item We say $A\backsimeq B$, if there exist hidden constants $c_1,c_2$ both positive such that $$c_1B\leq A \leq c_2 B.$$

    \item The conjugate H\"{o}lder exponent of $l>1$ is denoted by $l'$ and defined by $l':=\frac{l}{l-1}$.

    \item For a measurable subset $E \subset\R ^ n$, $|E|$ denotes the Lebesgue measure of $E$. Moreover, we denote $w(E):=\int_E w(x)\;dx.$ Further, $\overline{E}$ denotes the closure of $E$.
    
\item By $\sup$ and $\inf$, we mean $\esssup$ and $\essinf$, respectively.

    \item For $r>0$ and $x_0\in\mathbb{R}^n$, we denote $B_r(x_0):=\{x\in\mathbb{R}^n:|x-x_0|<r\}$.

    \item We also recall the notation $B_{x,y}:=B_\frac{|x-y|}{2}(\frac{x+y}{2})$.

    \item For a measurable function $f:E\to \R$, defined on a measurable subset $E\subset\mathbb{R}^n$, we define
    $$f_E:=\fint_E fw\;dx=\frac{1}{w(E)}\int_E f w\;dx.$$

    \item For $k\in\mathbb{R}$, we denote by $k_{\pm}:=\max\{\pm k,0\}$.

    \item By $c$ or $C$, we mean a generic constant, whose values may change from line to line or even on the same line. If a constant $C$ depends on the parameters $l_1,\ldots,l_m$, then we write $C=C(l_1,\ldots,l_m)$.
    
\end{itemize}

\subsection*{Organization of the paper}
The remainder of the paper is organized as follows. In Section~2, we present the necessary functional setting and auxiliary results, while Section~3 is devoted to establishing the required preliminaries. Finally, in Section 4, we prove our main results.

\section{Functional setting and auxiliary results}
\subsection{Functional setting}
Throughout the paper, unless otherwise mentioned, we assume that $1<p<\infty$ and $w$ belongs to the Muckenhoupt class $A_p$ (see \cite{Muc}), which means $w$ is a positive, locally integrable function on $\mathbb{R}^n$ such that
\begin{equation}\label{Ap}
[w]_{p}:=
\sup_{B\subset\mathbb{R}^n}
\left(\frac{1}{|B|}\int_Bw\,dx\right)
\left(\frac{1}{|B|}\int_Bw^{-\frac{1}{p-1}}\,dx\right)^{p-1}
<\infty,
\end{equation}
where the supremum is taken over all balls $B\subset\mathbb{R}^n$.

A fundamental example of an $A_p$ weight is the power weight
\[
w(x)=|x|^{\sigma},\qquad -n<{\sigma}<n(p-1).
\]
For $w\in A_p$ with $1<p<\infty$ and an open subset $\Omega\subset\mathbb{R}^n$, the weighted Sobolev space $W^{1,p}(\Omega,w)$ is defined by
$$W^{1,p}(\Omega,w):=\{u\in L^p(\Omega,w): |\nabla u|\in L^p(\Omega,w)\},$$
equipped with the norm
$$\|u\|_{W^{1,p}(\Omega,w)}:=\big(\|u\|_{L^{p}(\Omega,w)}^p+\|\nabla u\|_{L^{p}(\Omega,w)}^p\big)^\frac{1}{p},$$
where $L^p(\Omega,w)$ is the weighted Lebesgue space defined by
\[
L^p(\Omega,w):=
\left\{
u:\Omega\to\mathbb{R}\text{ measurable }:
\int_\Omega |u|^pw\,dx<\infty
\right\},
\]
endowed with the norm
\[
\|u\|_{L^p(\Omega,w)}
:=
\left(\int_\Omega |u|^pw\,dx\right)^{1/p}.
\]
We write $u\in L^p_{\mathrm{loc}}(\Omega,w)$ if
$u\in L^p({\Omega'},w)$ for every ${\Omega'}\Subset\Omega$. Similarly, we say that $u\in W^{1,p}_{\mathrm{loc}}(\Omega,w)$ if
$u\in W^{1,p}({\Omega'},w)$ for every ${\Omega'}\Subset\Omega$. When $\alpha>0$, to deal with the mixed problems, we define the space $W^{1,p}_{0}(\Omega,w)$ by
$$
W^{1,p}_{0}(\Omega,w):=\{u\in W^{1,p}(\mathbb{R}^n,w):u=0\text{ in }\mathbb{R}^n\setminus\Omega\}.
$$
When $\alpha=0$, we consider the space $W_0^{1,p}(\Omega,w)$ as the closure of $C_c^{\infty}(\Omega)$ under the norm $\|\cdot\|_{W^{1,p}(\Omega,w)}$.
For more details on the weighted Sobolev spaces, we refer to \cite{Heinonen} and the references therein.

We define the tail space $L^{p-1}_{ps}(\R^n,w)$ by
$$L^{p-1}_{ps}(\R^n,w):=\Bigg\{u\in L^{p-1}_{\mathrm{loc}}(\R^n,w):\int_{\R^n}\frac{|u(x)|^{p-1}}{(1+|x|^{sp})}\frac{w(x)dx}{w(B_{x,0})}<\infty\Bigg\}.$$
For $u\in L^{p-1}_{ps}(\R^n,w)$, the mixed local and nonlocal weighted tail of $u$ with respect to the ball $B_r(x_0)$ for $x_0\in\mb{R}^n$ and $r>0$ is denoted by $\mathrm{Tail}(u;x_0,r)$ and defined as
\begin{align}\label{tail}
    \mathrm{Tail}(u;x_0,r):=\Bigg(r^p\int_{\R^n\setminus B_r(x_0)}\frac{|u(x)|^{p-1}}{|x-x_0|^{sp}}\frac{w(x)dx}{w(B_{x,x_0})}\Bigg)^\frac{1}{p-1}.
\end{align}
We are now ready to define the notion of weak solutions to equation \eqref{ME}. To this end, we define the solution space $X_\alpha$ by
\begin{equation}\label{ss}
X_\alpha:=\begin{cases}
    W^{1,p}_{\mathrm{loc}}(\Omega,w)\text{ if }\alpha=0,\\
    W^{1,p}_{\mathrm{loc}}(\Omega,w)\cap L^{p-1}_{ps}(\R^n,w)\text{ if }\alpha>0.
\end{cases}
\end{equation}
\begin{Definition}\label{def}
    Let $\alpha\geq 0$,\,$0<s<1<p<\infty$ and $w\in A_p$. We say that a function $u\in X_\alpha$ is a weak subsolution (or supersolution) of equation \eqref{ME} if for every $\Omega'\Subset\Omega$ and for every nonnegative $\varphi \in W^{1,p}_{0}(\Omega',w)$, the following holds:
    \begin{equation}\label{wksoleqn}
    \begin{split}
        &\int_\Omega\mathcal{A}(x,\nabla u)\cdot\nabla\varphi\;dx\\
        &+\alpha\,\int_{\R^n}\int_{\R^n}\frac{|u(x)-u(y)|^{p-2}(u(x)-u(y))(\varphi(x)-\varphi(y))}{|x-y|^{sp}} K(x,y)\;dx\;dy\\
      &\leq~(\mbox{ or } \geq) \int_\Omega \mathcal{F}(x,u)\varphi\;dx.
    \end{split}
    \end{equation}
    A function $u\in X_\alpha$ is said to be a weak solution to \eqref{ME} if it is both a weak subsolution and a weak supersolution of \eqref{ME}.
\end{Definition}
\begin{Remark}\label{wdws}
When $\alpha>0$, by Lemma \ref{lemmatail} proved in the next subsection, Definition \ref{def} is well stated.
\end{Remark}

\subsection{Auxiliary results}
In this subsection, we establish Lemma \ref{lemmatail} and further mention some known results, which are crucial for us. To this end, first we have the following lemma.
\begin{Lemma}\label{tailf}
    If $u\in L^{p-1}_{ps}(\R^n,w)$, then for every $r>0$ and $x_0\in\mathbb{R}^n$, we have $\Tail(u;x_0,r)<\infty.$
\end{Lemma}
\begin{proof}
We observe that
\begin{align}\label{R1}
    \int_{\R^n\setminus B_r(x_0)}\frac{|u(x)|^{p-1}}{|x-x_0|^{sp}}\frac{w(x)dx}{w(B_{x,x_0})}=\int_{\R^n\setminus B_r(x_0)}\frac{|u(x)|^{p-1}}{1+|x|^{sp}}\frac{1+|x|^{sp}}{|x-x_0|^{sp}}\frac{w(B_{x_0,0})}{w(B_{x,x_0})}\frac{w(x)dx}{w(B_{x_0,0})}.
\end{align}
Taking $|x-x_0|>r$ into account, one has
\begin{align}\label{R2}
    \frac{1+|x|^{sp}}{|x-x_0|^{sp}}\leq c(s,p)\Big(\frac{1+|x-x_0|^{sp}+|x_0|^{sp}}{|x-x_0|^{sp}}\Big)\leq c(s,p)\bigg(1+\frac{1+|x_0|^{sp}}{r^{sp}}\bigg)<\infty.
\end{align}
{Utilizing the facts} $B_{x,0}\subset B_{\frac{|x|+|x_0|}{2}}(\frac{x+x_0}{2})$ and $B_{x,x_0}\subset B_{\frac{|x|+|x_0|}{2}}(\frac{x+x_0}{2})$ along with Lemma \ref{comlem}, we obtain 
\begin{align}\label{R3}
    \frac{w(B_{x_0,0})}{w(B_{x,x_0})}&=\frac{w(B_{x_0,0})}{w(B_{\frac{|x|+|x_0|}{2}}(\frac{x+x_0}{2}))}\frac{w(B_{\frac{|x|+|x_0|}{2}}(\frac{x+x_0}{2}))}{w(B_{x,x_0})}\nonumber\\
    &\leq c[w]_p\Big(\frac{|B_{x_0,0}|}{|B_{\frac{|x|+|x_0|}{2}}(\frac{x+x_0}{2})|}\Big)^\sigma \Big(\frac{|B_{\frac{|x|+|x_0|}{2}}(\frac{x+x_0}{2})|}{|B_{x,x_0}|}\Big)^p\nonumber\\
    &\leq c[w]_p\Big(\frac{|x_0|}{|x|+|x_0|}\Big)^{n\sigma} \Big(\frac{|x|+|x_0|}{|x-x_0|}\Big)^{np}\nonumber\\
    &\leq c[w]_p\bigg(1+\frac{2|x_0|}{r}\bigg)^{np}<\infty.
\end{align}
{Using \eqref{R2} and \eqref{R3} in \eqref{R1}, we obtain $\Tail(u;x_0,r)<\infty$.}
\end{proof}

\begin{Lemma}\label{comparingseminorms}
Suppose $u\in W^{1,p}_{\mathrm{loc}}(\Omega,w)$. Then for every ${\Omega'}\Subset\Omega$, we have
  \begin{align*}
       \int_{\Omega'}\int_{\Omega'} \frac{|u(x)-u(y)|^p}{|x-y|^{sp}}\frac{
    w(x)w(y)}{w(B_{x,y})}\;dx\;dy<\infty.
    \end{align*}
\end{Lemma}
\begin{proof}
    Let $\Omega'\Subset\Omega''\Subset\Omega$ and $r:=\frac{1}{2}\mathrm{dist}(\Omega',\partial\Omega'')>0$. Thus, using the compactness of $\overline{\Omega'}$, there exist finitely many balls $B_i:=B_{r}(x_i)$ for $i=1,2,...,N$, such that $\overline{\Omega'}\subset\cup_{i=1}^{N}B_i$. Using partition of unity, there exist a sequence $\{\eta_i\}_{i=1}^{N}\subset C^\infty_c(B_i)$ such that $0\leq \eta_i\leq 1$ in $B_i$ for every $i\in\{1,\ldots,N\}$ and $\sum_{i=1}^N\eta_i=1$ in $\Omega'$. We observe that in $\Omega'$, $u=\sum_{i=1}^{N}\eta_iu=\sum_{i=1}^{N}u_i$ with $u_i=\eta_iu\in W^{1,p}_{0}(B_i,w)$. Using \cite[Lemma 2.13]{Ok24}, we have 
    \begin{align}\label{SD}
     \int_{B_i}\int_{B_i}\frac{|u_i(x)-u_i(y)|^p}{|x-y|^{sp}}\frac{
    w(x)w(y)}{w(B_{x,y})}\;dx\;dy\leq C\int_{B_i}|\nabla u_i|^pw\;dx,
    \end{align}
    for some constant $C=C(n,p,[w]_p)>0$.
Now we observe that
\begin{align}\label{SD6}
    \int_{\Omega'}\int_{\Omega'} \frac{|u(x)-u(y)|^p}{|x-y|^{sp}}\frac{
    w(x)w(y)}{w(B_{x,y})}\;dx\;dy&=\int_{\Omega'}\int_{\Omega'} \frac{|\sum_{i=1}^Nu_i(x)-\sum_{i=1}^Nu_i(y)|^p}{|x-y|^{sp}}\frac{
    w(x)w(y)}{w(B_{x,y})}\;dx\;dy\nonumber\\
    &\leq C\sum_{i=1}^N\int_{\R^n}\int_{\R^n} \frac{|u_i(x)-u_i(y)|^p}{|x-y|^{sp}}\frac{
    w(x)w(y)}{w(B_{x,y})}\;dx\;dy\nonumber\\
    &\leq C\sum_{i=1}^N(I_{i1}+2I_{i2}),
\end{align}
for some constant $C=C(p)>0$, where $$I_{i1}:=\int_{B_i}\int_{B_i}\frac{|u_i(x)-u_i(y)|^p}{|x-y|^{sp}}\frac{
    w(x)w(y)}{w(B_{x,y})}\;dx\;dy,$$
    and
$$I_{i2}:=\int_{B_i}\int_{\R^n\setminus B_i}\frac{|u_i(x)-u_i(y)|^p}{|x-y|^{sp}}\frac{
    w(x)w(y)}{w(B_{x,y})}\;dx\;dy.$$
\textbf{Estimate of $I_{i1}$:}
Taking (\ref{SD}) into account, we obtain 
\begin{align}\label{SD5}
    I_{i1}\leq C\int_{B_i}|\nabla u_i|^pw\;dx&\leq C\Big(\int_{B_i}\eta_i^p|\nabla u|^pw\;dx+\int_{B_i} u^p|\nabla \eta_i|^pw\;dx\Big)\nonumber\\
    &\leq C \Big(\int_{B_i}|\nabla u|^pw\;dx+\int_{B_i} |u|^pw\;dx\Big)\leq C\|u\|_{W^{1,p}(\Omega'',w)},
\end{align}
for some constant $C=C(n,p,[w]_p,i)>0$.\\
\textbf{Estimate of $I_{i2}$:}
Note that if $x\in \R^n\setminus B_i$ and $y\in \mathrm{supp}\,\eta_i\Subset B_i$, then $|x-y|\geq r_i>0$ and hence
\begin{align}\label{SD1}
    \frac{|x-x_i|}{|x-y|}\leq 1+ \frac{|y-x_i|}{|x-y|}\leq 1+\frac{r}{r_i}\leq C,
\end{align}
for some $C>0$, may depend on $i$. Furthermore, since $B_{x,x_i}\cup B_{x,y}\subset B_{2|x-x_i|}(x_i)$, using Lemma \ref{comlem}, we have
\begin{align}\label{SD2}
    \frac{w(B_{x,x_i})}{w(B_{x,y})}\leq \frac{w(B_{x,x_i})}{w(B_{2|x-x_i|}(x_i))}\frac{w(B_{2|x-x_i|}(x_i))}{w(B_{x,y})}&\leq C\Big( \frac{|B_{x,x_i}|}{|B_{2|x-x_i|}(x_i)|}\Big)^\sigma\Big(\frac{|B_{2|x-x_i|}(x_i)|}{|B_{x,y}|}\Big)^p \nonumber\\
    &\leq C\Big(\frac{|x-x_i|}{|x-y|}\Big)^{np}\leq C,
\end{align}
for some constant $C=C(n,p,[w]_p,i)>0$, where we have used (\ref{SD1}). Thanks to Lemma \ref{intlem}, combining (\ref{SD1}) and (\ref{SD2}), we deduce that
\begin{align}\label{SD3}
    I_{i2} &=\int_{B_i}\int_{\R^n\setminus B_i}\frac{|u_i(y)|^p}{|x-y|^{sp}}\frac{
    w(x)w(y)}{w(B_{x,y})}\;dx\;dy\nonumber\\
    &\leq \int_{\mathrm{supp}\,\eta_i}\int_{\R^n\setminus B_i}\frac{|\eta_i(y)u(y)|^p}{|x-y|^{sp}}\frac{
    w(x)w(y)}{w(B_{x,y})}\;dx\;dy\nonumber\\
    &\leq C\int_{\mathrm{supp}\,\eta_i}\Big(\int_{\R^n\setminus B_i}\frac{1}{|x-x_i|^{sp}}\frac{
    w(x)}{w(B_{x,x_i})}\;dx\Big)|u(y)|^pw(y)\;dy\nonumber\\
    &\leq Cr^{-sp}\int_{B_i}|u|^pw\;dy\leq Cr^{-sp}\|u\|_{W^{1,p}(\Omega'',w)},
\end{align}
for some constant $C=C(n,p,[w]_p,i,s)>0$, where we have also used Lemma \ref{intlem} to obtain the above estimate. Utilizing the estimates \eqref{SD5} and \eqref{SD3} in \eqref{SD6}, the result follows.
\end{proof}

\begin{Lemma}\label{lemmatail}
    If $u\in W^{1,p}_{\mathrm{loc}}(\Omega,w) \cap L^{p-1}_{ps}(\R^n,w)$ and $\varphi\in W^{1,p}_{0}(\Omega',w)$ for some $\Omega'\Subset\Omega$, then
    $$\int_{\R^n}\int_{\R^n}\frac{|u(x)-u(y)|^{p-2}(u(x)-u(y))(\varphi(x)-\varphi(y))}{|x-y|^{sp}}K(x,y)\;dx\;dy<\infty.$$
\end{Lemma}
\begin{proof}
Let $\Omega'\Subset\Omega''\Subset\Omega$ with $r:=\frac{1}{4}\mathrm{dist}(\partial\Omega',\partial\Omega'')$. The compactness of $\overline{\Omega'}$ guaranties that there exists $x_1,x_2,...,x_N\in\Omega'$ such that $$\overline{\Omega'}{\subset}\cup_{i=1}^{N}B_r(x_i).$$ Now, we notice that
\begin{align}\label{If}  &\Big|\int_{\R^n}\int_{\R^n}\frac{|u(x)-u(y)|^{p-2}(u(x)-u(y))(\varphi(x)-\varphi(y))}{|x-y|^{sp}}K(x,y)\;dx\;dy\Big|\nonumber\\
    &\leq \Lambda\int_{\R^n}\int_{\R^n}\frac{|u(x)-u(y)|^{p-1}|\varphi(x)-\varphi(y)|}{|x-y|^{sp}}\frac{
    w(x)w(y)}{w(B_{x,y})}\;dx\;dy\nonumber\\
    &:=I_1+2I_2,
\end{align}
where $$I_1:=\Lambda\int_{\Omega''}\int_{\Omega''}\frac{|u(x)-u(y)|^{p-1}|\varphi(x)-\varphi(y)|}{|x-y|^{sp}}\frac{
    w(x)w(y)}{w(B_{x,y})}\;dx\;dy,$$
    and
$$I_2:=\Lambda\int_{\R^n\setminus\Omega''}\int_{\Omega''}\frac{|u(x)-u(y)|^{p-1}|\varphi(x)-\varphi(y)|}{|x-y|^{sp}}\frac{
    w(x)w(y)}{w(B_{x,y})}\;dx\;dy.$$
\textbf{Estimate of $I_1$:} Using H\"{o}lder's inequality and Lemma \ref{comparingseminorms}, we deduce
\begin{align}\label{I1f}
   I_1&\leq \Big(\int_{\Omega''}\int_{\Omega''}\frac{|u(x)-u(y)|^p|}{|x-y|^{sp}}\frac{
    w(x)w(y)}{w(B_{x,y})}\;dx\;dy\Big)^\frac{p-1}{p}\Big(\int_{\Omega''}\int_{\Omega''}\frac{|\varphi(x)-\varphi(y)|^p}{|x-y|^{sp}}\frac{
    w(x)w(y)}{w(B_{x,y})}\;dx\;dy\Big)^\frac{1}{p}\nonumber\\
    &<\infty.
\end{align}
\textbf{Estimate of $I_2$:} We observe that
\begin{align}\label{B}
I_2&=\Lambda\int_{\R^n\setminus\Omega''}\int_{\Omega''}\frac{|u(x)-u(y)|^{p-1}|\varphi(x)|}{|x-y|^{sp}}\frac{
    w(x)w(y)}{w(B_{x,y})}\;dx\;dy\nonumber\\
    &=\Lambda\int_{\R^n\setminus\Omega''}\int_{\Omega'}\frac{|u(x)-u(y)|^{p-1}|\varphi(x)|}{|x-y|^{sp}}\frac{
    w(x)w(y)}{w(B_{x,y})}\;dx\;dy\nonumber\\
    &\leq \Lambda \int_{\Omega'}\int_{\R^n\setminus \Omega''}\frac{|u(x)-u(y)|^{p-1}|\varphi(x)|}{|x-y|^{sp}}\frac{
    w(x)w(y)}{w(B_{x,y})}\;dy\;dx\nonumber\\
    &\leq 2^{p-1}\Lambda \Big(\int_{\Omega'}\int_{\R^n\setminus \Omega''}\frac{|u(x)|^{p-1}|\varphi(x)|}{|x-y|^{sp}}\frac{
    w(x)w(y)}{w(B_{x,y})}\;dy\;dx\nonumber\\
    &+\int_{\Omega'}\int_{\R^n\setminus \Omega''}\frac{|u(y)|^{p-1}|\varphi(x)|}{|x-y|^{sp}}\frac{
    w(x)w(y)}{w(B_{x,y})}\;dy\;dx\Big)\nonumber\\
    &\leq 2^{p-1}\Lambda \Big(\int_{\Omega'}|u(x)|^{p-1}|\varphi(x)|\Big(\int_{\R^n\setminus B_r(x)}\frac{1}{|x-y|^{sp}}\frac{
    w(y)}{w(B_{x,y})}\;dy\Big)w(x)\;dx\nonumber\\
    &+\sum_{i=1}^N \int_{B_r(x_i)}\Big(\int_{\R^n\setminus B_{2r}(x_i)}\frac{|u(y)|^{p-1}}{|x-y|^{sp}}\frac{
    w(y)}{w(B_{x,y})}\;dy\Big)|\varphi(x)|w(x)\;dx\Big).
\end{align}
For $x\in B_r(x_i),\;y\in \R^n\setminus {B_{2r}(x_i)}$, we have $B_{x,y} \cup B_{x_i,y}\subset B_{2|x_i-y|}(x_i)$. Thus, using Lemma \ref{comlem}, we get
\begin{align}\label{B0}
    \frac{w(B_{x_i,y})}{w(B_{x,y})}= \frac{w(B_{x_i,y})}{w(B_{2|x_i-y|}(x_i))} \frac{w(B_{2|x_i-y|}(x_i))}{w(B_{x,y})}&\leq C(p,[w]_p)\Big(\frac{|B_{x_i,y}|}{|B_{2|x_i-y|}(x_i)|}\Big)^\sigma\Big(\frac{|B_{2|x_i-y|}(x_i)|}{|B_{x,y}|}\Big)^p\nonumber\\
    &\leq C\Big(\frac{|x_i-y|}{|x-y|}\Big)^{np}\nonumber\\
    &\leq C\Big(1+\frac{|x_i-x|}{|x-y|}\Big)^{np}\leq C,
\end{align}
for some constant $C=C(n,p,[w]_p,i)>0$. Applying Lemma \ref{intlem} and (\ref{B0}), estimate (\ref{B}) yields 
\begin{align}\label{I2f}
    I_2&\leq C\Big(r^{-sp}\int_{\Omega'}|u(x)|^{p-1}|\varphi(x)|w(x)\;dx\nonumber\\
    &+\sum_{i=1}^N \int_{B_r(x_i)}\Big(\int_{\R^n\setminus B_{2r}(x_i)}\frac{|u(y)|^{p-1}}{|x_i-y|^{sp}}\frac{
    w(y)}{w(B_{x_i,y})}\;dy\Big)|\varphi(x)|w(x)\;dx\Big)\nonumber\\
    &\leq C\Big(r^{-sp}\int_{\Omega'}|u(x)|^{p-1}|\varphi(x)|w(x)\;dx\nonumber\\
    &+\sum_{i=1}^N r^{-p}\Tail^{p-1}(u;x_i,2r)\int_{B_r(x_i)}|\varphi(x)|w(x)\;dx\Big)\nonumber\\
    &\leq Cr^{-p}\Big(\int_{\Omega'}|u(x)|^{p-1}|\varphi(x)|w(x)\;dx\nonumber\\
    &+\Big(\int_{\Omega''}|\varphi(x)|w(x)\;dx\Big)\sum_{i=1}^N \Tail^{p-1}(u;x_i,2r)\Big)\nonumber\\
    &\leq Cr^{-p}\Big(\|\phi\|_{L^p(\Omega',w)}\|u\|_{L^p(\Omega',w)}^{p-1}
    +\|\varphi(x)\|_{L^1(\Omega'',w)}\sum_{i=1}^N \Tail^{p-1}(u;x_i,2r)\Big)<\infty,
\end{align}
for some constant $C=C(n,p,s,[w]_p,i)>0$, where we have used Lemma \ref{tailf}. Combining \eqref{I1f} and \eqref{I2f} into \eqref{If}, the result follows.
\end{proof}

\subsubsection{Known results}
For the following result, see \cite[Lemma 15.5 and Lemma 15.8]{Heinonen}.
\begin{Lemma}\label{comlem}
Let $1<p<\infty$ and $w\in A_p$. Suppose that $D$ is a measurable subset of a ball $B\subset\mathbb{R}^n$, then
\begin{equation*}
\frac{1}{[w]_{p}}
\left(\frac{|D|}{|B|}\right)^p
\le
\frac{w(D)}{w(B)}
\le
C
\left(\frac{|D|}{|B|}\right)^\sigma,
\end{equation*}
where the constants $C>0$ and $\sigma>0$ depend only on $n$, $p$, and $[w]_{p}$.
\end{Lemma}

As a corollary of Lemma \ref{comlem}, we have the following doubling property for $A_p$ weights.
\begin{Lemma}\label{doubling}
Let $1<p<\infty$ and $w\in A_p$. Then there exists a constant $c_w=c_w(n,p,[w]_p)>0$, such that
\begin{equation*}
w(2B)\le c_w\,w(B)
\end{equation*}
for every ball $B\subset\mathbb{R}^n$.
\end{Lemma}

The following result follows from \cite[Lemma 2.4]{Ok24}.
\begin{Lemma}\label{intlem}
For $B_r(x_0)\subset \R^ n$ and every $t\in\mathbb{R}$ such that $|t|\leq t_0$, the following hold:
\begin{itemize}
    \item [(i)] 
    $$
    \int_{B_r(x_0)}|x-x_0|^t\frac{w(x)}{w(B_{x,x_0})}\;dx~\backsimeq\frac{r^t}{t}
    ,\text{ for }t>0,
    $$

    \item [(ii)]
    $$
    \int_{\R^n\setminus B_r(x_0)}|x-x_0|^t\frac{w(x)}{w(B_{x,x_0})}\;dx~\backsimeq\frac{r^t}{|t|},\text{ for }t<0,
    $$
\end{itemize}
where the hidden constants depend only on $n,\;p,\;t_0$, and $[w]_p$.
\end{Lemma}

The following embedding theorem follows from \cite[Theorem 1.2]{Fabes82}.
\begin{Theorem}[Weighted embedding Theorem]\label{WIT}
    Let $B_R:=B_R(x_0)\subset\R^n$ be a ball with $n\geq 2$. Suppose $1<p<\infty$ and $w\in A_p$. Then there exist positive constants $c$ and $\eta$, depending only on $n,\;p,\;[w]_p$, such that
    $$\bigg(\fint_{B_R} |u|^{kp}w\;dx\bigg)^\frac{1}{kp}\leq cR\bigg(\fint_{B_R}|\nabla u|^pw\;dx\bigg)^\frac{1}{p},$$
    for every $u \in C^\infty_0(B_R)$ and $1\leq k\leq \frac{n}{n-1}+\eta$.
\end{Theorem}
\begin{Remark}\label{exprmk}
We denote 
    \begin{align}\label{exponent}
        \kappa=\frac{n}{n-1}+\eta.
    \end{align}
    We also emphasize that when $w\equiv 1$, by the classical Sobolev embedding, one may choose $\kappa=\frac{n}{n-p}$ for $1<p<n$.
\end{Remark}
The following theorem can be found in \cite[Theorem 1.5]{Fabes82}.
\begin{Theorem}[Weighted Sobolev-Poincar\'{e} Inequality]\label{SPI}
    Let $B_R:=B_R(x_0)\subset\R^n$ be a ball with $n\geq 2$. There exists positive constants $c$ and $\theta$, depending on $n,p,[w]_p$, such that, for all Lipschitz function $u$ on $\overline{B_R}$ and for every $1\leq k\leq \frac{n}{n-1}+\theta$,
    $$\Big(\fint_{B_R}|u-u_{B_R}|^{kp}w\;dx\Big)^\frac{1}{kp}\leq cR\Big(\fint_{B_R}|\nabla u|^{p}w\;dx\Big)^\frac{1}{p}.$$
\end{Theorem}

The following iteration lemma can be found in \cite[Lemma 7.1]{Enrico}. 
\begin{Lemma}\label{seqconv}
    Let $\{A_j\}_{j=0}^{\infty}$ be a sequence of positive real numbers such that $$A_{j+1}\leq ab^jA_j^{1+\beta},$$ for some $a>0,\;\beta>0$ and $b>1$. If $A_0\leq a^{-\frac{1}{\beta}}b^{-\frac{1}{\beta^2}}$, then $A_j\to 0$ as $j\to\infty$.
\end{Lemma}
The following result is taken from \cite[Lemma 1.1]{Enrico}.
\begin{Lemma}\label{iterativelem}
    Let $t_0\geq 0$ and $h:[t_0, t_1]\to [0,\infty)$ be a nonnegative bounded function such that for every $t_0\leq t<s\leq t_1$, $$h(t)\leq \theta h(s)+\frac{A}{(s-t)^\tau}+B,$$
    for some constants $\theta\in (0,1), \tau>0,\;A>0,\;B>0$. Then there exists a constant $C=C(\theta,\tau)>0$, such that $$h(\rho)\leq C\Big(\frac{A}{(R-\rho)^\tau}+B\Big),$$
    for every $t_0\leq \rho<R\leq t_1$.
\end{Lemma}
We begin by recalling the following Krylov--Safonov type \cite{KSc} covering lemma for a doubling measure $w$. We use the version which can be proved similarly as in \cite[Lemma 7.2]{JN}.
\begin{Lemma}\label{covering}
    Let $w$ be a doubling measure. Let $E\subset B_r(x_0)$ be a $w$-measurable set and $\delta\in(0,1)$. Define $$E_\delta:=\cup_{0<\rho\leq\frac{2r}{3}}\{B_r(x_0)\cap B_{3\rho}(x): x\in B_r(x_0), w(E\cap B_{3\rho}(x))>\delta w(B_\rho(x))\}.$$ Then, either $E_\delta=B_r(x_0)$ or $w(E_\delta)\geq \frac{w(E)}{\delta c(w)}$, where $c(w)$ is the doubling constant for the measure $w$.
\end{Lemma}

\section{Preliminary results}
\subsection{Energy and tail estimates}
In this subsection, we establish several energy estimates and a tail estimate. We first prove that positive truncations of weak subsolutions are again weak subsolutions. This fact is a key ingredient in the proof of the local Hölder continuity result in Theorem \ref{cty}.
\begin{Lemma}\label{Subsolution}
Let $u$ be a weak subsolution of \eqref{ME}, then $u_+$ is a weak subsolution of 
    \begin{align}\label{SS}
     \mathcal{M}_\alpha\,u-|g||u|^{p-2}u=|f|\text{ in }\Omega.
\end{align}
\end{Lemma}
\begin{proof}
     We define $u_m(x):=\min\{m(u(x))_+,1\}$, which converges to $\chi_{\mathrm{supp}(u_+)}$ pointwise almost everywhere in $\Omega$. Let $\Omega'\Subset\Omega$ and $\varphi\in W^{1,p}_{0}(\Omega',w)$ be a nonnegative function. Since $u$ is a weak subsolution of \eqref{ME}, incorporating $\psi=\varphi u_m$ as a test function in \eqref{wksoleqn} and using the property \eqref{F}, we get
\begin{align}\label{S11}
    I_1+\alpha\,I_2+I_3\leq I_4,
\end{align}
where $$I_1:=\int_{\Omega'}\mathcal{A}(x,\nabla u)\cdot\nabla(\varphi u_m)\;dx,$$
$$I_2:=\int_{\R^n}\int_{\R^n}\frac{|u(x)-u(y)|^{p-2}(u(x)-u(y))(u_m(x)\varphi(x)-u_m(y)\varphi(y))}{|x-y|^{sp}} K(x,y)\;dx\;dy,$$
$$I_3:=-\int_{\Omega'} |g||u|^{p-1}u_m\varphi \;dx\quad\mbox{ and } \quad I_4:=\int_{\Omega'}|f|u_m\varphi\;dx.$$
Using the definition of $u_m$ and the properties of $\mathcal{A}$, we deduce that 
\begin{align}\label{S12}
    I_1&\geq \int_{\Omega'}u_m \mathcal{A}(x,\nabla u) \cdot\nabla\varphi\;dx+\int_{\Omega'}\varphi \mathcal{A}(x,\nabla u) \cdot\nabla u_m\;dx\nonumber\\
    &{\geq }\int_{\Omega'}u_m \mathcal{A}(x,\nabla u)  \cdot\nabla\varphi\;dx+mC_1\int_{\{x\in\Omega':u_+(x)\leq \frac{1}{m}\}}\varphi|\nabla u_+|^{p}~ w\;dx\nonumber\\
    &\geq \int_{\Omega'}u_m\mathcal{A}(x,\nabla u)\cdot\nabla\varphi\;dx.
\end{align}
Along the lines of proof of \cite[estimate (2.11)-(2.12), page 5398]{GK} and using the Lebesgue's dominated convergence theorem, we deduce that 
\begin{align}
    \lim_{m\to\infty}I_2\geq \int_{\R^n}\int_{\R^n}\frac{|u_+(x)-u_+(y)|^{p-2}(u_+(x)-u_+(y))(\varphi(x)-\varphi(y))}{|x-y|^{sp}} K(x,y)\;dx\;dy.
\end{align}
Using the fact $q>\kappa'$, we obtain that $g|u|^{p-1}\varphi\in L^1(\Omega')$ and hence by the Lebesgue's dominated convergence theorem, we get 
\begin{align}
    \int_{\Omega'} |g||u|^{p-1}u_m\varphi \;dx\to \int_{\Omega'} |g| u_+^{p-1}\varphi \;dx \text{ as } m\to\infty,
\end{align}
and 
\begin{align}\label{S13}
    \int_{\Omega'} |f|u_m\varphi\;dx\to \int_{\Omega'} |f|\chi_{\mathrm{supp}(u_+)}\varphi\;dx\text{ as }m\to\infty.
\end{align}
Utilizing (\ref{S12})-(\ref{S13}) in (\ref{S11}) and taking $m\to\infty$, we deduce that
\begin{align}
  &\int_{\Omega'} \mathcal{A}(x,\nabla u_+)\cdot\nabla\varphi\;dx \nonumber\\
  &+\alpha\,\int_{\R^n}\int_{\R^n}\frac{|u_+(x)-u_+(y)|^{p-2}(u_+(x)-u_+(y))(\varphi(x)-\varphi(y))}{|x-y|^{sp}} K(x,y)\;dx\;dy\nonumber\\
  &{-}\int_{\Omega'} |g|u_+^{p-1}\varphi \;dx\leq \int_{\Omega'}|f|\varphi\;dx. 
\end{align}
Consequently, $u_+$ is a weak subsolution of (\ref{SS}).
\end{proof}

\begin{Remark}\label{rmksub}
If $u$ is a weak supersolution of \eqref{ME}, then we observe that $v=-u$ is a weak subsolution of
$$
\mathcal{M}_\alpha v=\mathcal{G}(x,v)\text{ in }\Omega,
$$
where $\mathcal{G}(x,v):=-\mathcal{F}(x,-v)$ with $\mathcal{G}$ satisfying \eqref{F}. Then by Lemma \ref{Subsolution}, it follows that $v_+=u_-$ is also a weak subsolution of \eqref{SS}.
\end{Remark}

The following energy estimate for weak subsolutions will be used repeatedly in the proofs of our main results.

\begin{Lemma}[Energy estimate for subsolutions]\label{Energy}
Suppose $B_r(x_0)\Subset \Omega$ and $\psi\in C^\infty_c(B_r(x_0))$ is a nonnegative function. Then there exists a constant $C=C(p,\Lambda,\alpha,C_1,C_2)>0$ such that
\begin{equation}\label{subsupeng}
\begin{split}
    \fint_{B_r(x_0)} \psi^p|\nabla v_{\pm}|^p\;w~dx&\leq C\Bigg\{ \fint_{B_r(x_0)} v_{\pm}^p|\nabla \psi|^p\;w\;dx\nonumber\\
    &+\fint_{B_r(x_0)}\int_{B_r(x_0)}\Big(\frac{|\psi(x)-\psi(y)|\max\{v_{\pm}(x),v_{\pm}(y)\}}{|x-y|^{s}}\Big)^p\frac{w(x)w(y)}{w(B_{x,y})}\;dx\;dy\nonumber\\
    &+\Big(\sup_{y\in \supp \,\psi}\int_{\R^n\setminus B_r(x_0)}\frac{(v(x))_{\pm}^{p-1}}{|x-y|^{sp}}\frac{w(x)}{w(B_{x,y})}\;dx\Big)\fint_{B_r(x_0)}v_{\pm}\psi^p w\;dx\Bigg\}\\
    &\qquad+\fint_{B_r(x_0)}|g||u|^{p-1}v_{\pm}\psi^p\;dx+\fint_{B_r(x_0)}|f|v_{\pm}\psi^p\;dx,
\end{split}
\end{equation}
holds for $v_{+}=(u-k)_{+}$ if $u$ is a weak subsolution and for $v_-=(u-k)_-$ if $u$ is a weak supersolution of \eqref{ME} with any $k\in\R$.
\end{Lemma}
\begin{proof}
First, we prove the result for weak subsolutions. To this end, let $u$ be a weak subsolution of \eqref{ME}, then incorporating $\varphi=v_+\psi^p$ in \eqref{wksoleqn}, we get
    \begin{equation}\label{eng1}
        I_1+\alpha\,I_2\leq I_3,
    \end{equation}
    where $$I_1:=\int_{\Omega}\mathcal{A}(x,\nabla u)\cdot\nabla\varphi\;dx,$$
    $$I_2:=\int_{\R^n}\int_{\R^n}\frac{|u(x)-u(y)|^{p-2}(u(x)-u(y))(\varphi(x)-\varphi(y))}{|x-y|^{sp}}K(x,y)\;dx\;dy,$$
    \text{ and }
    $$I_3:=\int_{\Omega}\mathcal{F}(x,u)\varphi\;dx.$$
Proceeding along the lines of proofs of \cite[Proposition 3.1, page 14]{Verena21}, we get
\begin{equation}\label{eng1I1}
\begin{split}
I_1\geq C\int_{B_r(x_0)}\psi^p|\nabla v_+|^pw\;dx-C\int_{B_r(x_0)}v_{+}^p|\nabla \psi|^pw\;dx,
\end{split}
\end{equation}
for some constant $C=C(p,C_1,C_2)>0$. Along the lines of the proof of \cite[{Theorem 1.4, estimates (3.2) and (3.4) on pages 1286-1287}]{DKPahp}, we obtain
and 
\begin{equation}\label{eng1I2}
\begin{split}
    I_2&\geq \frac{1}{2}\int_{B_r(x_0)}\int_{B_r(x_0)}\frac{|v_+(x)-v_+(y)|^p\max\{\psi(x),\psi(y)\}^p}{|x-y|^{sp}}K(x,y)\;dx\;dy\\
&-C\int_{B_r(x_0)}\int_{B_r(x_0)}\frac{|\psi(x)-\psi(y)|^p\max\{v_+(x),v_+(y)\}^p}{|x-y|^{sp}}K(x,y)\;dx\;dy\\
&-{C\Big(\sup_{y\in\mathrm{supp}\,\psi}\int_{\R^n\setminus B_r(x_0)}\frac{|v_+(x)|^{p-1}}{|x-y|^{sp}}\frac{w(x)}{w(B_{x,y})}\;dx\Big)\int_{B_r(x_0)}v_+(y)\psi(y)^p w(y)\;dy},
\end{split}
\end{equation}
for some constant $C=C(p,\Lambda,\alpha)>0$. Combining the above two estimates \eqref{eng1I1} and \eqref{eng1I2} in \eqref{eng1} together with the fact $0\leq K(x,y)\leq \frac{\Lambda w(x)w(y)}{w(B_{x,y})}$ {and the hypothesis \eqref{F}}, we obtain the desired estimate \eqref{subsupeng}.

For weak supersolutions, the result holds by observing the fact that, if $u$ is a weak supersolution of \eqref{ME}, then $v=-u$ is a weak subsolution of 
\begin{equation*}\label{-u}
\mathcal{M}_\alpha\,v=\mathcal{G}(x,v)\text{ in }\Omega,
\end{equation*}
where $\mathcal{G}(x,v):=-\mathcal{F}(x,-v)$, with $\mathcal{G}$ satisfying \eqref{F}. 
\end{proof}
The following energy estimate for weak supersolutions is a key ingredient in the proof of the reverse H\"older inequality, which subsequently leads to the weak Harnack inequality established in Theorem \ref{wkHarnack}.
\begin{Lemma}[Energy estimate for supersolutions]\label{EWH}
    Suppose $u$ is a weak supersolution of \eqref{ME} such that $u\geq 0$ in $B_R(x_0)\Subset\Omega$. Let $0<r\leq \frac{3R}{4}$ and $\psi\in C^\infty_c(B_r(x_0))$ {be a nonnegative function}. Then for any $1<\eta<p$ and any $d>0$, there exists a constant $C=C(n,p,\Lambda,\alpha,C_1,C_2,[w]_p)>0$ such that
    \begin{equation*}
    \begin{split}
        &\int_{B_r(x_0)}\psi^p|\nabla v|^p~w\;dx\leq C\Bigg[\frac{(p-\eta)^p}{(\eta-1)^{p}}\int_{B_r(x_0)}v^p|\nabla\psi|^pw\;dx\\
        &+ \frac{(p-\eta)^p}{(\eta-1)^p}\int_{B_r(x_0)}\int_{B_r(x_0)} {\frac{|\psi(x)-\psi(y)|^p}{|x-y|^{sp}}\max\{v(x),v(y)\}^pK(x,y)\;dx\;dy}\\
        &+\frac{(p-\eta)^p}{(\eta-1)}\Big\{\Big(\sup_{x\in\supp\,\psi}\int_{\R^n\setminus B_r(x_0)}\frac{1}{|x-y|^{sp}}\frac{w(y)}{w(B_{x,y})}\;dy+d^{1-p}R^{-p}\Tail^{p-1}(u_-;x_0,R)\Big)\\
        &\times
    \int_{B_r(x_0)}\psi^pv^pw\;dx+\int_{B_r(x_0)}|g|u^{p-1}\psi^p(u+d)^{1-\eta}\;dx+\int_{B_r(x_0)}|f|\psi^p(u+d)^{1-\eta}\;dx\Big\}\Bigg]
    \end{split}
    \end{equation*}
    with $v=(u+d)^\frac{p-\eta}{p}$. Here, $\Tail$ is given in \eqref{tail}.
\end{Lemma}
\begin{proof}
    Since $u$ is a weak supersolution of (\ref{ME}), choosing $\varphi=\psi^p (u+d)^{1-\eta}$ in \eqref{wksoleqn}, we arrive at
    \begin{equation}\label{I123}
        I_1+\alpha\,I_2\geq I_3,
    \end{equation}
    where $$I_1:=\int_{\Omega}\mathcal{A}(x,\nabla u)\cdot\nabla\varphi\;dx,$$
    $$I_2:=\int_{\R^n}\int_{\R^n}\frac{|u(x)-u(y)|^{p-2}(u(x)-u(y))(\varphi(x)-\varphi(y))}{|x-y|^{sp}}K(x,y)\;dx\;dy,$$
    \text{ and }
    $$I_3:=\int_{\Omega}\mathcal{F}(x,u)\varphi\;dx.$$
    Following the proof of \cite[estimate {(3.10)} on page 5401]{GK}, we obtain
    \begin{align}\label{WW1}
        I_1\leq -\bigg(\frac{\eta-1}{2}\bigg)\bigg(\frac{p}{p-\eta}\bigg)^p\int_{B_r(x_0)}|\nabla v|^p\psi^pw\;dx+\frac{C}{(\eta-1)^{p-1}}\int_{B_r(x_0)}|\nabla\psi|^pv^pw\;dx,
    \end{align}
    for some constant $C=C(p,{C_1,C_2)}>0$. 
  We divide $I_2$ into two parts as follows:
  \begin{align}
      I_2=I_{21}+I_{22},
  \end{align}
    with $$I_{21}:=\int_{B_r(x_0)}\int_{B_r(x_0)}\frac{|u(x)-u(y)|^{p-2}(u(x)-u(y))(\varphi(x)-\varphi(y))}{|x-y|^{sp}}K(x,y)\;dx\;dy$$
    and $$I_{22}:=2\int_{\R^n\setminus B_r(x_0)}\int_{B_r(x_0)}\frac{|u(x)-u(y)|^{p-2}(u(x)-u(y))(\varphi(x)-\varphi(y))}{|x-y|^{sp}}K(x,y)\;dx\;dy.$$
    Proceeding as in the proof of \cite[Lemma 5.1, {pages 1830–1833}]{DKPhar}, we obtain
    \begin{align}\label{WW2}
        I_{21}&\leq -{C(p,\eta,C_1,C_2)}\int_{B_r(x_0)}\int_{B_r(x_0)}\frac{|v(x)-v(y)|^p}{|x-y|^{sp}}|\psi(y)|^pK(x,y)\;dx\;dy\nonumber\\
    &+\frac{{C(p,C_1,C_2)}}{(\eta-1)^{p-1}}\int_{B_r(x_0)}\int_{B_r(x_0)}\frac{|\psi(x)-\psi(y)|^p}{|x-y|^{sp}}\max\{v(x),v(y)\}^pK(x,y)\;dx\;dy.
    \end{align}
    Moreover, proceeding as in the proof of \cite[Lemma 5.1, {page 1830}]{DKPhar}, we obtain
\begin{align}\label{WW3}
    I_{22}\leq C(p,\Lambda)\Bigg(&\sup_{x\in\supp\,\psi}\int_{\R^n\setminus B_r(x_0)}\frac{1}{|x-y|^{sp}}\frac{w(y)}{w(B_{x,y})}\;dy\nonumber\\
    &+d^{1-p}\int_{\R^n\setminus B_r(x_0)}\frac{{(u(y))_-^{p-1}}}{|x-y|^{sp}}\frac{w(y)}{w(B_{x,y})}\;dy\Bigg)\int_{B_r(x_0)}\psi^pv^pw\;dx.
\end{align}
In the last estimate above, we have also used the fact $K(x,y)\leq \frac{\Lambda w(x)w(y)}{w(B_{x,y})}$. For $0<r\leq \frac{3R}{4}$ and $y\in \R^n\setminus B_R(x_0)$, we observe that
\begin{align}
    \frac{|y-x_0|}{|x-y|}\leq 1+\frac{|x-x_0|}{|x-y|}\leq 1+\frac{r}{R-r}\leq 4.
\end{align}
Furthermore, since $B_{x,y}\cup B_{x_0,y}\subset B_{2|y-x_0|}(x_0)$, using Lemma \ref{comlem}, we have
\begin{align}
    \frac{w(B_{x_0,y})}{w(B_{x,y})}&=\frac{w(B_{x_0,y})}{w(B_{2|y-x_0|}(x_0))}\frac{w(B_{2|y-x_0|}(x_0))}{w(B_{x,y})}\nonumber\\
    &\leq C\Big(\frac{|B_{x_0,y}|}{|B_{2|y-x_0|}(x_0)|}\Big)^\sigma\Big(\frac{|B_{2|y-x_0|}(x_0)|}{|B_{x,y}|}\Big)^p\leq C\Big(\frac{|y-x_0|}{|x-y|}\Big)^{np}\leq C,
\end{align}
for some constant $C=C(n,p,[w]_p)>0$.
Therefore, we deduce that 
\begin{align}\label{WW4}
   {\int_{\R^n\setminus B_r(x_0)}\frac{(u(y))_-^{p-1}}{|x-y|^{sp}}\frac{w(y)}{w(B_{x,y})}\;dy}&\leq \int_{\R^n\setminus B_R(x_0)}\frac{(u(y))_-^{p-1}}{|x_0-y|^{sp}}\frac{w(y)}{w(B_{x_0,y})}\;dy\nonumber\\
   &=R^{-p}\Tail^{p-1}(u_-;x_0,R). 
\end{align}
Combining the estimates \eqref{WW1}, \eqref{WW2}, \eqref{WW3}, \eqref{WW4} in \eqref{I123} alongside the fact $0\leq K(x,y)\leq \frac{\Lambda w(x)w(y)}{w(B_{x,y})}$ and the hypothesis \eqref{F}, we obtain the desired estimate.
\end{proof}

The following logarithmic estimate for weak supersolutions is one of the key ingredients in the proof of the expansion of positivity property established in Lemma \ref{expansionofpositivity}.
\begin{Lemma}[Logarithmic estimate for supersolutions]\label{logestlem}
    Let $d>0$ and $u$ be a weak supersolution to (\ref{ME}) such that $u\geq 0$ in $B_R(x_0)\Subset \Omega$. Then, there exists a constant $C=C(n,p,s,\Lambda,\alpha,C_1,C_2,[w]_p,M)>0$ such that 
    \begin{align}\label{logestimate1}
    &\fint_{B_r(x_0)}|\nabla \log(u+d)|^pw\;dx\nonumber\\
        &\leq \frac{C}{r^p}\Big\{1+d^{1-p}(\frac{r}{R})^p\Tail^{p-1}(u_-;x_0,R)+d^{1-p}r^p\Big(\fint_{B_{2r}(x_0)}\Big(\frac{|f|}{w}\Big)^qw\;dx\Big)^\frac{1}{q}\Big\},
    \end{align}
    for every $0<r\leq 1$ with $r<\frac{R}{2}$.  Furthermore, if $$d\geq \Big(\frac{r}{R}\Big)^{p'}\Tail(u_-;x_0,R)+r^{p'}\Big(\fint_{B_{2r}(x_0)}\Big(\frac{|f|}{w}\Big)^qw\;dx\Big)^\frac{1}{q(p-1)},$$ then 
    \begin{align}\label{logestimate2}
        \fint_{B_r(x_0)}|\nabla \log(u+d)|^pw\;dx\leq \frac{C}{r^p}.
    \end{align}
    Here $M$ and $\Tail$ are given in \eqref{conditiong} and \eqref{tail}, respectively.
\end{Lemma}

\begin{proof}
We aim to prove only the estimate (\ref{logestimate1}), since the estimate \eqref{logestimate2} follows from \eqref{logestimate1}. To this end, let $\psi\in C^\infty_c(B_{\frac{3r}{2}}(x_0))$ be a nonnegative function such that $0\leq \psi \leq 1$ in $B_{\frac{3r}{2}}(x_0)$, $\psi\equiv1$ in $B_r(x_0)$ and $|\nabla\psi|\leq \frac{4}{r}$ in $B_{\frac{3r}{2}}(x_0)$. For $d>0$, incorporating $\varphi=\psi^p(u+d)^{1-p}$ in (\ref{wksoleqn}), we get
\begin{align}\label{weaksuper}
   0\leq  I_1+\alpha\,I_2-I_3,
\end{align}
where $$I_1:=\int_{B_{2r}(x_0)}\mathcal{A}(x,\nabla u)\cdot\nabla\varphi\;dx,$$
$$I_2:=\int_{\R^n}\int_{\R^n}\frac{|u(x)-u(y)|^{p-2}(u(x)-u(y))(\varphi(x)-\varphi(y))}{|x-y|^{sp}}K(x,y)\;dx\;dy,$$
and
$$I_3:=\int_{B_{2r}(x_0)}\mathcal{F}(x,u)\varphi\;dx.$$

Proceeding as in the proof of \cite[pages 717-718, Lemma 3.4]{KussiJuha} and using the properties of $\psi$, we arrive at
\begin{equation}\label{I1log}
\begin{split}
    I_1
    &\leq -C\int_{B_{r}(x_0)}|\nabla \log(u+d)|^p~w\;dx+\frac{C}{r^p}w(B_r(x_0))
    \end{split}
\end{equation}
for some constant $C=C(n,p,\Lambda,C_1,C_2,[w]_p)>0$. Along the lines of the proof of the estimates of $I_1$ and $I_2$ in {\cite[Proposition 3.10, pages 28-31]{Ok24}}, we deduce
\begin{equation}\label{I2log}
\begin{split}
    I_2&\leq \frac{C}{r^{sp}}w(B_r(x_0))
    +Cd^{1-p}w(B_r(x_0))R^{-p}\Tail^{p-1}(u_-;x_0,R)
\end{split}
\end{equation}
for some positive constant $C=C(n,p,s,\Lambda,[w]_p)$.

Using the hypothesis \eqref{F} and Lemma \ref{doubling} along with H\"{o}lder's inequality, we obtain there exists a constant $C=C(n,{p},[w]_p)>0$ such that
\begin{equation}\label{I3log}
\begin{split}
    |I_3|&\leq \int_{B_{2r}(x_0)}|g|u^{p-1}\psi^p(u+d)^{1-p}\;dx+\int_{B_{2r}(x_0)}|f|\psi^p(u+d)^{1-p}\;dx\\   
    &\leq \int_{B_{2r}(x_0)}\frac{|g|}{w}\,w\;dx+d^{1-p}\int_{B_{2r}(x_0)}\frac{|f|}{w}\,w\;dx\\
    &\leq w(B_{2r}(x_0))\Big(\fint_{B_{2r}(x_0)}\Big(\frac{|g|}{w}\Big)^qw\;dx\Big)^\frac{1}{q}+d^{1-p}w(B_{2r}(x_0))\Big(\fint_{B_{2r}(x_0)}\Big(\frac{|f|}{w}\Big)^qw\;dx\Big)^\frac{1}{q}\\
    &\leq Cw(B_{r}(x_0))\Big(\fint_{B_{2r}(x_0)}\Big(\frac{|g|}{w}\Big)^q w\;dx\Big)^\frac{1}{q}+Cd^{1-p}w(B_{r}(x_0))\Big(\fint_{B_{2r}(x_0)}\Big(\frac{|f|}{w}\Big)^qw\;dx\Big)^\frac{1}{q}.
\end{split}
\end{equation}
Substituting the estimates of \eqref{I1log}, \eqref{I2log} and \eqref{I3log} above into \eqref{weaksuper} and utilizing \eqref{conditiong} along with $0<r\leq 1$ and $0<s<1$, we arrive at
\begin{align*}
    &\fint_{B_{r}(x_0)}|\nabla \log(u+d)|^p~w\;dx 
    \nonumber\\
    &\leq \frac{C}{r^p}\Big\{1+d^{1-p}\Big(\frac{r}{R}\Big)^p\Tail^{p-1}(u_-;x_0,R)+d^{1-p}r^p\Big(\fint_{B_{2r}(x_0)}\Big(\frac{|f|}{w}\Big)^qw\;dx\Big)^\frac{1}{q}\Big\},
\end{align*}
for some constant $C=C(n,p,s,\Lambda,\alpha,C_1,C_2,[w]_p,M)>0$.
\end{proof}

The following tail estimate holds under the additional hypothesis $\la>0$, which is a key ingredient in the proof of the Harnack inequality in Theorem \ref{Harthm}.
\begin{Lemma}[Tail estimate for supersolutions]\label{Tailestimate}
    Assume, in addition to the standing hypotheses, that the constant $\lambda>0$ in \eqref{kernal}. Suppose that $u$ is a weak supersolution of equation \eqref{ME} such that $u\geq 0$ in $B_R(x_0)\Subset\Omega$. Then there exists a constant $C=C(n,p,s,\la,\Lambda,\alpha,C_1,C_2,[w]_p,\\M)>0$ such that
    \begin{align}
        \Tail(u_+;x_0,r)\leq C\Big\{\sup_{B_r(x_0)}u+\Big(\frac{r}{R}\Big)^{p'}\Tail(u_-;x_0,R)+r^{p'}\Big(\fint_{B_r(x_0)}\Big(\frac{|f|}{w}\Big)^qw\;dx\Big)^\frac{1}{q(p-1)}\Big\},
    \end{align}
    for every $0<r\leq 1$ with $r<R$.
\end{Lemma}
\begin{proof}
    Let $\varphi\in C^\infty_c(B_r(x_0))$ be a nonnegative function such that $0\leq \varphi\leq 1$ in $B_r(x_0)$, $|\nabla\varphi|\leq \frac{8}{r}$ in $B_r(x_0)$ and $\phi\equiv 1$ in $B_\frac{r}{2}(x_0)$. We define $l:=\sup_{B_r(x_0)} u$ and $v:=u-2l$. Incorporating $\varphi=v\,\psi^p$ as a test function in \eqref{wksoleqn}, we obtain
    \begin{equation}\label{supertail}
        I+\alpha\,J
        \leq K+L,
    \end{equation}
    where $$I:=\int_{B_r(x_0)}\mathcal{A}(x,\nabla u)\cdot\nabla \varphi \;dx,$$
    $$J:=\int_{\R^n}\int_{\R^n}\frac{|u(x)-u(y)|^{p-2}(u(x)-u(y))(\varphi(x)-\varphi(y))}{|x-y|^{sp}}K(x,y)\;dx\;dy,$$
$$K:=\int_{B_r(x_0)}|g||u|^{p-1}|\varphi| \;dx,\quad\mbox{ and }\quad L:=\int_{B_r(x_0)}|f||\varphi|\;dx.$$
\textbf{Estimate of $I$:} Using the estimate \cite[estimate (3.6) on page 7 in Lemma 3.4]{GKK}, we have
\begin{align}\label{estItail}
    I\geq -\frac{Cl^p}{r^p}w(B_r(x_0)),
\end{align}
for some constant $C=C(n,p,C_1,C_2,[w]_p)>0$.\\
\textbf{Estimate of $J$:} Taking into account $r\in(0,1]$ and $\la>0$ and Proceeding along the lines of the proof of the estimates
(4.11) and (4.9) in \cite[pages 1827-1828]{DKPhar}, we obtain
\begin{align}\label{estJtail}
    J\geq Clr^{-p}\Tail^{p-1}(u_+;x_0,r)w(B_r(x_0))&-ClR^{-p}\Tail^{p-1}(u_-;x_0,R)w(B_r(x_0))\nonumber\\
    &-Cl^pr^{-p}w(B_r(x_0)),
\end{align}
for some $C=C(n,p,s,{\lambda},\Lambda,[w]_p)>0$.\\
\textbf{Estimates of $K$ and $L$:} Since $u\leq l$ and $v\leq 2l$ in $B_r(x_0)$, we have 
\begin{align}\label{estKtail}
    |K|\leq 2l^p\int_{B_r(x_0)}|g|\;dx\leq 2l^p w(B_r(x_0))\Big(\fint_{B_r(x_0)}\Big(\frac{|g|}{w}\Big)^q w\;dx\Big)^\frac{1}{q},
\end{align}
and
\begin{align}\label{estLtail}
    |L|\leq 2l w(B_r(x_0))\Big(\fint_{B_r(x_0)} \Big(\frac{|f|}{w}\Big)^q w\;dx\Big)^\frac{1}{q}.
\end{align}
Substituting the above estimates \eqref{estItail}, \eqref{estJtail}, \eqref{estKtail} and \eqref{estLtail} into (\ref{supertail}) and using \eqref{conditiong} in the resulting estimate, we obtain
\begin{align}\label{tailf1}
    \Tail^{p-1}(u_+;x_0,r)
    &\leq C\Big\{l^{p-1}+\big(\frac{r}{R}\big)^p\Tail^{p-1}(u_-;x_0,R)+r^p\Big(\fint_{B_r(x_0)} \Big(\frac{|f|}{w}\Big)^q w\;dx\Big)^\frac{1}{q}\Big\},
\end{align}
where $C=C(n,p,s,{\lambda},\Lambda,\alpha,C_1,C_2,[w]_p,M)>0$ is a constant.
    Thus, (\ref{tailf1}) yields
    $$\Tail(u_+;x_0,r)\leq C\Big\{l+\Big(\frac{r}{R}\Big)^{p'}\Tail(u_-;x_0,R)+r^{p'}\Big(\fint_{B_r(x_0)} \Big(\frac{|f|}{w}\Big)^q w\;dx\Big)^\frac{1}{q(p-1)}\Big\},$$
    for some constant $C=C(n,p,s,{\lambda},\Lambda,\alpha,C_1,C_2,[w]_p,M)>0$. This completes the proof.
\end{proof}

\subsection{Expansion of positivity}
In this subsection, our primary objective is to establish the expansion of positivity property for weak supersolutions. This property plays a crucial role in proving the local Hölder continuity result stated in {Theorem \ref{cty}}. Furthermore, it is instrumental in deriving the preliminary version of the weak Harnack inequality presented in Lemma \ref{WHI0}, which subsequently serves as a key step in the proofs of the Harnack and weak Harnack inequalities.
\begin{Lemma}\label{expansionofpositivity}
    Let $u$ be a weak supersolution of \eqref{ME} such that $u\geq 0$ in $B_R(x_0)\Subset\Omega$. Let $k\geq 0$ and assume that there exists a constant $\tau\in(0,1)$ such that 
    \begin{align}\label{density}
        w(\{B_{2r}(x_0):u\geq k\})\geq \tau w(B_{2r}(x_0),
    \end{align}
    for some $0<r\leq 1$ with $r<\frac{R}{4}$. Then there exists a constant $\zeta=\zeta(n,p,s,q,\kappa,\Lambda,\alpha,C_1,C_2,[w]_p,$ $M,\tau)\in (0,\frac{1}{4})$ such that   \begin{align}\label{densityconclusion}
        \inf_{B_r(x_0)}u\geq \zeta k-\Big(\frac{r}{R}\Big)^{p'}\Tail(u_-;x_0,R)-r^{p'}\Big(\fint_{B_{2r}(x_0)}\Big(\frac{|f|}{w}\Big)^qw\;dx\Big)^\frac{1}{q(p-1)},
    \end{align}
    where $M$ and $\Tail$ are given in (\ref{conditiong}) and \eqref{tail}, respectively.
\end{Lemma}
\begin{proof}
We prove the result into the following two steps.\\
\textbf{Step-I:}
We show that for every $\delta\in (0,\frac{1}{4})$ and for every $\epsilon>0$, there exists a constant $C=C(n,p,s,q,\Lambda,\alpha,C_1,C_2,[w]_p,M)>0$ such that 
\begin{align}\label{L0}
    \frac{w(\{B_{2r}(x_0):u\leq 2\delta k-\frac{d}{2}-\epsilon\})}{w(B_{2r}(x_0))}\leq \frac{C}{\tau\log\big(\frac{1}{2\delta}\big)},
\end{align}
where
$$
d:=\Big(\frac{r}{R}\Big)^{p'}\Tail(u_-;x_0,R)+r^{p'}\Big(\fint_{B_{2r}(x_0)}\Big(\frac{|f|}{w}\Big)^qw\;dx\Big)^\frac{1}{q(p-1)}.
$$
To this end, we denote $t_\epsilon=\frac{d}{2}+\epsilon>0$. Due to Lemma \ref{logestlem}, we have
\begin{align}
    \fint_{B_{2r}(x_0)}|\nabla \log(u+t_\epsilon)|^pw\;dx\leq \frac{C}{r^p},
\end{align}
for some constant $C=C(n,p,s,\Lambda,\alpha,C_1,C_2,[w]_p,M)>0$. Now we define $$v:=\min\Big[\log\Big(\frac{1}{2\delta}\Big),\Big\{\log\Big(\frac{k+t_\epsilon}{u(x)+t_\epsilon}\Big)\Big\}_+\Big].$$
Note that, $v=0$ if and only if $u\geq k$. Moreover,
\begin{align}\label{L1}
    \fint_{B_{2r}(x_0)}|\nabla v|^pw\;dx\leq \fint_{B_{2r}(x_0)}|\nabla \log(u+t_\epsilon)|^pw\;dx\leq \frac{C}{r^p},
\end{align}
for some constant $C=C(n,p,s,\Lambda,\alpha,C_1,C_2,[w]_p,M)>0$.
 Using H\"{o}lder's inequality, Theorem \ref{SPI} and (\ref{L1}), we deduce
 \begin{align}\label{LL}
     \fint_{B_{2r}(x_0)}|v-(v)_{B_{2r}(x_0)}|w\;dx\leq \Big(\fint_{B_{2r}(x_0)}|v-(v)_{B_{2r}(x_0)}|^pw\;dx\Big)^\frac{1}{p}&\leq Cr\Big(\fint_{B_{2r}(x_0)}|\nabla v|^pw\;dx\Big)^\frac{1}{p}\nonumber\\
     &\leq C,
 \end{align}
for some constant $C=C(n,p,s,\Lambda,\alpha,C_1,C_2,[w]_p,M)>0$.
Thanks to (\ref{density}), we have 
\begin{align}\label{L2}
    \log\Big(\frac{1}{2\delta}\Big)&=\frac{1}{w(\{B_{2r}(x_0): v=0\})}\int_{\{B_{2r}(x_0): v=0\}}\Big(\log\Big(\frac{1}{2\delta}\Big)-v\Big)w\;dx\nonumber\\
    &\leq \frac{1}{w(\{B_{2r}(x_0): u\geq k\})}\int_{B_{2r}(x_0)}\Big(\log\Big(\frac{1}{2\delta}\Big)-v\Big)w\;dx\nonumber\\
    &\leq \frac{1}{\tau w(B_{2r}(x_0))}\int_{B_{2r}(x_0)}\Big(\log\Big(\frac{1}{2\delta}\Big)-v\Big)w\;dx=\frac{1}{\tau}\Big(\log\Big(\frac{1}{2\delta}\Big)-(v)_{B_{2r}(x_0)}\Big).
\end{align}
{Integrating both sides of \eqref{L2} over $w(\{B_{2r}(x_0):v=\log\big(\frac{1}{2\delta}\big)\})$, and using \eqref{LL}, one has
\begin{align}
    \log\Big(\frac{1}{2\delta}\Big)w\Big(\Big\{B_{2r}(x_0):v=\log\Big(\frac{1}{2\delta}\Big)\Big\}\Big)
    &\leq \frac{1}{\tau}\int_{w\big(\{B_{2r}(x_0):v=\log\big(\frac{1}{2\delta}\big)\big\}\big)}
    \Big(\log\Big(\frac{1}{2\delta}\Big)-(v)_{B_{2r}(x_0)}\Big)w\;dx\nonumber\\
    &=\frac{1}{\tau}
    \int_{w\big(\{B_{2r}(x_0):v=\log\big(\frac{1}{2\delta}\big)\big\}\big)}\Big(v-(v)_{B_{2r}(x_0)}\Big)w\;dx\nonumber\\
    &\leq \frac{1}{\tau}\int_{B_{2r}(x_0)}|v-(v)_{B_{2r}(x_0)}|w\;dx\leq \frac{C}{\tau}w(B_{2r}(x_0)),
\end{align}
which yields
\begin{align}\label{L3}
    \frac{w(\{B_{2r}(x_0):v=\log\big(\frac{1}{2\delta}\big)\})}{w(B_{2r}(x_0))}\leq \frac{C}{\tau \log(\frac{1}{2\delta})},
\end{align}
for some constant $C=C(n,p,s,\Lambda,\alpha,C_1,C_2,[w]_p,M)>0$.
Finally, we observe that $v(x)=\log\big(\frac{1}{2\delta}\big)$ whenever $u\leq 2\delta k-\frac{d}{2}-\epsilon$. Thus, (\ref{L3}) gives (\ref{L0}).\\
\textbf{Step-II:} We aim to show that for every $\epsilon>0$, there exists a constant $\zeta=\zeta(n,p,s,q,\Lambda,\alpha,C_1,\\C_2,[w]_p,M){\in (0,\frac{1}{4})}$ such that
\begin{equation}\label{s2}
\inf_{B_r(x_0)} u\geq \zeta k-d-2\epsilon.
\end{equation}
Once, we prove the estimate \eqref{s2} above, the resulting estimate \eqref{densityconclusion} holds true. To prove the above estimate, without loss of generality, we assume that 
\begin{equation}\label{assumption}
\zeta k>d+2\epsilon,
\end{equation}
since otherwise the estimate \eqref{s2} above holds true due to the assumption $u\geq 0$ in $B_R(x_0)$.
For $j\in\mathbb{N}\cup\{0\}$, we define 
\[
r_j=(1+2^{-j})r,\quad \Tilde{r_j}=\frac{r_j+r_{j+1}}{2}
\]
and 
\[
B_j=B_{r_j}(x_0),\quad \Tilde{B}_j=B_{\Tilde{r}_j}(x_0).
\]
For $\zeta\in(0,\frac{1}{4})$ {satisfying \eqref{assumption}}, we also define
\[
k_j=(1+2^{-j-1})\zeta k,\quad \Tilde{k}_j=\frac{k_j+k_{j+1}}{2}.
\]
We further set
\[
u_j=(k_j-u)_+,\quad \Tilde{u}_j=(\Tilde{k}_j-u)_+.
\]
It is clear that 
\[
r\leq r_{j+1}\leq \Tilde{r}_j\leq r_{j}\leq 2r,\quad d\leq\zeta k\leq k_{j+1}\leq \Tilde{k}_j\leq k_{j}\leq \frac{3\zeta k}{2},
\] and hence $u_{j+1}\leq \Tilde{u}_j\leq u_{j}$. Let $\{\varphi_j\}_{j=0}^{\infty}\in C^\infty_c(\Tilde{B}_j(x_0))$ be a sequence of nonnegative function such that $0\leq \varphi_j\leq 1$ in $B_j$, $|\nabla\varphi_j|\leq \frac{2^{j+1}}{r_j}$ in $B_j$, and $\varphi_j\equiv 1$ in $B_{j+1}$. Now, we define a sequence of positive real numbers $\{A_j\}_{j=0}^{\infty}$ such that $$A_j:=\frac{w(\{B_j:u\leq k_j\})}{w(B_j)}.$$ In order to show that $A_j$ satisfies the assumptions of Lemma \ref{seqconv}, we use Theorem \ref{WIT} and the properties of $\varphi_j$ to estimate the following:
\begin{align}\label{SA1}
    A_{j+1}^\frac{1}{\kappa}(k_j-k_{j+1})^p&= \Big(\frac{1}{w(B_{j+1})}\int_{\{B_{j+1}:u\leq k_{j+1}\}}(k_j-k_{j+1})^{p\kappa}\;w\;dx\Big)^\frac{1}{\kappa}\nonumber\\
    &\leq \Big(\frac{1}{w(B_{j+1})}\int_{\{B_{j+1}:u\leq k_{j+1}\}}(k_j-u)^{p\kappa}\;w\;dx\Big)^\frac{1}{\kappa}\nonumber\\
    &\leq \Big(\fint_{B_{j+1}}u_j^{p\kappa}\;w\;dx\Big)^\frac{1}{\kappa}\nonumber\\
    &\leq C\Big(\fint_{B_{j}}(\varphi_ju_j)^{p\kappa}\;w\;dx\Big)^\frac{1}{\kappa}\nonumber\\
    &\leq Cr_j^p\fint_{B_j}|\nabla (\varphi_ju_j)|^pw\;dx\nonumber\\
    &\leq Cr_j^p\Big(\fint_{B_j}\varphi_j^p|\nabla u_j|^pw\;dx+\fint_{B_j} u_j^p|\nabla \varphi_j|^pw\;dx\Big),
\end{align}
where $C=C(n,p,[w]_p)>0$ is some constant. Since $u$ is a weak supersolution of (\ref{ME}), Lemma \ref{Energy} guaranties the existence  of a constant $C=C(n,p,\Lambda,\alpha,C_1,C_2,[w]_p)>0$ such that 
\begin{align}\label{SA2}
   \fint_{B_j}\varphi_j^p|\nabla u_j|^pw\;dx\leq C(I_1+I_2+I_3)+I_4+I_5, 
\end{align}
where
$$I_1:=\fint_{B_j} u_j^p|\nabla \varphi_j|^p\;w\;dx,$$
$$I_2:=\fint_{B_j}\int_{B_j}\Big(\frac{|\varphi_j(x)-\varphi_j(y)|\max\{u_j(x),u_j(y)\}}{|x-y|^{s}}\Big)^p\frac{ w(x)w(y)}{w(B_{x,y})}\;dx\;dy,$$
$$I_3:=\Bigg(\sup_{y\in \supp\,\varphi_j}\int_{\R^n\setminus B_j}\frac{|u_j(x)|^{p-1}}{|x-y|^{sp}}\frac{w(x)}{w(B_{x,y})}\;dx\Bigg)\fint_{B_j}u_j\varphi_j^p w\;dx,$$
$$I_4:=\fint_{B_j}|f|u_j\varphi_j^p\;dx,\mbox{ and }I_5:=\fint_{B_j}|g||u|^{p-1}u_j\varphi_j^p\;dx.$$
\textbf{Estimate of $I_1$:} Using the properties of $\varphi_j$ and the fact that $u_j\leq k_j$, we get 
\begin{align}\label{SAA3}
    I_1\leq \frac{2^{(j+1)p}k_j^p}{r_j^p}A_j.
\end{align}
\textbf{Estimate of $I_2$:} Taking into account Lemma \ref{intlem} and using the properties of $\varphi_j$, we deduce
\begin{align}\label{SA3}
    I_2&\leq \frac{C2^{jp}}{r_j^p}\fint_{B_j}\int_{B_j}u_j^p(y)|x-y|^{(1-s)p}\frac{ w(x)w(y)}{w(B_{x,y})}\;dx\;dy\leq \frac{C2^{jp}k_j^p}{r_j^{sp}}A_j,
\end{align}
for some constant $C=C(n,p,s,[w]_p)>0$.\\
\textbf{Estimate of $I_3$:}
Using (\ref{T17}) and (\ref{T18}), proceeding as in (\ref{T19}) one has
\begin{align}\label{S1}
    I_3&\leq C2^{j(p+sp+np)}k_j A_j\int_{\R^n\setminus B_j}\frac{u_j(x)^{p-1}}{|x-x_0|^{sp}}\frac{w(x)}{w(B_{x,x_0})}\;dx\nonumber\\
    &\leq C2^{j(p+sp+np)}k_j A_j\int_{\R^n\setminus B_j}\frac{(k_j+(u(x))_{-})^{p-1}}{|x-x_0|^{sp}}\frac{w(x)}{w(B_{x,x_0})}\;dx\nonumber\\
    &\leq C2^{j(p+sp+np)}k_j A_j\Big(\int_{\R^n\setminus B_j}\frac{k_j^{p-1}}{|x-x_0|^{sp}}\frac{w(x)}{w(B_{x,x_0})}\;dx+\int_{\R^n\setminus B_j}\frac{(u(x))_{-}^{p-1}}{|x-x_0|^{sp}}\frac{w(x)}{w(B_{x,x_0})}\;dx\Big),
\end{align}
for some constant $C=C(n,p,s,[w]_p)>0$.
Due to Lemma \ref{intlem} and the facts $u\geq 0$ in $B_R(x_0)$ and $k_j^{p-1}\geq d^{p-1}\geq \big(\frac{r}{R}\big)^p\Tail^{p-1}(u_-;x_0,R)$; the estimate (\ref{S1}) yields
\begin{align}\label{SA4}
    I_3&\leq C2^{j(p+sp+np)}k_jA_j\Big(\frac{k_j^{p-1}}{r_j^{sp}}+\frac{\Tail^{p-1}(u_-;x_0,R)}{R^p}\Big)\nonumber\\
    &\leq C2^{j(p+sp+np)}k_jA_j\Big(\frac{k_j^{p-1}}{r_j^{sp}}+\frac{k_j^{p-1}}{r^p}\Big)\nonumber\\
    &\leq \frac{C2^{j(p+sp+np)}k_jA_j}{r_j^p}\Big(k_j^{p-1}r_j^{p-sp}+k_j^{p-1}\frac{r_j^p}{r^p}\Big)\nonumber\\
    &\leq \frac{C2^{j(p+sp+np)}k_j^p}{r_j^p}A_j,
\end{align}
for some constant $C=C(n,p,s,[w]_p)>0$.\\
\textbf{Estimate of $I_4$:}
Using {the fact that}
$$
k_j^{p-1}\geq d^{p-1}\geq r^p \big(\fint_{B_{2r}(x_0)}\Big(\frac{|f|}{w}\Big)^qw\;dx\big)^\frac{1}{q}
$$ 
and H\"{o}lder's inequality, we obtain the following:
\begin{align}\label{SA5}
    I_4\leq \fint_{B_j}|f|u_j\varphi_j^p\;dx&\leq k_j\fint_{B_j}\frac{|f|}{w}\chi_{\{B_j:u\leq k_j\}}w\;dx\leq k_j\Big(\fint_{B_j}\Big(\frac{|f|}{w}\Big)^qw\;dx\Big)^\frac{1}{q}A_j^{1-\frac{1}{q}}\nonumber\\
    &\leq Ck_j\Big(\fint_{B_{2r}(x_0)}\Big(\frac{|f|}{w}\Big)^qw\;dx\Big)^\frac{1}{q}A_j^{1-\frac{1}{q}}\nonumber\\
    &\leq \frac{Ck_jd^{p-1}}{r^p}A_j^{1-\frac{1}{q}}\leq \frac{Ck_j^p}{r_j^p}A_j^{1-\frac{1}{q}},
\end{align}
for some constant $C=C(n,p,[w]_p)>0$.\\
\textbf{Estimate of $I_5$:} Utilizing (\ref{conditiong}), we have
\begin{align}\label{SA6}
    I_5\leq \fint_{B_j}|g|~|u|^{p-1}u_j\;dx\leq k_j^p \Big(\fint_{B_j}\Big(\frac{|g|}{w}\Big)^qw\;dx\Big)^\frac{1}{q}A_j^{1-\frac{1}{q}}\leq \frac{M k_j^p}{r_j^p}A_j^{1-\frac{1}{q}}.
\end{align}
Incorporating (\ref{SAA3}), (\ref{SA3}), (\ref{SA4}), (\ref{SA5}), and (\ref{SA6}) into (\ref{SA2}), and applying the resulting estimate in (\ref{SA1}) to deduce 
\begin{align}\label{SA7}
    A_{j+1}^\frac{1}{\kappa}(k_j-k_{j+1})^p\leq C2^{j(p+np+sp)}k_j^p\Big(A_j+A_j^{1-\frac{1}{q}}\Big)\leq C2^{j(p+np+sp)}k_j^pA_j^{1-\frac{1}{q}},
\end{align}
for some constant $C=C(n,p,s,\Lambda,\alpha,C_1,C_2,[w]_p,M)>0$. In the last inequality, we have also used the fact that $A_j\leq 1$. Thus, (\ref{SA7}) yields
\begin{align}
    A_{j+1}\leq C2^{j(p+np+sp)\kappa}\Big(\frac{k_j}{k_j-k_{j+1}}\Big)^{p\kappa}A_j^{(1-\frac{1}{q})\kappa}\leq C2^{j(2p+np+sp)\kappa}A_j^{(1-\frac{1}{q})\kappa},
\end{align}
for some constant $C=C(n,p,s,\Lambda,\alpha,C_1,C_2,[w]_p,M)>0$. Therefore, $\{A_j\}_{j=0}^{\infty}$ satisfies the assumption of Lemma \ref{seqconv} with $a=C$, $b=2^{(2p+np+sp)\kappa}>1$ and $\beta=(1-\frac{1}{q})\kappa-1>0$ (as $q>\kappa'$). Moreover, since \eqref{assumption} holds, we obtain
$$
k_0=\frac{3}{2}\zeta k\leq 2\zeta k-\frac{d}{2}-\epsilon
$$ and therefore using \eqref{L0}, we arrive at the estimate
$$A_0\leq \frac{w(\{B_{2r}(x_0):u\leq 2\zeta k-\frac{d}{2}-\epsilon\})}{w(B_{2r}(x_0))}\leq \frac{C}{\tau\log(\frac{1}{2\zeta})},$$
for some constant $C=C(n,p,s,\Lambda,\alpha,C_1,C_2,[w]_p,M)>0$ and for every $\zeta\in(0,\frac{1}{4})$ {satisfying \eqref{assumption}}
Choose $\zeta=\zeta(n,p,s,q,\kappa,\Lambda,\alpha,C_1,C_2,[w]_p,M,\tau)\in(0,\frac{1}{4})$ such that $$\frac{C}{\tau\log(\frac{1}{2\zeta})}\leq a^{-\frac{1}{\beta}}b^{-\frac{1}{\beta^2}}.$$ Hence, by Lemma \ref{seqconv}, we conclude that $A_j\to 0$ as $j\to\infty$. Consequently, $\inf_{B_r(x_0)}u\geq \zeta k$. Thus, the proof is complete.}
\end{proof}

\subsection{De-Giorgi type lemma}
This subsection is mainly devoted to the proof of the semicontinuity result stated in Theorem \ref{lscthm}. To this end, we establish Proposition \ref{lowersemi}, by adapting the argument in the proof of \cite[Theorem 2.1]{Liao}. This yields a semicontinuous representative under a suitable De Giorgi-type result in Lemma \ref{semlem}.

Let $B_\rho(x_0)\Subset\Omega$ and introduce the real numbers $a,L$ and $\mu^-$ satisfying
\begin{equation}\label{seqn}
a\in(0,1),\quad L>0,\quad \mu^{-}\leq\inf_{B_\rho(x_0)}\,u.
\end{equation}
We say that 
$u$ satisfies the property $(\mathcal{D})$, if there exists a constant ${\sigma}\in(0,1)$ depending on $a,\mu^-,L$ and other data (may depend on the partial differential equation and will be made
precise in Lemma \ref{semlem}), but independent of $\rho$, such that if
$$
w(\{B_\rho(x_0):u\leq \mu^-+L\})\leq {\sigma}\, w(B_\rho(x_0)),
$$
then $u\geq \mu^-+a L$ in $B_{{\frac{\rho}{2}}}(x_0)$.
Further, we define
$$\mathcal{F}:=\Bigg\{x\in\Omega:|u(x)|<\infty,\quad \lim_{r\to 0}\fint_{B_r(x)}|u(x)-u(y)|w(y)dy=0\Bigg\}.$$
Using Lebesgue's differentiation theorem, it follows that
$|\mathcal{F}|=|\Omega|$, {(see \cite[Theorem 1.33]{EG}).}
\begin{prop}\label{lowersemi}
Suppose $u$ is a locally integrable and locally essentially bounded below function
    in $\Omega$ satisfying the property $(\mathcal{D})$. Then $u(x)=u_*(x)$ for every $x\in\mathcal{F}$, where $$u_*(x):=\lim_{r\to 0}\inf_{B_r(x)}u.$$ In particular, $u_*$ is a lower semicontinuous representative of $u$ in $\Omega$.
\end{prop}
\begin{proof}
First we observe that, for every $x\in\mathcal{F}$, 
    \begin{align*}
        u(x)=\lim_{r\to 0}\fint_{B_r(x)}u(y)w(y)dy\geq \lim_{r\to 0}\inf_{B_r(x)}u=u_*(x).
    \end{align*}
    To prove the converse, let us assume by contradiction that there exists a point $x_0\in\mathcal{F}$ such that $u(x_0)>u_*(x_0)$. We fix $R>0$ such that $B_R(x_0)\Subset\Omega$ and let $\mu^-,L$ satisfy 
    $\mu^-:=\inf_{B_R(x_0)}\,u\leq u_*(x_0)<\mu^-+L<u(x_0)$. Now after fixing the quantities $u_*(x_0),\mu^-$ and $L$, we choose {$a\in(0,1)$,} such that 
    \begin{align}\label{semi1}
        u_*(x_0)<\mu^-+a L,\text{ i.e. }\frac{u_*(x_0)-\mu^-}{L}<a<1.
    \end{align}
    Now with such $a$ being fixed in \eqref{semi1}, we determine $\nu$ depending on $a,\mu^-,L$ and other given data but independent of $r$ as per the property $(\mathcal{D})$. Next, we claim that there must be some $\rho\in(0,{R})$ such that
{\begin{equation}\label{semd}
    w(\{B_\rho(x_0):u\leq \mu^-+L\})< \sigma w(B_\rho(x_0)).
    \end{equation}}
Otherwise, we would arrive at the estimate
{\begin{align*}
    \int_{B_\rho(x_0)}|u(x_0)-u(y)|w(y)\;dy
    &\geq \int_{\{y\in B_\rho(x_0):u(y)\leq \mu^-+L\}}(u(x_0)-\mu^--L)w(y)\;dy\nonumber\\
    &=(u(x)-\mu^--L)w(\{y\in B_\rho(x_0):u(y)\leq \mu^-+L\})\nonumber\\
    &\geq \sigma (u(x_0)-\mu^--L) w(B_\rho(x_0)),
\end{align*}}
for all $\rho\in(0,R)$, which is a contradiction to the fact that $x_0\in\mathcal{F}$, since $\nu$ does not depend on $\rho$. Now by the property $(\mathcal{D})$, the measure information \eqref{semd} and the choice of $a$ in \eqref{semi1} imply that
$$
u\geq \mu^-+aL>u_*(x_0)\text{ in }B_{{\frac{\rho}{2}}}(x_0),
$$
which contradicts the definition of $u_*(x_0)$. As a result, we must have, $u_*(x)\geq u(x)$ for every $x\in\mathcal{F}$.
\end{proof}

\begin{Lemma}[De-Giorgi type result]\label{semlem}
   {Suppose $\mathcal{F}$ satisfies (\ref{F}) where ${f}/{w},{g}/{w}\in L^q_{\mathrm{loc}}(\Omega,w)$ with $q>\kappa'$, and $f,g$ satisfy condition (\ref{conditiong}).} Let $u$ be a weak supersolution of \eqref{ME} which is bounded below in $\mb{R}^n$. Suppose that $B_r(x_0)\Subset\Omega$ with $r\in(0,1]$, $\mu_1\leq \inf_{B_r(x_0)}u$, $\mu_2\leq \inf_{\R^n}u$, and $\delta\in (0,1),\;L>0$. Then there exists a constant $\sigma=\sigma(n,p,q,s,\Lambda,\alpha,C_1,C_2,$ $\mu_1,\mu_2,[w]_p,L,M,\delta)\in(0,1)$, such that if $$w(\{B_r(x_0):u\leq \mu_1+L\})\leq \sigma w(B_r(x_0)),$$
then $u\geq \mu_1+\delta L$ in $B_{\frac{r}{2}}(x_0)$.
\end{Lemma}
\begin{proof}
    For $i\in\mathbb{N}\cup\{0\}$, we denote $r_i:=(1+2^{-i})\frac{r}{2}$, $\Tilde{r}_i:=\frac{r_i+r_{i+1}}{2}$, $B_{i}:=B_{r_i}(x_0)$, and $\Tilde{B}_i:=B_{\Tilde{r}_i}(x_0)$. We further define $k_i:=\mu_1+L+(1-2^{-i})(\delta-1)L$, $\Tilde{k}_i:=\frac{k_i+k_{i+1}}{2}$. Let $u_i:=(k_i-u)_+$ and $\Tilde{u}_i:=(\Tilde{k}_i-u)_+$. We observe that 
    \begin{align}\label{boundofu}
        u_i\leq(\mu_1+L-\mu_2)_+\leq\max\{(\mu_1+L-\mu_2)_+,1\}:=\Tilde{L}.
    \end{align}
    It is clear that $u_{i+1}\leq \Tilde{u}_i\leq u_i$. Let $\{\varphi_i\}_{i=0}^{\infty}\subset C_c^\infty(\Tilde{B}_i)$ be a sequence of nonnegative function such that $0\leq \varphi_i\leq 1$ in $\Tilde{B}_i$, $\varphi_i\equiv 1$ in $\Tilde{B}_i$ and $|\nabla\varphi_i|\leq \frac{2^{i+1}}{r_i}$. We define the sequence $\{A_i\}_{i=0}^{\infty}$ by $$A_i:=\frac{w(\{B_i: u\leq k_i\})}{w(B_i)}.$$
    We claim that this sequence satisfies the assumption of Lemma \ref{seqconv}. To this end, we estimate the following integral:
\begin{align}\label{SAA1}
    A_{i+1}^\frac{1}{\kappa}(k_i-k_{i+1})^p&= \Big(\frac{1}{w(B_{i+1})}\int_{\{B_{i+1}:u\leq k_{i+1}\}}(k_i-k_{i+1})^{p\kappa}\;w\;dx\Big)^\frac{1}{\kappa}\nonumber\\
    &\leq \Big(\frac{1}{w(B_{i+1})}\int_{\{B_{i+1}:u\leq k_{i+1}\}}(k_i-u)^{p\kappa}\;w\;dx\Big)^\frac{1}{\kappa}\nonumber\\
    &\leq \Big(\fint_{B_{i+1}}u_i^{p\kappa}\;w\;dx\Big)^\frac{1}{\kappa}\nonumber\\
    &\leq C\Big(\fint_{B_{i}}(\varphi_iu_i)^{p\kappa}\;w\;dx\Big)^\frac{1}{\kappa}\nonumber\\
    &\leq Cr_i^p\fint_{B_i}|\nabla (\varphi_iu_i)|^pw\;dx\nonumber\\
    &\leq Cr_i^p\Big(\fint_{B_i}\varphi_i^p|\nabla u_i|^pw\;dx+\fint_{B_i} u_i^p|\nabla \varphi_i|^pw\;dx\Big),
\end{align}
for some constant $C=C(n,p,[w]_p)>0$. Since $u$ is a weak supersolution of \eqref{ME}, by Lemma \ref{Energy} we have 
\begin{align}\label{SAA2}
   \fint_{B_i}\varphi_i^p|\nabla u_i|^pw\;dx\leq C(I_1+I_2+I_3+I_4+I_5), 
\end{align}
for some constant $C=C(p,\Lambda,\alpha,C_1,C_2,[w]_p)>0$. Here
$$I_1:=\fint_{B_i} u_i^p|\nabla \varphi_i|^p\;w\;dx,$$
$$I_2:=\fint_{B_i}\int_{B_i}\Big(\frac{|\varphi_i(x)-\varphi_i(y)|\max\{u_i(x),u_i(y)\}}{|x-y|^{s}}\Big)^p\frac{ w(x)w(y)}{w(B_{x,y})}\;dx\;dy,$$
$$I_3:=\big(\fint_{B_i}u_i\varphi_i^p w\;dx\big)\sup_{y\in \supp (\Tilde{B}_i)}\int_{\R^n\setminus B_i}\frac{|u_i(x)|^{p-1}}{|x-y|^{sp}}\frac{w(x)}{w(B_{x,y})}\;dx,$$
$$I_4:=\fint_{B_i}|f|u_i\varphi_i^p\;dx,\mbox{ and }I_5:=\fint_{B_i}|g||u|^{p-1}u_i\varphi_i^p\;dx.$$
\textbf{Estimate of $I_1$:} Using properties of $\varphi_i$ and (\ref{boundofu}), we get 
\begin{align}\label{SAAA3}
    I_1\leq \frac{2^{(i+1)p}\Tilde{L}^p}{r_i^p}A_i.
\end{align}
\textbf{Estimate of $I_2$:} Using Lemma \ref{intlem} and (\ref{boundofu}) along with $0<r\leq 1$, we deduce
\begin{align}\label{SAA4}
    I_2&\leq \frac{C2^{ip}}{r_i^p}\fint_{B_i}\Big(\int_{B_i}|x-y|^{p-sp}\frac{ w(x)}{w(B_{x,y})}\;dx\Big)u_i^pw\;dy\nonumber\\
    &\leq \frac{C2^{ip}}{r_i^{sp}}\fint_{B_i}u_i^p(y)w(y)\;dy\leq \frac{C2^{ip}\Tilde{L}^p}{r_i^{p}}A_i,
\end{align}
for some constant $C=C(n,p,s,[w]_p)>0$.\\
\textbf{Estimate of $I_3$:}
Using (\ref{T17}) and (\ref{T18}), proceeding as in (\ref{T19}) and using $0<r\leq 1$, one has
\begin{align}\label{SS1}
    I_3&\leq C2^{i(p+sp+np)}\Tilde{L}A_i\int_{\R^n\setminus B_i}\frac{|u_i(x)|^{p-1}}{|x-x_0|^{sp}}\frac{w(x)}{w(B_{x,x_0})}\;dx\nonumber\\
    &\leq \frac{C2^{i(p+sp+np)}\Tilde{L}^p}{r_i^{p}}A_i
\end{align}
for some constant $C=C(n,p,s,\Lambda,[w]_p)>0$.\\
\textbf{Estimate of $I_4$:}
Using H\"{o}lder's inequality and \eqref{conditiong}, we obtain
\begin{align}\label{SAA5}
    I_4\leq \fint_{B_i}|f|u_i\varphi_i^p\;dx&\leq \Tilde{L}\big(\fint_{B_i}|\frac{f}{w}|^qw\;dx\big)^\frac{1}{q}A_i^{1-\frac{1}{q}}\nonumber\\
    &\leq \frac{M\Tilde{L}}{r_i^p}A_i^{1-\frac{1}{q}}\leq \frac{M\Tilde{L}^p}{r_i^p}A_i^{1-\frac{1}{q}},
\end{align}
where we have used the fact $\Tilde{L}\geq 1$.\\
\textbf{Estimate of $I_5$:} Utilizing (\ref{conditiong}), we have
\begin{align}\label{SAA6}
    I_5\leq \fint_{B_i}|g|~|u|^{p-1}u_i\;dx\leq \Tilde{L}^p \big(\fint_{B_i}|\frac{g}{w}|^qw\;dx\big)^\frac{1}{q}A_i^{1-\frac{1}{q}}\leq \frac{M \Tilde{L}^p}{r_i^p}A_i^{1-\frac{1}{q}}
\end{align}
Incorporating (\ref{SAAA3}), (\ref{SAA4}), (\ref{SS1}), (\ref{SAA5}), and (\ref{SAA6}) into (\ref{SAA2}), and applying the resulting estimate in (\ref{SA1}) to deduce 
\begin{align}\label{SAA7}
    A_{i+1}^\frac{1}{\kappa}(k_i-k_{i+1})^p\leq C2^{i(p+np+sp)}\Tilde{L}^p\Big(A_i+A_i^{1-\frac{1}{q}}\Big)\leq C2^{i(p+np+sp)}A_i^{1-\frac{1}{q}},
\end{align}
for some constant $C=C(n,p,s,\Lambda,\alpha,C_1,C_2,\mu_1,\mu_2,[w]_p,L,M,\delta)>0$.
 In the last inequality, we used the fact that $A_i\leq 1$. Thus, (\ref{SAA7}) yields
\begin{align}
    A_{i+1}\leq C2^{i(p+np+sp)\kappa}\Big(\frac{1}{k_i-k_{i+1}}\Big)^{p\kappa}A_i^{(1-\frac{1}{q})\kappa}\leq C2^{i(2p+np+sp)\kappa}A_i^{\big(1-\frac{1}{q}\big)\kappa},
\end{align}
for some constant $C=C(n,p,s,\Lambda,\alpha,C_1,C_2,\mu_1,\mu_2,[w]_p,L,M,\delta)>0$. Therefore, the sequence $\{A_i\}_{i=0}^{\infty}$ satisfies the assumption of Lemma \ref{seqconv} with $a=C$, $b=2^{(2p+np+sp)\kappa}>1$, and $\beta=(1-\frac{1}{q})\kappa-1>0$ (as $q>\kappa'$). We choose $\sigma:=a^{-\frac{1}{\beta}}b^{-\frac{1}{\beta^2}}$. Thus, by Lemma \ref{seqconv}, if $A_0\leq \sigma$, then $A_i\to 0$ as $i\to\infty$. Consequently, $u\geq \mu_1+\delta L$ in $B_\frac{r}{2}(x_0)$.
\end{proof}

\section{Proof of the main results}
\subsection{Local boundedness}
\begin{proof}[Proof of Theorem \ref{Bdd}]
For $j\in\mathbb{N}\cup\{0\}$, we define 
$$
r_j=(1+2^{-j})\frac{r}{2},\,\Tilde{r_j}=\frac{r_j+r_{j+1}}{2}\text{ and }B_j=B_{r_j}(x_0),\,\Tilde{B}_j=B_{\Tilde{r}_j}(x_0).
$$
For $k>0$, we also define $k_j=(1-2^{-j})k$ and $\Tilde{k}_j=\frac{k_j+k_{j+1}}{2}$. Set $u_j=(u-k_j)_+$ and $\Tilde{u}_j=(u-\Tilde{k}_j)_+$. It is clear that $r_{j+1}\leq \Tilde{r}_j\leq r_{j}$, $k_j\leq \Tilde{k}_j\leq k_{j+1}$, and hence $u_{j+1}\leq \Tilde{u}_j\leq u_j$. Let $\varphi_j\in C^\infty_c(\Tilde{B}_j(x_0))$ be a nonnegative function such that $0\leq \varphi_j\leq 1$ in $B_j$, $|\nabla\varphi_j|\leq \frac{{2^{j+1}}}{r_j}$ in $B_j$, and $\varphi_j\equiv 1$ in $B_{j+1}$. Using the fact $\Tilde{u}_j\geq k_{j+1}-\Tilde{k}_j=\frac{k}{2^{j+2}}$ in $\supp(u_{j+1})$ and Lemma \ref{comlem}, we estimate the following integral:
\begin{align}\label{T11}
    \Big(\fint_{B_{j+1}}u_{j+1}^p\;w\;dx\Big)^\frac{1}{\kappa}&\leq \Big(\fint_{B_{j+1}}\Tilde{u}_{j}^p\big(\frac{2^{j+2}\Tilde{u}_{j}}{k}\big)^{p(\kappa-1)}\;w\;dx\Big)^\frac{1}{\kappa}\nonumber\\
    &\leq \big(\frac{2^{j+2}}{k}\big)^\frac{p(\kappa-1)}{\kappa}\Big(\fint_{B_{j+1}}\Tilde{u}_{j}^{\kappa p}\;w\;dx\Big)^\frac{1}{\kappa}\nonumber\\
    &\leq \big(\frac{2^{j+2}}{k}\big)^\frac{p(\kappa-1)}{\kappa}\Big(\frac{w(B_j)}{w(B_{j+1})}\fint_{B_{j}}|\varphi_j\Tilde{u}_{j}|^{\kappa p}\;w\;dx\Big)^\frac{1}{\kappa}\nonumber\\
    &\leq C\big(\frac{2^{j+2}}{k}\big)^\frac{p(\kappa-1)}{\kappa}\Big(\fint_{B_{j}}|\varphi_j\Tilde{u}_{j}|^{\kappa p}\;w\;dx\Big)^\frac{1}{\kappa},
\end{align}
for some constant $C=C(n,p,[w]_p)>0$. Due to Theorem \ref{WIT}, we have
\begin{align}\label{T12}
    \Big(\fint_{B_{j}}|\varphi_j\Tilde{u}_{j}|^{\kappa p}\;w\;dx\Big)^\frac{1}{\kappa}&\leq Cr_j^p \fint_{B_j}|\nabla (\varphi_j\Tilde{u}_j)|^p\;w\;dx\nonumber\\
    &\leq Cr_j^p \big( \fint_{B_j}|\nabla \varphi_j|^p\Tilde{u}_j^p\;w\;dx+\fint_{B_j}|\nabla \Tilde{u}_j|^p\varphi_j^p\;w\;dx\big),
\end{align}
for some constant $C=C(n,p,[w]_p)>0$. Utilizing Lemma \ref{Energy}, we deduce
\begin{align}\label{T13}
    \fint_{B_j}|\nabla \Tilde{u}_j|^p\varphi_j^p\;w\;dx&\leq C\Big(I_1+I_2+I_3\Big)+I_4+I_5,
\end{align}
for some constant $C=C(p,\Lambda,\alpha,C_1,C_2)>0$,
where 
$$I_1:=\fint_{B_j} \Tilde{u}_j^p|\nabla \varphi_j|^p\;w\;dx,$$
$$I_2:=\fint_{B_j}\int_{B_j}\Big(\frac{|\varphi_j(x)-\varphi_j(y)|\max\{\Tilde{u}_j(x),\Tilde{u}_j(y)\}}{|x-y|^{s}}\Big)^p\frac{ w(x)w(y)}{w(B_{x,y})}\;dx\;dy,$$
$$I_3:=\big(\fint_{B_j}\Tilde{u}_j\varphi_j^p w\;dx\big)\sup_{y\in \supp \,\varphi_j}\int_{\R^n\setminus B_j}\frac{|\Tilde{u}_j(x)|^{p-1}}{|x-y|^{sp}}\frac{w(x)}{w(B_{x,y})}\;dx,$$
$$I_4:=\fint_{B_j}|f|\Tilde{u}_j\varphi_j^p\;dx,\mbox{ and }I_5:=\fint_{B_j}|g||u|^{p-1}\Tilde{u}_j\varphi_j^p \;dx.$$
\textbf{Estimate of $I_1$:} Using the properties of $\varphi_j$ and the fact $\Tilde{u}_j\leq u_j$, we get
\begin{align}\label{T14}
    I_1\leq \frac{2^{(j+1)p}}{r_j^p}\fint_{B_j} u_j^p\;w\;dx.
\end{align}\\
\textbf{Estimate of $I_2$:} Using Lemma \ref{intlem}, we observe that
\begin{align}\label{T15}
    I_2&\leq \frac{2^{jp}C}{r_j^p}\fint_{B_j}\int_{B_j}|x-y|^{(1-s)p}\Tilde{u}_j(y)^p\frac{ w(x)w(y)}{w(B_{x,y})}\;dx\;dy\nonumber\\
    &\leq \frac{2^{jp}C}{r_j^{p}}\fint_{B_j} u_j^p~w\;dx,
\end{align}
where $C=C(n,p,{s},[w]_p)>0$ is a constant. The last inequality holds due to the fact {$r_j\leq 1$}. \\
\textbf{Estimate of $I_3$:} Note that 
\begin{align}\label{note}
    u_j=u-k_j=u-\Tilde{k}_j+\Tilde{k}_j-k_j\geq \Tilde{k}_j-k_j\geq \frac{k_{j+1}-k_j}{2}=\frac{k}{2^{j+2}} \mbox{ in } \supp(\Tilde{u}_j).
\end{align}
Using $0\leq \varphi_j\leq 1$, we deduce
\begin{align}\label{T16}
I_3&=\big(\fint_{B_j}\Tilde{u}_j\varphi_j^p\;w\;dx\big)\sup_{y\in \Tilde{B}_j}\int_{\R^n\setminus B_j}\frac{|\Tilde{u}_j(x)|^{p-1}}{|x-y|^{sp}}\frac{w(x)}{w(B_{x,y})}\;dx\nonumber\\
    &\leq \frac{2^{(j+2)(p-1)}}{k^{p-1}}\big(\fint_{B_j}u_j^{p}w\;dx\big)\sup_{y\in \Tilde{B}_j}\int_{\R^n\setminus B_j}\frac{|\Tilde{u}_j(x)|^{p-1}}{|x-y|^{sp}}\frac{w(x)}{w(B_{x,y})}\;dx.
\end{align}
For $x\in \R^n\setminus B_j,\;y\in \Tilde{B}_{j}$, we have
\begin{align}\label{T17}
    \frac{|x-x_0|}{|x-y|}\leq 1+\frac{|y-x_0|}{|x-y|}\leq 1+\frac{\Tilde{r}_j}{(r_j-\Tilde{r}_j)}\leq 2^{j+3}.
\end{align}
Furthermore, we have $B_{x,y}\cup B_{x,x_0}\subset B_{3|x-x_0|}(x_0)$. This together with Lemma \ref{comlem} and \eqref{T17} reveals 
\begin{align}\label{T18}
    \frac{w(B_{x,x_0})}{w(B_{x,y})}=\frac{w(B_{x,x_0})}{w(B_{3|x-x_0|}(x_0))}\frac{w(B_{3|x-x_0|}(x_0))}{w(B_{x,y})}&\leq C\Big(\frac{|B_{x,x_0}|}{|B_{3|x-x_0|}(x_0)|}\Big)^\sigma\Big(\frac{|B_{3|x-x_0|}(x_0)|}{|B_{x,y}|}\Big)^p\nonumber\\
    &\leq C\Big(\frac{|x-x_0|}{|x-y|}\Big)^{np}\leq C2^{(j+3)np},
\end{align}
for some constant $C=C(n,p,[w]_p)>0$. Combining (\ref{T17}) and (\ref{T18}), estimate (\ref{T16}) yields
\begin{align}\label{T19}
    I_3&\leq \frac{C2^{(j+2)(p-1)+(j+3)(sp+np)}}{k^{p-1}}\big(\fint_{B_j}u_j^{p}w\;dx\big)\int_{\R^n\setminus B_j}\frac{|\Tilde{u}_j(x)|^{p-1}}{|x-x_0|^{sp}}\frac{w(x)}{w(B_{x,x_0})}\;dx\nonumber\\
    &\leq \frac{C2^{j(p+sp+np)}}{r_j^p}\frac{\Tail^{p-1}(\Tilde{u}_j;x_0,r_j)}{k^{p-1}}\big(\fint_{B_j}u_j^{p}w\;dx\big)\nonumber\\
    &\leq \frac{C2^{j(p+sp+np)}}{r_j^p}\frac{\Tail^{p-1}(u_+;x_0,\frac{r}{2})}{k^{p-1}}\big(\fint_{B_j}u_j^{p}w\;dx\big),
\end{align}
where $C=C(n,p,s,[w]_p)>0$.\\
\textbf{Estimate of $I_4$:} Applying H\"{o}lder's inequality, Young's inequality and $q>\kappa'$, we obtain
\begin{align}\label{T20}
    r_j^pI_4&\leq r_j^p\fint_{B_j}\frac{|f|}{w}\Tilde{u}_j\varphi_j^pw\;dx\nonumber\\
    &\leq r_j^p\Big(\fint_{B_j}\Big(\frac{|f|}{w}\Big)^qw\;dx\Big)^\frac{1}{q}\Big(\fint_{B_j}\big|\Tilde{u}_j\varphi_j\big|^{p\kappa}w\;dx\Big)^\frac{1}{p\kappa}\Big(\frac{w(\supp(\Tilde{u}_j\varphi_j) )}{w(B_j)}\Big)^{1-\frac{1}{q}-\frac{1}{p\kappa}}\nonumber\\
    &\leq \epsilon\Big(\fint_{B_j}\big|\Tilde{u}_j\varphi_j\big|^{p\kappa}w\;dx\Big)^\frac{1}{\kappa}+C(\epsilon)r_j^{pp'}\Big(\fint_{B_j}\Big(\frac{|f|}{w}\Big)^qw\;dx\Big)^\frac{p'}{q}\Big(\frac{w(\supp(\Tilde{u}_j\varphi_j) )}{w(B_j)}\Big)^{p'(1-\frac{1}{q}-\frac{1}{p\kappa})}\nonumber\\
    &\leq \epsilon\Big(\fint_{B_j}\big|\Tilde{u}_j\varphi_j\big|^{p\kappa}w\;dx\Big)^\frac{1}{\kappa}+C(\epsilon)r_j^{pp'}\Big(\fint_{B_j}\Big(\frac{|f|}{w}\Big)^qw\;dx\Big)^\frac{p'}{q}\Big(\frac{w(\supp(\Tilde{u}_j\varphi_j) )}{w(B_j)}\Big)^{(1-\frac{1}{q})},
\end{align}
where $\epsilon>0$ is an arbitrary real number. Applying \eqref{note}, we obtain
\begin{align}\label{T21}
    \frac{w(\supp(\Tilde{u}_j\varphi_j) )}{w(B_j)}\leq \frac{1}{w(B_j)}\int_{\{B_j:\Tilde{u}_j\geq \Tilde{k}_j\}}w\;dx \leq 
    \frac{2^{(j+2)p}}{k^{p}}\fint_{B_j}u_j^pw\;dx.
\end{align}
Utilizing \eqref{T21} into (\ref{T20}) gives
\begin{align}\label{T22}
    r_j^pI_4\leq \epsilon\Big(\fint_{B_j}\big|\Tilde{u}_j\varphi_j\big|^{p\kappa}w\;dx\Big)^\frac{1}{\kappa}+C(\epsilon,p)r_j^{pp'}\Big(\frac{2^j}{k}\Big)^{p(1-\frac{1}{q})}\Big(\fint_{B_j}\Big(\frac{|f|}{w}\Big)^qw\;dx\Big)^\frac{p'}{q}\Big(\fint_{B_j}u_j^pw\;dx\Big)^{(1-\frac{1}{q})}.
\end{align}
\textbf{Estimate of $I_5$:} Utilizing the definitions of $\Tilde{u}_j$ and $\varphi_j$, we deduce
\begin{align}\label{T23}
   r_j^p I_5&\leq r_j^p\fint_{B_j}|g|~|u|^{p-1}\Tilde{u}_j\varphi_j^p \;dx\leq r_j^p\fint_{B_j}|g|~|\Tilde{u}_j+\Tilde{k}_j|^{p-1}\Tilde{u}_j\varphi_j^p \;dx\nonumber\\
    &\leq Cr_j^p\fint_{B_j}|g|~|\Tilde{u}_j|^{p}\varphi_j^p \;dx+Cr_j^p\Tilde{k}_j^{p-1}\fint_{B_j}|g|\Tilde{u}_j\varphi_j^p \;dx\nonumber\\
    &\leq Cr_j^p \fint_{B_j}|g|~|\Tilde{u}_j|^{p}\varphi_j^p \;dx+Cr_j^p\frac{\Tilde{k}_j^{p-1}}{w(B_j)}\Big(\int_{\{B_j:u\geq 2\Tilde{k}_j\}}|g|~\Tilde{u}_j\varphi_j^p \;dx+\int_{\{B_j:\Tilde{k}_j\leq u\leq 2\Tilde{k}_j\}}|g|~\Tilde{u}_j\varphi_j^p \;dx\Big)\nonumber\\
    &\leq Cr_j^p\fint_{B_j}|g|~|\Tilde{u}_j|^{p}\varphi_j^p \;dx+Cr_j^p\fint_{B_j}|g|~|\Tilde{u}_j|^{p}\varphi_j^p \;dx+ Cr_j^p\frac{\Tilde{k}_j^p}{w(B_j)}\int_{B_j\cap \supp(\Tilde{u}_j\varphi_j)}|g|~\varphi_j^p \;dx\nonumber\\
    &\leq Cr_j^p\fint_{B_j}|g|~|\Tilde{u}_j|^{p}\varphi_j^p \;dx+Cr_j^p\frac{\Tilde{k}_j^p}{w(B_j)}\int_{B_j\cap \supp(\Tilde{u}_j\varphi_j)}|g| \;dx,
\end{align}
for some $C=C(p)>0$. Due to H\"{o}lder's inequality, interpolation inequality, (\ref{conditiong}) and $q>\kappa'$, we have
\begin{align}\label{T24}
   r_j^p\fint_{B_j}|g|~|\Tilde{u}_j|^{p}\varphi_j^p \;dx &\leq r_j^p\Big(\fint_{B_j}\Big(\frac{|g|}{w}\Big)^qw\;dx\Big)^\frac{1}{q}\Big(\fint_{B_j}|\Tilde{u}_j\varphi_j|^{pq'}w\;dx\Big)^\frac{1}{q'}\nonumber\\ 
   &\leq M\Big(\fint_{B_j}|\Tilde{u}_j\varphi_j|^{p}w\;dx\Big)^\theta \Big(\fint_{B_j}|\Tilde{u}_j\varphi_j|^{p\kappa}w\;dx\Big)^\frac{1-\theta}{\kappa}\nonumber\\
   &\leq \epsilon\Big(\fint_{B_j}|\Tilde{u}_j\varphi_j|^{p\kappa}w\;dx\Big)^\frac{1}{\kappa}+C(\epsilon,M,\theta)\Big(\fint_{B_j}|\Tilde{u}_j\varphi_j|^{p}w\;dx\Big)
\end{align}
where $\theta=\frac{q\kappa-\kappa-q}{q(\kappa-1)}\in (0,1)$ and $\epsilon>0$ is an arbitrary small number. Thanks to (\ref{T21}) and (\ref{conditiong}), we get
\begin{align}\label{T25}
   \frac{r_j^p}{w(B_j)}\int_{B_j\cap \supp(\Tilde{u}_j\varphi_j)}|g| \;dx
    &\leq r_j^p\Big(\fint_{B_j}\Big(\frac{|g|}{w}\Big)^q w\;dx\Big)^\frac{1}{q}\Big(\frac{w(\supp(\Tilde{u}_j\varphi_j) )}{w(B_j)}\Big)^{(1-\frac{1}{q})}\nonumber\\
    &\leq M\Big(\frac{2^{(j+2)}}{k}\Big)^{p(1-\frac{1}{q})}  \Big(\fint_{B_j}u_j^pw\;dx\Big)^{(1-\frac{1}{q})}
\end{align}
Combining (\ref{T24}) and (\ref{T25}), estimate (\ref{T23}) yields
\begin{align}\label{T26}
    r_j^pI_5&\leq \epsilon\Big(\fint_{B_j}|\Tilde{u}_j\varphi_j|^{p\kappa}w\;dx\Big)^\frac{1}{\kappa}+C(\epsilon,M,\theta)\Big(\fint_{B_j}|\Tilde{u}_j\varphi_j|^{p}w\;dx\Big)\nonumber\\
    &+M 2^{(j+2)p(1-\frac{1}{q})}k^\frac{p}{q} \Big(\fint_{B_j}u_j^pw\;dx\Big)^{(1-\frac{1}{q})}.
\end{align}
Substituting the estimates (\ref{T14}), (\ref{T15}), \eqref{T19}, (\ref{T22}) and (\ref{T26}) into (\ref{T13}) and subsequently applying the resulting bound in (\ref{T12}), we get
\begin{align*}
    \Big(\fint_{B_{j}}|\varphi_j\Tilde{u}_{j}|^{\kappa p}w\;dx\Big)^\frac{1}{\kappa}&\leq C\Bigg\{2^{j(p+sp+np)}\Big(1+\frac{\Tail^{p-1}(u_+;x_0,\frac{r}{2})}{k^{p-1}}\Big)\fint_{B_j}u_j^pw\;dx\nonumber\\
    &+C(\epsilon,p)r_j^{pp'}\Big(\frac{2^{j+2}}{k}\Big)^{p(1-\frac{1}{q})}\Big(\fint_{B_j}\Big(\frac{|f|}{w}\Big)^qw\;dx\Big)^\frac{p'}{q}\Big(\fint_{B_j}u_j^pw\;dx\Big)^{(1-\frac{1}{q})}\nonumber\\
    &+M 2^{(j+2)p(1-\frac{1}{q})}k^\frac{p}{q} \Big(\fint_{B_j}u_j^pw\;dx\Big)^{(1-\frac{1}{q})}+C(\epsilon,M,\theta)\Big(\fint_{B_j}|\Tilde{u}_j\varphi_j|^{p}w\;dx\Big)\Bigg\}\nonumber\\
&+2C\epsilon\Big(\fint_{B_j}\big|\Tilde{u}_j\varphi_j\big|^{p\kappa}w\;dx\Big)^\frac{1}{\kappa},
\end{align*}
for some constant {$C=C(n,p,s,\Lambda,\alpha,C_1,C_2,[w]_p)>0$.}
Choosing $\epsilon=\frac{1}{4C}$ and using $\Tilde{u}_j\leq u_j$, we deduce that
\begin{align}\label{T27}
    \Big(\fint_{B_{j}}|\varphi_j\Tilde{u}_{j}|^{\kappa p}w\;dx\Big)^\frac{1}{\kappa}&\leq C2^{j(p+sp+np)}\Bigg\{\Big(1+\frac{\Tail^{p-1}(u_+;x_0,\frac{r}{2})}{k^{p-1}}\Big)\fint_{B_j}u_j^pw\;dx\nonumber\\
    &+r_j^{pp'}\Big(\frac{1}{k}\Big)^{p(1-\frac{1}{q})}\Big(\fint_{B_j}\Big(\frac{|f|}{w}\Big)^qw\;dx\Big)^\frac{p'}{q}\Big(\fint_{B_j}u_j^pw\;dx\Big)^{(1-\frac{1}{q})}\nonumber\\
    &+k^\frac{p}{q} \Big(\fint_{B_j}u_j^pw\;dx\Big)^{(1-\frac{1}{q})}\Bigg\},
\end{align}
for some constant {$C=C(n,p,s,q,\kappa,\Lambda,\alpha,C_1,C_2,[w]_p,M)>0$.} Using \eqref{T27} in \eqref{T11}, we obtain
\begin{align*}
    \Big(\fint_{B_{j+1}}u_{j+1}^p w\;dx\Big)^\frac{1}{\kappa}&\leq C\frac{2^{j(2p+sp+np)}}{k^{p(1-\frac{1}{\kappa})}}\Bigg\{\Big(1+\frac{\Tail^{p-1}(u_+;x_0,\frac{r}{2})}{k^{p-1}}\Big)\fint_{B_j}u_j^pw\;dx\nonumber\\
    &+r_j^{pp'}\Big(\frac{1}{k}\Big)^{p(1-\frac{1}{q})}\Big(\fint_{B_j}\Big(\frac{|f|}{w}\Big)^qw\;dx\Big)^\frac{p'}{q}\Big(\fint_{B_j}u_j^pw\;dx\Big)^{(1-\frac{1}{q})}\nonumber\\
    &+k^\frac{p}{q} \Big(\fint_{B_j}u_j^pw\;dx\Big)^{(1-\frac{1}{q})}\Bigg\},
\end{align*}
for some constant {$C=C(n,p,s,q,\kappa,\Lambda,\alpha,C_1,C_2,[w]_p,M)>0$.}
For $\delta\in (0,1]$, we choose 
\begin{align}\label{K1}
    k\geq k_1:=\delta\Tail(u_+;x_0,\frac{r}{2})
\end{align}
and deduce
\begin{align}\label{T28}
    \Big(\fint_{B_{j+1}}u_{j+1}^p w\;dx\Big)^\frac{1}{\kappa}&\leq C\frac{2^{j(2p+sp+np)}}{k^{p(1-\frac{1}{\kappa})}}\Bigg\{\delta^{1-p}\fint_{B_j}u_j^pw\;dx\nonumber\\
    &+r_j^{pp'}\Big(\frac{1}{k}\Big)^{p(1-\frac{1}{q})}\Big(\fint_{B_j}\Big(\frac{|f|}{w}\Big)^qw\;dx\Big)^\frac{p'}{q}\Big(\fint_{B_j}u_j^pw\;dx\Big)^{(1-\frac{1}{q})}\nonumber\\
    &+k^\frac{p}{q} \Big(\fint_{B_j}u_j^pw\;dx\Big)^{(1-\frac{1}{q})}\Bigg\},
\end{align}
for some constant {$C=C(n,p,s,q,\kappa,\Lambda,\alpha,C_1,C_2,[w]_p,M)>0$.}
We define a sequence {$\{A_j\}_{j=0}^{\infty}$} such that 
$$
A_j=k^{-p}\fint_{B_j}u_j^pw\;dx.
$$
and observe that 
\begin{align}\label{T29}
    A_j\leq k^{-p}\frac{w(B_0)}{w(B_j)}\fint_{B_0}u_0^pw\;dx&\leq k^{-p}\frac{w(B_r(x_0))}{w(B_\frac{r}{2}(x_0))}\fint_{B_r(x_0)}{u_+}^pw\;dx\nonumber\\
    &\leq C(n,p,[w]_p)k^{-p}\fint_{B_r(x_0)}{u_+}^pw\;dx.
\end{align}
We further choose 
\begin{align}\label{K2}
    k\geq k_2:=\Big(C(n,p,[w]_p)\fint_{B_r(x_0)}{(u_+)}^pw\;dx\Big)^\frac{1}{p}
\end{align}
to ensure that $A_j\leq 1$ {for every $j\in\mathbb{N}\cup\{0\}$}. Consequently, (\ref{T28}) reveals
\begin{align*}
    \Big(k^pA_{j+1}\Big)^\frac{1}{\kappa}&\leq C\frac{2^{j(2p+sp+np)}}{k^{p(1-\frac{1}{\kappa})}}\Bigg\{\delta^{1-p}k^pA_j+r_j^{pp'}\Big(\fint_{B_j}\Big(\frac{|f|}{w}\Big)^qw\;dx\Big)^\frac{p'}{q}A_j^{(1-\frac{1}{q})}+k^p A_j^{(1-\frac{1}{q})}\Bigg\}\nonumber\\
    &\leq C\frac{2^{j(2p+sp+np)}}{k^{p(1-\frac{1}{\kappa})}}\Bigg\{\delta^{1-p}k^p+r_j^{pp'}\Big(\fint_{B_j}\Big(\frac{|f|}{w}\Big)^qw\;dx\Big)^\frac{p'}{q}+k^p\Bigg\}A_j^{(1-\frac{1}{q})},
\end{align*}
which yields
\begin{align*}
    A_{j+1}\leq \big(C2^{j(2p+sp+np)}\big)^\kappa\Bigg\{ \delta^{1-p}+\frac{r^{pp'}}{k^p}\Big(\fint_{B_r(x_0)}\Big(\frac{|f|}{w}\Big)^qw\;dx\Big)^\frac{p'}{q}+1 \Bigg\}^\kappa A_j^{\kappa(1-\frac{1}{q})},
\end{align*}
for some constant {$C=C(n,p,s,q,\kappa,\Lambda,\alpha,C_1,C_2,[w]_p,M)>0$,} where we have used Lemma \ref{comlem}, and the fact $r_j\leq r$. Whenever 
\begin{align}\label{K3}
    k\geq k_3:=r^{p'}\Big(\fint_{B_r}\Big(\frac{|f|}{w}\Big)^qw\;dx\Big)^\frac{1}{q(p-1)},
\end{align}
we have
\begin{align}\label{T30}
     A_{j+1}\leq C2^{j(2p+sp+np)\kappa} \delta^{(1-p)\kappa} A_j^{\kappa(1-\frac{1}{q})},
\end{align}
for some constant {$C=C(n,p,s,q,\kappa,\Lambda,\alpha,C_1,C_2,[w]_p,M)>0$}. Therefore, $\{A_j\}_{j=0}^{\infty}$ satisfies the assumption of Lemma \ref{seqconv} with $a=C\delta^{(1-p)\kappa},\;b=2^{(2p+sp+np)\kappa}>1$ and $\beta=\kappa\big(1-\frac{1}{q}\big)-1>0$. Thus, $A_j\to 0$ as $j\to\infty$, whenever $$A_0=k^{-p}\fint_{B_r(x_0)}{u_+}^pw\;dx\leq (C\delta^{(1-p)\kappa})^{-\frac{1}{\beta}}2^{-\frac{(2p+sp+np)\kappa}{\beta^2}},$$
that is, 
\begin{align}\label{K4}
    k\geq k_4:=C^{\frac{1}{p\beta}} \delta^{\frac{-(p-1)}{p}\frac{\kappa q'}{\kappa-q'} }2^{\frac{(2p+sp+np)\kappa}{p\beta^2}}\Big(\fint_{B_r}{u_+}^pw\;dx\Big)^\frac{1}{p},
\end{align}
{where $C=C(n,p,s,q,\kappa,\Lambda,\alpha,C_1,C_2,[w]_p,M)>0$} is some constant.
 For $\mu>0$, we finally choose 
$k=k_1+k_2+k_3+k_4+\mu>0$, ensuring (\ref{K1}), (\ref{K2}), (\ref{K3}) and (\ref{K4}), and hence $A_j\to 0$ as $j\to\infty$. Consequently, $u\leq k_1+k_2+k_3+k_4+\mu$ in $B_\frac{r}{2}(x_0)$. Taking $\mu\to 0$, we obtain our desired result. 
\end{proof}

\subsection{H\"older continuity}
\begin{proof}[Proof of Theorem \ref{cty}]
Since $u$ is a weak solution of (\ref{ME}) with $g\equiv 0$, by Lemma \ref{Subsolution} and Remark \ref{rmksub}, we conclude that $u_+$ and $u_-$ are weak subsolutions of 
\begin{align*}
    \mathcal{M}_\alpha\,u=|f|\text{ in }\Omega.
\end{align*}
Let $B_r(x_0)\Subset\Omega$, {for some $0<r\leq 1$.} Using Theorem \ref{Bdd} with $\delta=1$, there exists a constant $C=C(n,p,s,q,\kappa,\Lambda,\alpha,$ $C_1,C_2,[w]_p,M)>0$ such that $\sup_{B_\frac{r}{2}(x_0)}u_+$ and $\sup_{B_\frac{r}{2}(x_0)}u_-$ are bounded by the following quantity:
\begin{align*}
    C\Big(\fint_{B_r(x_0)}|u|^pw\;dx\Big)^\frac{1}{p}+\mathrm{Tail}\Big(u;x_0,\frac{r}{2}\Big)+r^{p'}\Big(\fint_{B_r}\Big(\frac{|f|}{w}\Big)^qw\;dx\Big)^\frac{1}{q(p-1)}.
\end{align*}
Therefore, 
    \begin{align}\label{basecase}
        \mathrm{ osc }_{B_\frac{r}{2}(x_0)} u\leq \tau_0:=2\Big\{C\Big(\fint_{B_r(x_0)}|u|^pw\;dx\Big)^\frac{1}{p}+\mathrm{Tail}\big(u;x_0,\frac{r}{2}\big)+r^{p'}\Big(\fint_{B_r}\Big(\frac{|f|}{w}\Big)^qw\;dx\Big)^\frac{1}{q(p-1)}\Big\}.
    \end{align}
For $i\in\mathbb{N}\cup\{0\}$, we define $r_i:=\eta^{i}\frac{r}{2}$ and $B_i:=B_{r_i}(x_0)$, where {$\eta\in(0,\frac{1}{4})$} will be determined later. In order to conclude (\ref{Holder}), it suffices to show that 
\begin{align}\label{oscdecay}
    \osc_{B_i}u\leq \eta^{\sigma i}\tau_0\quad\mbox{for all }i\in\mathbb{N}\cup\{0\},
\end{align}
for some {$\sigma\in(0,1)$} to be determined later. We shall prove this by induction. To this end, from (\ref{basecase}), we observe that (\ref{oscdecay}) holds for $i=0$. Let us assume that (\ref{oscdecay}) holds for $i=0,1,2,...,j$. {We aim to show the validity of (\ref{oscdecay}) for $i=j+1$.} To this end, we define $$\tau_i=\eta^{\sigma i}\tau_0,\; m_i=\inf_{B_i}u \mbox{ and }M_i=\sup_{B_i}u.$$ {Since $\eta\in(0,\frac{1}{4})$, one has} $B_{2r_{j+1}}(x_0)\Subset B_j$ and the following dichotomy holds:
\begin{itemize}
    \item[(a)] $w(\{{B_{2r_{j+1}}(x_0)}:u-m_j\geq \frac{\tau_j}{2}\})\geq \frac{1}{2}w({B_{2r_{j+1}}(x_0)})$.
    \item[(b)] $w(\{{B_{2r_{j+1}}(x_0)}:\tau_j-(u-m_j)\geq \frac{\tau_j}{2}\})\geq \frac{1}{2}w({B_{2r_{j+1}}(x_0)})$.
\end{itemize}
We define $$u_j=\begin{cases}
    u-m_j&\mbox{ if (a)} \mbox{ holds},\\
    \tau_j-(u-m_j)&\mbox{ if (b)} \mbox{ holds, }
\end{cases}$$
which is nonnegative in $B_j$. {We notice that $u_j$ is a weak supersolution of
$$
\mathcal{M}_\alpha\,v=\mathcal{F}(x,v+m_j)\text{ if (a) holds},
$$
and
$$
\mathcal{M}_\alpha\,v=-\mathcal{F}(x,\tau_j+m_j-v)\text{ if (b) holds}.
$$
}
Therefore, using Lemma \ref{expansionofpositivity} with $R=r_j,r=r_{j+1},\;k=\frac{\tau_j}{2},\tau=\frac{1}{2}$, there exists a constant $\zeta=\zeta(n,p,s,q,\kappa,\Lambda,\alpha,C_1,C_2,[w]_p,\tau)\in (0,\frac{1}{4})$ such that 
\begin{align}\label{inf}
    \inf_{B_{j+1}}u_j\geq \frac{\zeta\tau_j}{2}-d
\end{align}
with$$d:=\Big(\underbrace{\frac{r_{j+1}}{r_j}\Big)^{p'}\Tail({(u_j)}_-;x_0,r_j)}_{d_1}+\underbrace{r_{j+1}^{p'}\Big(\fint_{B_{2r_{j+1}}(x_0)}\Big(\frac{|f|}{w}\Big)^qw\;dx\Big)^\frac{1}{q(p-1)}}_{d_2}.$$
Note that $|u_j|\leq 2\tau_i$ in $B_i$ for $i=0,1,2,...,j$ and $|u_j|\leq |u|+3\tau_0$ in $\R^n\setminus B_0$. Using these facts together with Lemma \ref{intlem}, we deduce 
\begin{align*}
    d_1^{p-1}&=\eta^p \Tail^{p-1}({(u_j)}_-;x_0,r_j)=\eta^p r_j^p\int_{\R^n\setminus B_j}\frac{|{(u_j)}_-(x)|^{p-1}}{|x-x_0|^{sp}}\frac{w(x)}{w(B_{x,x_0})}\;dx\nonumber\\
    &\leq \eta^pr_j^p\Big(\int_{\R^n\setminus B_0}\frac{(|u(x)|+3\tau_0)^{p-1}}{|x-x_0|^{sp}}\frac{w(x)}{w(B_{x,x_0})}\;dx+\sum_{i=0}^{j-1}\int_{B_i\setminus B_{i+1}}\frac{|u_j(x)|^{p-1}}{|x-x_0|^{sp}}\frac{w(x)}{w(B_{x,x_0})}\;dx\Big)\nonumber\\
    &\leq C\eta^pr_j^p\Big(\int_{\R^n\setminus B_0}\frac{|u(x)|^{p-1}}{|x-x_0|^{sp}}\frac{w(x)}{w(B_{x,x_0})}\;dx+(3\tau_0)^{p-1} \int_{\R^n\setminus B_0}\frac{1}{|x-x_0|^{sp}}\frac{w(x)}{w(B_{x,x_0})}\;dx\nonumber\\
    &+\sum_{i=0}^{j-1}\int_{\R^n\setminus B_{i+1}}\frac{(2\tau_i)^{p-1}}{|x-x_0|^{sp}}\frac{w(x)}{w(B_{x,x_0})}\;dx\Big)\nonumber\\
    &\leq C\eta^pr_j^p\Big(\frac{\Tail^{p-1}(u;x_0,r_0)}{r_0^{p}}+\frac{\tau_0^{p-1}}{r_0^{sp}}+\sum_{i=0}^{j-1}\frac{\tau_i^{p-1}}{r_{i+1}^{sp}}\Big),
\end{align*}
where $C=C(n,p,[w]_p)>0$. Since {$\tau_0>\Tail(u;x_0,r_0)$} and $r_1\leq r_0\leq 1$, the above estimate reveals 
\begin{align}\label{MC1}
    d_1^{p-1}\leq C\eta^pr_j^{p}\sum_{i=0}^{j-1}\frac{\tau_i^{p-1}}{r_{i+1}^{p}}&=C\eta^p\tau_j^{p-1}\sum_{i=0}^{j-1}\frac{(\frac{\tau_i}{\tau_j})^{p-1}}{(\frac{r_{i+1}}{r_j})^{p}}=C\eta^p\tau_j^{p-1}\sum_{i=0}^{j-1}\frac{\eta^{\sigma(p-1)(i-j)}}{\eta^{(i-j+1)p}}\nonumber\\
    &=C\tau_j^{p-1}\sum_{i=0}^{j-1}\eta^{(\sigma (p-1)-p)(i-j)}\leq C\tau_j^{p-1}\sum_{i=0}^{\infty}\eta^{(p-\sigma (p-1))(i+1)}\nonumber\\
    &\leq C\eta^\frac{p}{2}\tau_j^{p-1}\sum_{i=0}^{\infty}\eta^{\frac{pi}{2}}=\frac{C\eta^\frac{p}{2}}{1-\eta^\frac{p}{2}}\tau_j^{p-1},
\end{align}
for every $\sigma\in(0,\frac{p}{2(p-1)})$. We choose $\eta>0$ such that 
\begin{align}\label{coneta1}
    \frac{C\eta^\frac{p}{2}}{1-\eta^\frac{p}{2}}\leq \Big(\frac{\zeta}{8}\Big)^{p-1}.
\end{align}
Thus, we get $d_1\leq \frac{\delta \tau_j}{8}.$ Now, applying Lemma \ref{comlem}, we estimate $d_2$ in the following way.
\begin{align*}
    d_2^{p-1}=r_{j+1}^p\Big(\fint_{B_{2r_{j+1}}(x_0)}\Big(\frac{|f|}{w}\Big)^qw\;dx\Big)^\frac{1}{q}&\leq \eta^{(j+1)p}r^p\Big(\frac{w(B_r(x_0))}{w(B_{2r_{j+1}}(x_0))}\fint_{B_r(x_0)}\Big(\frac{|f|}{w}\Big)^qw\;dx\Big)^\frac{1}{q}\nonumber\\
    &\leq C\eta^{(j+1)p} \Big(\frac{|B_r(x_0)|}{|B_{2r_{j+1}}(x_0)|}\Big)^\frac{p}{q} r^p\Big(\fint_{B_r(x_0)}\Big(\frac{|f|}{w}\Big)^qw\;dx\Big)^\frac{1}{q}\nonumber\\
    &\leq C\eta^{(j+1)p(1-\frac{n}{q})} r^p\Big(\fint_{B_r(x_0)}\Big(\frac{|f|}{w}\Big)^qw\;dx\Big)^\frac{1}{q},
\end{align*}
for some $C=C(n,p,[w]_p)>0$. Since 
$$
\tau_0^{p-1}\geq r^p\Big(\fint_{B_r(x_0)}\Big(\frac{|f|}{w}\Big)^qw\;dx\Big)^\frac{1}{q}
$$ and $q>n$, we have 
\begin{align}\label{d2}
    d_2^{p-1}\leq C\eta^{(j+1)p(1-\frac{n}{q})}\tau_0^{p-1}&\leq C\eta^{(j+1)p(1-\frac{n}{q})}\Big(\frac{\tau_0}{\tau_j}\Big)^{p-1}\tau_j^{p-1}\nonumber\\
    &\leq C\eta^{(j+1)p(1-\frac{n}{q})-\sigma j(p-1)}\tau_j^{p-1}\nonumber\\
    &\leq C\eta^{j\{p(1-\frac{n}{q})-\sigma(p-1)\}+p(1-\frac{n}{q})}\tau_j^{p-1}\nonumber\\
    &\leq C\eta^{p(1-\frac{n}{q})}\tau_j^{p-1}
\end{align}
for every $\sigma\in(0,\frac{p}{2(p-1)}(1-\frac{n}{q}))$. We further choose $\eta>0$ such that 
\begin{align}\label{coneta2}
    C\eta^{p(1-\frac{n}{q})}\leq \Big(\frac{\zeta}{8}\Big)^{p-1}
\end{align}
and hence (\ref{d2}) yields $d_2\leq \Big(\frac{\zeta}{8}\Big)\tau_j$. Consequently, $d\leq d_1+d_2\leq \Big(\frac{\zeta}{8}\Big)\tau_j+\Big(\frac{\zeta}{8}\Big)\tau_j\leq \Big(\frac{\zeta}{4}\Big)\tau_j$. Using this {upper} bound of $d$, (\ref{inf}) leads to the conclusion that 
\begin{align}\label{inf1}
    \inf_{B_{j+1}}u_j\geq \frac{\zeta\tau_j}{4}.
\end{align}
If (a) holds then $m_{j+1}-m_j\geq \frac{\zeta\tau_j}{4}$. Consequently, $M_{j+1}-m_{j+1}\leq M_j-m_j+m_j-m_{j+1}\leq \tau_j-\frac{\zeta\tau_j}{4}=(1-\frac{\zeta}{4})\tau_j$.
Similarly, if (b) holds then $M_{j+1}-m_j\leq \tau_j-\frac{\zeta\tau_j}{4}$ and hence $M_{j+1}-m_{j+1}\leq (1-\frac{\zeta}{4})\tau_j$. {Now, we fix $\eta\in (0,\frac{1}{4})$ such that (\ref{coneta1}) and (\ref{coneta2}) are satisfied. Since $\eta^t\to 1$ as $t\to 0$, we choose $\sigma\in (0,\min\{1,\frac{p}{2(p-1)}(1-\frac{n}{q})\})$ such that $(1-\frac{\zeta}{4})\leq \eta^\sigma$.} Thus, we get
\begin{align*}
    \osc_{B_{j+1}}\,{u}\leq \eta^\sigma\tau_j=\tau_{j+1}.
\end{align*}
Hence, by induction, \eqref{oscdecay} is satisfied.
\end{proof}

\subsection{Harnack Inequalities}
\subsubsection{Preliminary version of weak Harnack inequality}
\begin{Lemma}\label{WHI0}
    Suppose $\mathcal{F}$ satisfies \eqref{F} where $\frac{f}{w},\; \frac{g}{w} \in L^q_{\mathrm{loc}}(\Omega,w)$ with $q>n$ and $g$ satisfies \eqref{conditiong}. Let $u$ be a weak supersolution of \eqref{ME} such that $u\geq 0$ in $B_R(x_0)\Subset\Omega$. Then there exist constants $p_0=p_0(n,p,s,q,\kappa,\Lambda,\alpha,C_1,C_2,[w]_p,M)>0$ and $C=C(n,p,s,q,\kappa,\Lambda,\alpha,C_1,C_2,[w]_p,\\M)>0$, such that
\begin{align*}
\Big(\fint_{B_r(x_0)}u^{p_0}w\;dx\Big)^\frac{1}{p_0}\leq C\Big\{\inf_{B_r(x_0)}u+\Big(\frac{r}{R}\Big)^{p'}\Tail(u_-;x_0,R)+r^{p'}\Big(\fint_{B_{8r}(x_0)}\Big(\frac{|f|}{w}\Big)^qw\;dx\Big)^\frac{1}{q(p-1)}\Big\},
\end{align*}
for every $r\in(0,1]$ with $r<\frac{R}{16}$.
\end{Lemma}
\begin{proof}
We prove the result by using Cavalieri’s principle, {which ensures that 
\begin{align}\label{CP}
    \fint_{B_r(x_0)}u^{\Theta}w\;dx&=\Theta\int_0^\infty t^{\Theta-1}\frac{w(\{x\in B_r(x_0):u(x)>t\})}{w(B_r(x_0))}\;dt,
\end{align}
for any $\Theta>0$.}
To this end, we mainly estimate the quantity
$$
\frac{w(\{x\in B_r(x_0):u(x)\geq t\})}{w(B_r(x_0))}
$$
for every $t>0$, for which we shall use the Krylov--Safonov covering lemma stated in Lemma \ref{covering}.\\
For $t>0$ and $i\in\mathbb{N}\cup\{0\}$, we define the following set:
    $$A^i_t:=\Big\{x\in B_r(x_0): u(x)>t\zeta^i-\frac{d}{1-\zeta}\Big\},$$
    where $\zeta\in (0,\frac{1}{4})$ is the constant obtained in Lemma \ref{expansionofpositivity} with $\tau:=\frac{\epsilon}{[w]_p 6^{np}}$, where $\epsilon:=\frac{1}{2c_w}$ with $c_w$ being the doubling constant given in Lemma \ref{doubling}. Here,
$$
d:=\hat{c}\Big\{\Big(\frac{r}{R}\Big)^{p'}\Tail(u_-;x_0,R)+r^{p'}\Big(\fint_{B_{{8}r}(x_0)}\Big(\frac{|f|}{w}\Big)^qw\;dx\Big)^\frac{1}{q(p-1)}\Big\},
$$
where $\hat{c}$ will be determined below in \eqref{c}. Now, we will use Lemma \ref{covering} with $E=A_t^{i-1}$. It is clear that $A^{i-1}_t\subset A^{i}_t$ for every $i\in\mathbb{N}$. We claim that $$[A^{i-1}_t]_\epsilon\subset A^{i}_t,$$ {for every $i\in\mb{N}$,}
where
$$
[A^{i-1}_t]_\epsilon:=\cup_{0<\rho\leq \frac{2r}{3}}\Big\{B_r(x_0)\cap B_{3\rho}(x){, x\in B_r(x_0): w(A^{i-1}_t\cap B_{3\rho}(x))>\epsilon w(B_\rho(x))}\Big\}.
$$
{To this end, for $i\in\mb{N}$, let $y\in B_r(x_0)$ be such that $B_r(x_0)\cap B_{3\rho}(y)\subset [A_t^{i-1}]_{\epsilon},$ { for some $\rho\in(0,\frac{2r}{3}]$}. Then 
$$
w(A_t^{i-1}\cap B_{3\rho}(y))>\epsilon w(B_\rho(y)).
$$}
Using the above estimate along with Lemma \ref{comlem}, we obtain
\begin{align*}
\frac{w(A^{i-1}_t\cap B_{6\rho}(y))}{w(B_{6\rho}(y))}&\geq\frac{w(A^{i-1}_t\cap B_{3\rho}(y))}{w(B_{6\rho}(y))}\nonumber\geq \frac{\epsilon\,w(B_{\rho}(y))}{w(B_{6\rho}(y))}\nonumber\geq \frac{\epsilon}{[w]_p6^{np}}:=\tau.
\end{align*}
{Taking into account $0<r<\frac{R}{16}$,\;$0<\rho\leq \frac{2r}{3}$ and $y\in B_r(x_0)$, we have $B_{3\rho}(y)\subset B_{12\rho}(y)\subset B_{\frac{R}{2}}(y)\subset B_R(x_0)$ and $B_{6\rho}(y)\subset B_{8r}(x_0)$. Since $3\rho<\frac{R}{8}$ and $u\geq 0$ in $B_\frac{R}{2}(y)\Subset\Omega$,} applying Lemma \ref{expansionofpositivity}{ with $r=3\rho$, $R=\frac{R}{2}$, $k=t\zeta^{i-1}-\frac{d}{1-\zeta}$ and $\tau=\frac{\epsilon}{[w]_p 6^{np}}$}, we have 
\begin{align}\label{inf2}
    \inf_{B_{3\rho}(y)}\geq \zeta\Big(t\,\zeta^{i-1}-\frac{d}{1-\zeta}\Big)-d',
\end{align}
where 
$$
d'=\Big(\frac{6\rho }{R}\Big)^{p'}\Tail\Big(u_-;y,\frac{R}{2}\Big)+(3\rho)^{p'}\Big(\fint_{B_{6\rho}(y)}\Big(\frac{|f|}{w}\Big)^qw\;dx\Big)^\frac{1}{q(p-1)}.
$$
Since $B_{\frac{R}{2}}(y)\subset B_R(x_0)$, $B_{6\rho}(y)\subset B_{8r}(x_0)$ and $0<\rho\leq \frac{2r}{3}$, utilizing Lemma \ref{comlem}, we deduce that
\begin{align}\label{d'}
    d'&\leq C\Big\{ \big(\frac{r}{R}\big)^{p'}\Tail(u_-;x_0,R)+\rho^{p'}\Big(\frac{w(B_{8r}(x_0))}{w(B_{6\rho}(y))}\Big)^\frac{1}{q(p-1)}\Big(\fint_{B_{8r}(x_0)}\Big(\frac{|f|}{w}\Big)^qw\;dx\Big)^\frac{1}{q(p-1)}\Big\}\nonumber\\
    &\leq C\Big\{\big(\frac{r}{R}\big)^{p'}\Tail(u_-;x_0,R)+\rho^{p'}\Big(\frac{r}{\rho}\Big)^\frac{np}{q(p-1)}\Big(\fint_{B_{8r}(x_0)}\Big(\frac{f}{w}\Big)^qw\;dx\Big)^\frac{1}{q(p-1)}\Big\}\nonumber\\
    &\leq C\Big\{\big(\frac{r}{R}\big)^{p'}\Tail(u_-;x_0,R)+\rho^{(1-\frac{n}{q})\frac{p}{p-1}}r^\frac{np}{q(p-1)}\Big(\fint_{B_{8r}(x_0)}\Big(\frac{|f|}{w}\Big)^qw\;dx\Big)^\frac{1}{q(p-1)}\Big\}\nonumber\\
    &\leq C\Big\{\big(\frac{r}{R}\big)^\frac{p}{p-1}\Tail(u_-;x_0,R)+r^{p'}\Big(\fint_{B_{8r}(x_0)}\Big(\frac{|f|}{w}\Big)^qw\;dx\Big)^\frac{1}{q(p-1)}\Big\},
\end{align}
for some constant $C=C(n,p,q,[w]_p)>0$. In the last inequality, we used $q>n$ and $3\rho\leq 2r$. Now, we choose 
\begin{equation}\label{c}
\hat{c}=2C,
\end{equation}
where $C$ is given in \eqref{d'} above. Hence, we have $d'<d$. Thus, (\ref{inf2}) leads to the conclusion that 
\begin{align}
    \inf_{B_{3\rho}(y)}\,u>t\,\zeta^{i}-\frac{d}{1-\zeta},
\end{align}
Consequently, $[A^{i-1}_t]_\epsilon\subset A^i_t$. By Lemma \ref{covering}, we conclude that for each $i\in\mathbb{N}$ and every $t>0$, either $A^i_t=B_r(x_0)$ or $w(A^i_t)\geq \frac{1}{c_w\epsilon}w(A^{i-1}_t)=2w({A}^{i-1}_t)$.

Now we claim that if 
\begin{equation}\label{m}
w(A^0_t)>2^{-m}w(B_r(x_0)),
\end{equation}
holds for some $m\in\mathbb{N}$, then $A^m_t=B_r(x_0)$ and therefore
\begin{equation}\label{m1}
u(x)>t\zeta^m-\frac{d}{1-\zeta}\text{ in }B_r(x_0).
\end{equation}
To this end, we observe that if $A^i_t=B_r(x_0)$, for some $i\in\{1,2,...,m-1\}$, then $B_r(x_0)=A^{i}_t\subset A^m_t\subset B_r(x_0)$, which in turn gives $A_t^m=B_r(x_0)$. Otherwise, if {$w(A^i_t)\geq 2w(A^{i-1}_t)$} for some $i\in\{1,2,...,m-1\}$, then using the presumed estimate \eqref{m}, we observe that 
$$
w(A^{m}_t)\geq 2w(A^{m-1}_t)\geq...\geq 2^{m}w(A^0_t)>w(B_r(x_0)),
$$
which gives $A_t^m=B_r(x_0)$. Thus, our claim is proved. Our next claim is that for every $t>0$, 
\begin{align}\label{WH1}
    \frac{w(\{x\in B_r(x_0):u(x)>t\})}{w(B_r(x_0))}\leq\frac{w({A}^0_t)}{w(B_r(x_0))}\leq \Big(\frac{1}{\zeta t}\Big(\inf_{B_r(x_0)} u+\frac{d}{1-\zeta}\Big)\Big)^\beta,
\end{align}
with $\beta=\mathrm{log}_\zeta(\frac{1}{2})$.
If $w(A^0_t)=0$, the above conclusion \eqref{WH1} follows immediately. Suppose $w(A^0_t)>0$, then we choose the smallest $m\in\mathbb{N}$ such that \eqref{m} holds, that is 
$$
2^{-m+1}\geq\frac{w(A^0_t)}{w(B_r(x_0))}>2^{-m}.
$$ Therefore, by the claim \eqref{m1}, it follows that
\begin{align*}
    \inf_{B_r(x_0)}u\geq t\zeta\Big(\frac{w(A^0_t)}{w(B_r(x_0))}\Big)^\frac{1}{\beta}-\frac{d}{1-\zeta},
\end{align*}
which proves \eqref{WH1}. We set $p_0=\frac{\beta}{2}>0$, and {for a given $\mu>0$}, we define $a=\inf_{B_r(x_0)} u+\frac{d}{1-\zeta}+{\mu}$. {Using (\ref{CP}),} we obtain
\begin{align}\label{WH2}
   \fint_{B_r(x_0)}u^{p_0}w\;dx&=p_0\int_0^\infty t^{p_0-1}\frac{w(\{x\in B_r(x_0):u(x)>t\})}{w(B_r(x_0))}\;dt\nonumber\\
   &\leq p_0\int_0^at^{p_0-1}\;dt+p_0\Big(\frac{a}{\zeta}\Big)^\beta\int_a^\infty t^{p_0-1-\beta}\;dt\nonumber\\
   &=a^{p_0}(1+\zeta^{-2p_0})\leq 2a^{p_0}\zeta^{-2p_0},
\end{align}
where we have used \eqref{WH1}. Therefore, {letting $\mu\to 0$} in \eqref{WH2} yields
\begin{align*}
    \Big(\fint_{B_r(x_0)}u^{p_0}w\;dx\Big)^\frac{1}{p_0}&\leq C\big(\inf_{B_r(x_0)} u+d\big),
\end{align*}    
for some constant $C=C(n,p,s,q,\kappa,\Lambda,\alpha,C_1,C_2,[w]_p,M)>0$. The proof is thereby completed.
\end{proof}
\subsubsection{Harnack inequality}
\begin{proof}[Proof of Theorem \ref{Harthm}]
     Since $u$ is a weak solution of \eqref{ME}, it is a weak subsolution as well as a weak supersolution of \eqref{ME}.  Applying Theorem \ref{Bdd} and Lemma \ref{Tailestimate}, there exists a constant $C=C(n,p,s,q,\kappa,{\lambda},\Lambda,\alpha,C_1,C_2,[w]_p,M)>0$ such that for every $0<\rho<r$, we have
\begin{align*}
    \sup_{B_\frac{\rho}{2}(x_0)}u\leq &C\delta^{-\frac{p-1}{p}\frac{\kappa q'}{\kappa-q'}}\Big(\fint_{B_\rho(x_0)}u^pw\;dx\Big)^\frac{1}{p}+\delta \mathrm{Tail}\Big(u_+;x_0,\frac{\rho}{2}\Big)\nonumber\\
    &+\rho^{p'}\Big(\fint_{B_\rho(x_0)}\Big(\frac{|f|}{w}\Big)^qw\;dx\Big)^\frac{1}{q(p-1)}\nonumber\\
    &\leq C\delta^{-\frac{p-1}{p}\frac{\kappa q'}{\kappa-q'}}\Big(\fint_{B_\rho(x_0)}u^pw\;dx\Big)^\frac{1}{p}
    +\rho^{{p'}}\big(\fint_{B_\rho(x_0)}\Big(\frac{|f|}{w}\Big)^qw\;dx\big)^\frac{1}{q(p-1)}\nonumber\\
    &+C\delta \Big(\sup_{B_\frac{\rho}{2}(x_0)}u+\Big(\frac{\rho}{R}\Big)^{p'}\Tail(u_-;x_0,R)+\rho^{p'}\Big(\fint_{B_\rho(x_0)}\Big(\frac{|f|}{w}\Big)^qw\;dx\Big)^\frac{1}{q(p-1)}\Big),
\end{align*}
for every $\delta\in(0,1]$. We choose $\frac{1}{2}\leq \sigma_1<\sigma_2\leq 1$ and $\rho=(\sigma_2-\sigma_1)r$. For $y\in B_{\sigma_1 r}(x_0)$, taking Lemma \ref{comlem} into account, we have 
\begin{align}\label{cov1}
    \sup_{B_\frac{\rho}{2}(y)}u&\leq C\delta^{-\frac{p-1}{p}\frac{\kappa q'}{\kappa-q'}}\Big(\fint_{B_\rho(y)}u^pw\;dx\Big)^\frac{1}{p}
    +(C\delta+1)\rho^{p'}\Big(\fint_{B_\rho(y)}\Big(\frac{|f|}{w}\Big)^qw\;dx\Big)^\frac{1}{q(p-1)}\nonumber\\
    &+C\delta \Big(\sup_{B_\frac{\rho}{2}(y)}u+\Big(\frac{\rho}{R}\Big)^{p'}\Tail(u_-;x_0,R)\Big)\nonumber\\
    &\leq C\delta^{-\frac{p-1}{p}\frac{\kappa q'}{\kappa-q'}}\frac{1}{(\sigma_2-\sigma_1)^n}\Big(\fint_{B_{\sigma_2r}(x_0)}u^pw\;dx\Big)^\frac{1}{p}\nonumber\\
    &+(C\delta+1)\frac{\rho^{{p'}}}{(\sigma_2-\sigma_1)^\frac{n{p'}}{q}}\Big(\fint_{B_r(x_0)}\Big(\frac{|f|}{w}\Big)^qw\;dx\Big)^\frac{1}{q(p-1)}\nonumber\\
    &+C\delta \Big(\sup_{B_\frac{\rho}{2}(y)}u+\Big(\frac{\rho}{R}\Big)^{p'}\Tail(u_-;x_0,R)\Big),
\end{align}
where we have also used the fact that $u\geq 0$ in $B_R(x_0)$. By a covering argument, from \eqref{cov1}, we deduce
\begin{align}\label{coveringargment}
    \sup_{B_{\sigma_1r}(x_0)}u&\leq C\delta^{-\frac{p-1}{p}\frac{\kappa q'}{\kappa-q'}}\frac{1}{(\sigma_2-\sigma_1)^n}\Big(\fint_{B_{\sigma_2r}(x_0)}u^pw\;dx\Big)^\frac{1}{p}\nonumber\\
    &+(C\delta+1)r^{p'}(\sigma_2-\sigma_1)^{(1-\frac{n}{q})p'}\big(\fint_{B_r(x_0)}\Big(\frac{|f|}{w}\Big)^qw\;dx\big)^\frac{1}{q(p-1)}\nonumber\\
    &+C\delta \Big(\sup_{B_{\sigma_2r}(x_0)}u+\Big(\frac{r}{R}\Big)^{p'}\Tail(u_-;x_0,R)\Big)\nonumber\\
    &\leq C\delta^{-\frac{p-1}{p}\frac{\kappa q'}{\kappa-q'}}\frac{1}{(\sigma_2-\sigma_1)^n}\Big(\fint_{B_{\sigma_2r}(x_0)}u^pw\;dx\Big)^\frac{1}{p}\nonumber\\&+(C\delta+1)r^{p'}\big(\fint_{B_r(x_0)}\Big(\frac{|f|}{w}\Big)^qw\;dx\big)^\frac{1}{q(p-1)}\nonumber\\
    &+C\delta \Big(\sup_{B_{\sigma_2r}(x_0)}u+\Big(\frac{r}{R}\Big)^{p'}\Tail(u_-;x_0,R)\Big),
\end{align}
where $q>n$ has been used. For $0<t<p$, choosing $\delta=\frac{1}{4C}$ and using Young's inequality in \eqref{coveringargment} yields
\begin{align}\label{Har1}
    \sup_{B_{\sigma_1r}(x_0)}u&\leq \frac{1}{2}\sup_{B_{\sigma_2 r}(x_0)} u+\frac{C}{(\sigma_2-\sigma_1)^\frac{np}{t}}\Big(\fint_{B_{\sigma_2r}(x_0)}u^tw\;dx\Big)^\frac{1}{t}\nonumber\\
    &+Cr^{p'}\Big(\fint_{B_r(x_0)}\Big(\frac{|f|}{w}\Big)^qw\;dx\Big)^\frac{1}{q(p-1)}+C\Big(\frac{r}{R}\Big)^{p'}\Tail(u_-;x_0,R),
\end{align}
for some constant $C=C(n,p,s,q,\kappa,{\lambda},\Lambda,\alpha,C_1,C_2,[w]_p,M,t)>0$. Moreover, using H\"{o}lder's inequality, one can easily obtain \eqref{Har1} from \eqref{coveringargment} for any $t\geq p$. Thus, (\ref{Har1}) is valid for every $t>0$.
Using Lemma \ref{iterativelem}, we obtain
\begin{align}\label{MC2}
     \sup_{B_{\frac{r}{2}}(x_0)}u&\leq C\Big\{\Big(\fint_{B_{r}(x_0)}u^tw\;dx\Big)^\frac{1}{t}+\Big(\frac{r}{R}\Big)^{p'}\Tail(u_-;x_0,R)+r^{p'}\Big(\fint_{B_r(x_0)}\Big(\frac{|f|}{w}\Big)^qw\;dx\Big)^\frac{1}{q(p-1)}\Big\},
\end{align}
for some constant $C=C(n,p,s,q,\kappa,{\lambda},\Lambda,\alpha,C_1,C_2,[w]_p,M,t)>0$. In particular, choosing $t=p_0$ in \eqref{MC2}, and applying Lemma \ref{WHI0}, we obtain 
\begin{align*}
    \sup_{B_{\frac{r}{2}}(x_0)}u\leq C\Big\{\inf_{B_r(x_0)}u+\Big(\frac{r}{R}\Big)^{p'}\Tail(u_-;x_0,R)+r^{p'}\Big(\fint_{B_{8r}(x_0)}\Big(\frac{|f|}{w}\Big)^qw\;dx\Big)^\frac{1}{q(p-1)}\Big\},
\end{align*}
for some constant $C=C(n,p,s,q,\kappa,{\lambda},\Lambda,\alpha,C_1,C_2,[w]_p,M)>0$.
\end{proof}
\subsubsection{Weak Harnack inequality}
\begin{proof}[Proof of Theorem \ref{wkHarnack}]
    We begin with the observation that for $0<l\leq p_0$, H\"{o}lder's inequality together with Lemma \ref{WHI0} yields the estimate (\ref{FWH}). Thus, it suffices to prove (\ref{FWH}) for $p_0<l<\kappa(p-1)$. In this regard, our objective is to obtain a reverse H\"{o}lder estimate by using Moser's iteration technique. To this end, we assume $\frac{1}{2}\leq \tau_1\leq \tau_2\leq \frac{3}{4}$, $\Tilde{\tau}=\frac{\tau_1+\tau_2}{2}$ and $\varphi\in C_c^\infty(B_{\Tilde{\tau}r}(x_0))$ is a nonnegative function such that $0\leq \varphi\leq 1$ in $B_{\Tilde{\tau}r}(x_0)$, $\varphi\equiv 1$ in $B_{\tau_1r}(x_0)$ and $|\nabla\varphi|\leq \frac{8}{(\tau_2-\tau_1)r}$ in $B_{\Tilde{\tau}r}(x_0)$. For $t>0$, we denote by $v=(u+t)^\frac{p-\eta}{p}$. Then for any $1<\eta<p$, using Lemma \ref{EWH}, we get
    \begin{align}\label{SB2}
       \fint_{B_{\tau_2r}(x_0)}|\nabla (\varphi v)|^pw\;dx&\leq  2^{p-1}\Big(\fint_{B_{\tau_2r}(x_0)}v^p|\nabla \varphi|^pw\;dx+\fint_{B_{\tau_2r}(x_0)}\varphi^p|\nabla v|^pw\;dx\Big)\nonumber\\
       & \leq C(I_1+I_2+I_3+I_4+I_5),
    \end{align}
for some constant $C=C(n,p,\Lambda,\alpha,C_1,C_2,[w]_p)>0$. {Here,}
$$I_1:=\frac{(p-\eta)^p}{(\eta-1)^p}\fint_{B_{\tau_2r}(x_0)}v^p|\nabla\varphi|^pw\;dx,$$
$$I_2:=\frac{(p-\eta)^p}{(\eta-1)^p}\fint_{B_{\tau_2r}(x_0)}\int_{B_{\tau_2r}(x_0)} \frac{\max\{v(x),v(y)\}^p|\varphi(x)- \varphi(y)|^p}{|x-y|^{sp}}K(x,y)\;dx\;dy,$$
\begin{align*}
    I_3:=&\frac{(p-\eta)^p}{(\eta-1)}\Big(\sup_{x\in\supp\,\varphi}\int_{\R^n\setminus B_{\tau_2r}(x_0)}\frac{1}{|x-y|^{sp}}\frac{w(y)}{w(B_{x,y})}\;dy\nonumber+t^{1-p}R^{-p}\Tail(u_-;x_0,R)\Big)\\
    &\qquad\times\fint_{B_{\tau_2r}(x_0)}\varphi^pv^pw\;dx,
\end{align*}
$$I_4:=\fint_{B_{{\tau_2r}}(x_0)}|g|u^{p-1}\varphi^p(u+t)^{1-\eta}\;dx,\quad\mbox{ and }\quad I_5:=\fint_{B_{\tau_2r}(x_0)}|f|\varphi^p(u+t)^{1-\eta}\;dx.$$
Now we estimate each of the terms $I_i$. To this end, we suppose $1+\epsilon\leq\eta<p$. \\
\textbf{Estimate of $I_1$:}
Using the properties of $\varphi$, we have
\begin{align}\label{SB3}
    I_1\leq \frac{C}{(\tau_2-\tau_1)^pr^p}\fint_{B_{\tau_2r}(x_0)}v^pw\;dx,
\end{align}
for some constant $C=C(p,\epsilon)>0$.\\
\textbf{Estimate of $I_2$:}
Due to property \eqref{kernal} and Lemma \ref{intlem} along with $r\in(0,1]$, we obtain
\begin{align}\label{SB11}
    I_2\leq&C\fint_{B_{\tau_2r}(x_0)}\int_{B_{\tau_2r}(x_0)} \frac{v^p(y)~\|\nabla\varphi\|_{L^\infty(B_{r})}^p|x-y|^p}{|x-y|^{sp}}\frac{w(x)w(y)}{w(B_{x,y})}\;dx\;dy\nonumber\\
    &\leq \frac{C}{(\tau_2-\tau_1)^p r^{sp}}\fint_{B_{\tau_2r}(x_0)}v^p(y)w(y)~\;dy\leq \frac{C}{(\tau_2-\tau_1)^p r^{p}}\fint_{B_{\tau_2r}(x_0)}v^p(y)w(y)~\;dy,
\end{align}
for some $C=C(n,p,s,\Lambda,\epsilon,[w]_p)>0$. \\
\textbf{Estimate of $I_3$:}
In order to estimate $I_3$, we observe that for $x\in B_{\Tilde{\tau} r}(x_0),\;y\in\mathbb{R}^n\setminus B_{\tau_2 r}(x_0)$,
\begin{equation}\label{nn1}
\frac{|y-x_0|}{|y-x|}\leq 1+\frac{|x-x_0|}{|y-x|}\leq 1+\frac{\tau_2+\tau_1}{\tau_2-\tau_1}\leq \frac{2}{(\tau_2-\tau_1)}.
\end{equation}
Moreover, we observe that $B_{x,y}\cup B_{y,x_0}\subset B_{2|y-x_0|}(x_0)$. Thanks to Lemma \ref{comlem}, we have
\begin{align}\label{nn2}
    \frac{w(B_{y,x_0})}{w(B_{x,y})}\leq \frac{w(B_{y,x_0})}{w(B_{2|y-x_0|}(x_0))}\frac{{w(B_{2|y-x_0|}(x_0))}}{w(B_{x,y})}&\leq C\Big(\frac{|B_{y,x_0}|}{|B_{2|y-x_0|}(x_0)|}\Big)^\sigma\Big(\frac{{w(B_{2|y-x_0|}(x_0))}}{w(B_{x,y})}\Big)^p\nonumber\\
    &\leq C\Big(\frac{|y-x_0|}{|x-y|}\Big)^{np}\leq \frac{C}{(\tau_2-\tau_1)^{np}},
\end{align}
for some $C=C(n,p,[w]_p)>0$. Taking these two estimates \eqref{nn1} and \eqref{nn2} into account, we obtain
\begin{align}\label{SB4}
    I_3&\leq C\Big(\sup_{x\in B_{\Tilde{\tau}r}(x_0)}\int_{\R^n\setminus B_{\tau_2r}(x_0)}\frac{1}{|x-y|^{sp}}\frac{w(y)}{w(B_{x,y})}\;dy\nonumber+t^{1-p}R^{-p}\Tail^{p-1}(u_-;x_0,R)\Big)\\
    &\qquad\times\fint_{B_{\tau_2r}(x_0)}\varphi^pv^pw\;dx\nonumber\\
    &\leq \frac{C}{(\tau_2-\tau_1)^{sp+np}}\Big(\int_{\R^n\setminus B_{\tau_2r}(x_0)}\frac{1}{|x_0-y|^{sp}}\frac{w(y)}{w(B_{x_0,y})}\;dy\nonumber+t^{1-p}R^{-p}\Tail^{p-1}(u_-;x_0,R)\Big)\\
    &\qquad\times\fint_{B_{\tau_2r}(x_0)}\varphi^pv^pw\;dx\nonumber\\
    &\leq \frac{C}{\tau_2^{sp}(\tau_2-\tau_1)^{np+sp}r^{p}}\Big(1+t^{1-p}(\frac{r}{R})^p\Tail^{p-1}(u_-;x_0,R)\Big)\fint_{B_{\tau_2r}(x_0)}\varphi^pv^pw\;dx,
\end{align}
for some constant $C=C(n,p,s,\epsilon,[w]_p)>0$. For $\xi>0$, we choose 
\begin{align}\label{SB8}
    t\geq\Big(\frac{r}{R}\Big)^{{p'}}\Tail(u;x_0,R)+\xi>0
\end{align}
to ensure
\begin{align}\label{SB12}
    I_3\leq \frac{C}{r^p(\tau_2-\tau_1)^{np+sp}}\fint_{B_{\tau_2r}(x_0)}v^pw\;dx,
\end{align}
for some constant $C=C(n,p,s,\epsilon,[w]_p)>0$.\\
\textbf{Estimate of $I_4$:}
Applying H\"older's inequality and Interpolation Inequality, using \eqref{conditiong} we deduce
    \begin{align*}
        r^pI_4&\leq r^p\fint_{B_{{\tau_2r}}(x_0)}|g|u^{p-1}\varphi^p(u+t)^{1-\eta}\;dx\leq r^p\fint_{B_{{\tau_2r}}(x_0)}|g|\varphi^pv^pdx\nonumber\\
        &\leq r^p\Big(\fint_{B_{{\tau_2r}}(x_0)}\Big(\frac{|g|}{w}\Big)^q~w\;dx\Big)^\frac{1}{q}\Big(\fint_{B_{{\tau_2r}}(x_0)}|\varphi v|^{pq'}w\;dx\Big)^\frac{p}{pq'}\nonumber\\
        &\leq M\Big(\fint_{B_{\tau_2r}(x_0)}|\varphi v|^{p}~w\;dx\Big)^{(1-\theta)} \Big(\fint_{B_{\tau_2r}(x_0)}|\varphi v|^{p\kappa}~w\;dx\Big)^{\frac{\theta}{\kappa}},
    \end{align*}
    where $\theta=\frac{k'}{q}\in (0,1)$. By using Young's inequality, we get
    \begin{align}\label{SB1}
        r^pI_4\leq \gamma \Big(\fint_{B_{\tau_2r}(x_0)}|\varphi v|^{p\kappa}~w\;dx\Big)^{\frac{1}{\kappa}}+C(\gamma,n,p,q,M)\fint_{B_{\tau_2r}(x_0)}|\varphi v|^{p}~w\;dx,
    \end{align}
for every $\gamma>0$ (to be determined below in \eqref{gamma}).\\
\textbf{Estimate of $I_5$:} Assume that
\begin{align}\label{SB9}
    t>r^{p'}\Big(\fint_{B_{8r}(x_0)}\Big(\frac{|f|}{w}\Big)^qw\;dx\Big)^\frac{1}{q(p-1)}.
\end{align}
Proceeding as in (\ref{SB1}), we deduce 
\begin{align}\label{SB5}
    r^pI_5&\leq r^p\fint_{B_{\tau_2r}(x_0)}|f|\varphi^p(u+t)^{1-\eta}\;dx\leq r^pt^{1-p}\fint_{B_{\tau_2r}(x_0)}|f|\varphi^pv^p\;dx\nonumber\\
    &\leq t^{1-p}r^p\Big(\fint_{B_{{\tau_2r}}(x_0)}\Big(\frac{|f|}{w}\Big)^q~w\;dx\Big)^\frac{1}{q}\Big(\fint_{B_{{\tau_2r}}(x_0)}|\varphi v|^{pq'}w\;dx\Big)^\frac{p}{pq'}\leq C\Big(\int_{B_{{\tau_2r}}(x_0)}|\varphi v|^{pq'}w\;dx\Big)^\frac{p}{pq'}\nonumber\\
    &\leq \gamma \Big(\fint_{B_{\tau_2r}(x_0)}|\varphi v|^{p\kappa}~w\;dx\Big)^{\frac{1}{\kappa}}+C(\gamma,n,p,q)\fint_{B_{\tau_2r}(x_0)}|\varphi v|^{p}~w\;dx,
\end{align}
for every $\gamma>0$. Substituting \eqref{SB3}, \eqref{SB11}, \eqref{SB12}, \eqref{SB1},\eqref{SB5} into \eqref{SB2} and applying Theorem \ref{WIT} in the resulting estimate, we arrive at
\begin{align}\label{SB6}
    \Big(\fint_{B_{\tau_2r}(x_0)}|\varphi v|^{p\kappa}w\;dx\Big)^\frac{1}{\kappa}\leq \frac{{\hat{C}}}{(\tau_2-\tau_1)^{np+sp}}\fint_{B_{\tau_2r}(x_0)}v^pw~\;dx+2C\gamma \Big(\fint_{B_{\tau_2r}(x_0)}|\varphi v|^{p\kappa}w\;dx\Big)^\frac{1}{\kappa},
\end{align}
{for some constants $\hat{C}=\hat{C}(n,p,q,s,\Lambda,\alpha,C_1,C_2,\epsilon,\gamma,\Lambda,M,[w]_p)>0$ and $C=C(n,p,q,s,\Lambda,\alpha,\\C_1,C_2,\epsilon,\Lambda,M,[w]_p)>0$.} Choosing 
\begin{equation}\label{gamma}
2C\gamma=\frac{1}{2}
\end{equation}
into \eqref{SB6}, we deduce
\begin{align*}
    \Big(\fint_{B_{\tau_2r}(x_0)}|\varphi v|^{p\kappa}w\;dx\Big)^\frac{1}{\kappa}\leq \frac{C}{(\tau_2-\tau_1)^{np+sp}}\fint_{B_{\tau_2r}(x_0)}v^pw~\;dx,
\end{align*}
which yields
\begin{align*}
    \Big(\fint_{B_{\tau_1r}(x_0)}|v|^{p\kappa}w\;dx\Big)^\frac{1}{\kappa}\leq \frac{C}{(\tau_2-\tau_1)^{np+sp}}\fint_{B_{\tau_2r}(x_0)}v^pw~\;dx,
\end{align*}
for some constant $C=C(n,p,q,s,\Lambda,\alpha,C_1,C_2,\epsilon,M,[w]_p)>0$. Here we used the fact that $\varphi\equiv1$ in $B_{\tau_1r}(x_0)$. We choose 
\begin{align}\label{SB13}
    t=\Big(\frac{r}{R}\Big)^{{p'}}\Tail(u;x_0,R)+r^{p'}\Big(\int_{B_{8r}(x_0)}\Big(\frac{|f|}{w}\Big)^qw\;dx\Big)^\frac{1}{q(p-1)}+\xi
\end{align}
to ensure (\ref{SB8}) and (\ref{SB9}), where $\xi>0$. Set $\Tilde{u}=u+t$ and we arrived 
\begin{align}\label{SB7}
    \Big(\fint_{B_{\tau_1r}(x_0)}|\Tilde{u}|^{(p-\eta)\kappa}w\;dx\Big)^\frac{1}{(p-\eta)\kappa}\leq \Big(\frac{C}{(\tau_2-\tau_1)^{np+sp}}\Big)^\frac{1}{p-\eta}\Big(\fint_{B_{\tau_2r}(x_0)}\Tilde{u}^{p-\eta}w~\;dx\Big)^\frac{1}{p-\eta},
\end{align}
for some constant $C=C(n,p,q,s,\Lambda,\alpha,C_1,C_2,\epsilon,M,[w]_p)>0$, provided $1+\epsilon\leq \eta<p$.

Now we are ready to define an iteration scheme. To this end, for $0<p_0<l<\kappa(p-1)$, we assume that $m\in\mathbb{N}\cup\{0\}$ is such that $p_0\kappa^{m-1}<l\leq p_0\kappa^m$. Denote $$\rho:=\frac{l}{\kappa^m} \mbox{ and } \eta_i:=p-\kappa^i\rho\in (1,p),\; i=0,1,2,...,{m-1}.$$ It is clear that for $i=0,1,2,....,m-1$, {one has} $\eta_i\geq 1+\epsilon$ with $\epsilon=p-1-\frac{l}{k}>0$. Let $r_i=\tau_ir$ with $\tau_i=\frac{3}{4}-\frac{1}{4}\Big(\frac{1-2^{-i}}{1-2^{-m}}\Big)$.
Thus, $\tau_i-\tau_{i+1}=\frac{1}{4}\Big(\frac{1-2^{-i-1}}{1-2^{-m}}\Big)-\frac{1}{4}\Big(\frac{1-2^{-i}}{1-2^{-m}}\Big)=\frac{1}{4}\frac{2^{-i-1}}{1-2^{-m}}$. Putting $\eta=\eta_{m-1}$ in (\ref{SB7}) we get
\begin{align*}
    \Big(\fint_{B_{r_{m}}(x_0)}|\Tilde{u}|^{\kappa^m\rho}w\;dx\Big)^\frac{1}{\kappa^m\rho}&\leq \Big(C2^{m(np+sp)}\Big)^\frac{1}{\kappa^{m-1}\rho}\Big(\fint_{B_{r_{m-1}}(x_0)}\Tilde{u}^{\kappa^{m-1}\rho}w~\;dx\Big)^\frac{1}{\kappa^{m-1}\rho}\nonumber\\
    &\leq C^{\frac{1}{\rho}\sum_{i=0}^{m-1}\frac{1}{\kappa^i}}2^{(np+sp)\sum_{i=0}^{m}\frac{i}{\kappa^{i-1}}}\Big(\fint_{B_{r_0}(x_0)}\Tilde{u}^{\rho}w~\;dx\Big)^\frac{1}{\rho}\nonumber\\
    &\leq C\Big(\fint_{B_{r_0}(x_0)}\Tilde{u}^{p_0}w~\;dx\Big)^\frac{1}{p_0},
\end{align*}
for some constant $C=C(n,p,q,s,\Lambda,\alpha,C_1,C_2,l,[w]_p,M)>0$. Here, in the last inequality we used H\"{o}lder's inequality with exponent $\frac{p_0}{\rho}\geq 1$. Consequently,   
\begin{align}\label{SB10}
    \Big(\fint_{B_{\frac{r}{2}}(x_0)}|\Tilde{u}|^{l}w\;dx\Big)^\frac{1}{l}&\leq C\Big(\fint_{B_{\frac{3r}{4}}(x_0)}\Tilde{u}^{p_0}w~\;dx\Big)^\frac{1}{p_0},
\end{align}
for some constant $C=C(n,p,q,s,\Lambda,\alpha,C_1,C_2,l,[w]_p,M)>0$. Recalling the definition $\Tilde{u}$ into (\ref{SB10}), we have
\begin{align*}
    \Big(\fint_{B_{\frac{r}{2}}(x_0)}u^{l}~w\;dx\Big)^\frac{1}{l}&\leq C\Big(\fint_{B_{\frac{3r}{4}}(x_0)}u^{p_0}~w~\;dx\Big)^\frac{1}{p_0}+Ct,
\end{align*}
for some constant $C=C(n,p,q,s,\Lambda,\alpha,C_1,C_2,l,[w]_p,M)>0$. Using Lemma \ref{WHI0} and (\ref{SB13}), the above inequality yields
\begin{align*}
    \Big(\fint_{B_{\frac{r}{2}}(x_0)}u^{l}~w\;dx\Big)^\frac{1}{l}&\leq  C\Big\{\inf_{B_\frac{r}{2}(x_0)}u+\Big(\frac{r}{R}\Big)^\frac{p}{p-1}\Tail(u_-;x_0,R)\nonumber\\
    &+r^\frac{p}{p-1}\Big(\fint_{B_{8r}(x_0)}\Big(\frac{|f|}{w}\Big)^qw\;dx\Big)^\frac{1}{q(p-1)}\Big\}+C\xi,
\end{align*}
for every $\xi>0$, for some constant $C=C(n,p,q,s,\Lambda,\alpha,C_1,C_2,l,[w]_p,M)>0$. Taking $\xi\to 0$, we obtain our desired result.
\end{proof}
\subsection{Semicontinuity}
\begin{proof}[Proof of Theorem \ref{lscthm}]
The proof follows combining Proposition \ref{lowersemi} along with Lemma \ref{semlem}.
\end{proof}

\begin{proof}[Proof of Corollary \ref{upper}]
Since $u$ is weak subsolution of equation (\ref{ME}), $v=-u$ is a weak supersolution of 
$$
\mathcal{M}_\alpha\,v=\mathcal{G}(x,v)\text{ in }\Omega,
$$
where $\mathcal{G}(x,v):=-\mathcal{F}(x,-v)$ with $\mathcal{G}$ satisfying \eqref{F}. Moreover, $v$ is bounded below in $\R^n$. Hence, by using Theorem \ref{lscthm}, we conclude that $u$ is upper semicontinuous. This completes the proof.
\end{proof}

\section*{Acknowledgement} Prashanta Garain acknowledges the financial support of the Anusandhan National Research Foundation (ANRF), India under the ARG-MATRICS grant, File No. ANRF/ARGM/2025\\/001315/MTR.

\end{document}